\documentclass[11pt]{article}

\usepackage[T1]{fontenc}      % Police contenant les caractÃ¨res franÃ§ais
\usepackage[utf8]{inputenc}   % LaTeX, comprends les accents !
\usepackage{hyperref}
\usepackage{amsmath,amsthm} 
\usepackage{amssymb,mathrsfs} 
\usepackage{amssymb}
\usepackage{amsfonts}
\usepackage{upgreek}
\usepackage{nicefrac}
\usepackage[a4paper, total={7in, 8in}]{geometry}
\usepackage{authblk}
\usepackage{graphicx}
 \usepackage{color}
 \usepackage{algorithm2e}
 \usepackage{a4wide}
\usepackage{stmaryrd}
\DeclareMathAlphabet{\mathpzc}{OT1}{pzc}{m}{it}
\usepackage[inline]{enumitem}
\usepackage{url}
\usepackage{graphicx} %insertion d'images 
\usepackage{lmodern} %% improve pdf rendering
\usepackage{xcolor}
\usepackage{bbm}
\usepackage{ifthen}
\usepackage{xargs}
\usepackage{caption}
\usepackage{subcaption}
\usepackage[disable]{todonotes}

\hypersetup{
  colorlinks = true,
  citecolor = blue,
}

\usepackage{aliascnt}
\usepackage{cleveref}
\usepackage{autonum}
\makeatletter
\newtheorem{theorem}{Theorem}
\crefname{theorem}{theorem}{Theorems}
\Crefname{Theorem}{Theorem}{Theorems}

\newaliascnt{lemma}{theorem}
\newtheorem{lemma}[lemma]{Lemma}
\aliascntresetthe{lemma}
\crefname{lemma}{lemma}{lemmas}
\Crefname{Lemma}{Lemma}{Lemmas}

\newaliascnt{corollary}{theorem}
\newtheorem{corollary}[corollary]{Corollary}
\aliascntresetthe{corollary}
\crefname{corollary}{corollary}{corollaries}
\Crefname{Corollary}{Corollary}{Corollaries}

\newaliascnt{proposition}{theorem}
\newtheorem{proposition}[proposition]{Proposition}
\aliascntresetthe{proposition}
\crefname{proposition}{proposition}{propositions}
\Crefname{Proposition}{Proposition}{Propositions}

\newaliascnt{definition}{theorem}

\aliascntresetthe{definition}
\crefname{definition}{definition}{definitions}
\Crefname{Definition}{Definition}{Definitions}

\newaliascnt{definitionProposition}{theorem}

\aliascntresetthe{definitionProposition}
\crefname{Proposition and Definition}{Proposition and Definition}{Proposition and Definition}
\Crefname{Proposition and Definition}{Proposition and Definition}{Proposition and Definition}

\newaliascnt{remark}{theorem}
\newtheorem{remark}[remark]{Remark}
\aliascntresetthe{remark}
\crefname{remark}{remark}{remarks}
\Crefname{Remark}{Remark}{Remarks}

\newtheorem{example}[theorem]{Example}
\crefname{example}{example}{examples}
\Crefname{Example}{Example}{Examples}

\crefname{figure}{figure}{figures}
\Crefname{Figure}{Figure}{Figures}

\newtheorem{assumption}{\textbf{H}\hspace{-3pt}}
\Crefname{assumption}{\textbf{H}\hspace{-3pt}}{\textbf{H}\hspace{-3pt}}
\crefname{assumption}{\textbf{H}}{\textbf{H}}

\newtheorem{assumptionAO}{\textbf{AO}\hspace{-3pt}}
\Crefname{assumptionAO}{\textbf{AO}\hspace{-3pt}}{\textbf{AO}\hspace{-3pt}}
\crefname{assumptionAO}{\textbf{AO}}{\textbf{AO}}

\Crefname{assumptionL}{\textbf{L}\hspace{-3pt}}{\textbf{L}\hspace{-3pt}}
\crefname{assumptionL}{\textbf{L}}{\textbf{L}}

\newtheorem{assumptionA}{\textbf{A}\hspace{-3pt}}
\Crefname{assumptionA}{\textbf{A}\hspace{-3pt}}{\textbf{A}\hspace{-3pt}}
\crefname{assumptionA}{\textbf{A}}{\textbf{A}}

\Crefname{assumptionG}{\textbf{G}\hspace{-3pt}}{\textbf{G}\hspace{-3pt}}
\crefname{assumptionG}{\textbf{G}}{\textbf{G}}

\def\msa{\mathsf{A}}

\def\msk{\mathsf{K}}

\def\msc{\mathsf{C}}

\def\msf{\mathsf{F}}

\def\msu{\mathsf{U}}

\def\mcbb{\mathcal{B}}  %%% \mcb est déjà pris
\newcommand{\mcb}[1]{\mathcal{B}(#1)}

\def\mcf{\mathcal{F}}
\def\mcg{\mathcal{G}}

\def\rset{\mathbb{R}}

\def\nset{\mathbb{N}}
\def\nsets{\mathbb{N}^*}

\def\mrl{\mathrm{L}}

\def\rmd{\mathrm{d}}

\def\rme{\mathrm{e}}

\def\rmC{\mathrm{C}}

\def\rmF{\mathrm{F}}
\def\rmE{\mathrm{E}}

\newcommand{\cco}{\llbracket}
\newcommand{\ccf}{\rrbracket}
\newcommand{\po}{\left(}
\newcommand{\pf}{\right)}

\newcommand{\R}{\mathbb R}

\newcommand{\dd}{\mathrm{d}}

\newcommand{\na}{\nabla}

\newcommandx{\functionspace}[2][1=+]{\mathbb{F}_{#1}(#2)}
\newcommandx{\VarDeux}[3][3=]{\operatorname{Var}^{#3}_{#1}\left\{#2 \right\}}

\newcommand{\1}{\mathbbm{1}}

\newcommand{\LeftEqNo}{\let\veqno\@@leqno}

\newcommand{\floor}[1]{\left\lfloor #1 \right\rfloor}
\newcommand{\ceil}[1]{\left\lceil #1 \right\rceil}

\newcommand{\floorLigne}[1]{\lfloor #1 \rfloor}

\newcommand{\N}{\ensuremath{\mathbb{N}}}

\newcommand{\PE}{\mathbb{E}}
\newcommand{\PP}{\mathbb{P}}

\newcommand{\abs}[1]{\left\vert #1 \right\vert}
\newcommand{\absLigne}[1]{\vert #1 \vert}
\newcommand{\tvnorm}[1]{\| #1 \|_{\mathrm{TV}}}

\newcommandx{\Vnorm}[2][1=V]{\| #2 \|_{#1}}
\newcommandx{\VnormEq}[2][1=V]{\left\| #2 \right\|_{#1}}
\newcommandx{\norm}[2][1=]{\ifthenelse{\equal{#1}{}}{\left\Vert #2 \right\Vert}{\left\Vert #2 \right\Vert^{#1}}}
\newcommandx{\normLigne}[2][1=]{\ifthenelse{\equal{#1}{}}{\Vert #2 \Vert}{\Vert #2\Vert^{#1}}}

\newcommand{\parenthese}[1]{\left(#1 \right)}
\newcommand{\parentheseLigne}[1]{(#1 )}
\newcommand{\parentheseDeux}[1]{\left[ #1 \right]}
\newcommand{\defEns}[1]{\left\lbrace #1 \right\rbrace }
\newcommand{\defEnsLigne}[1]{\lbrace #1 \rbrace }

\newcommand{\ps}[2]{\left\langle#1,#2 \right\rangle}
\newcommand{\psLigne}[2]{\langle#1,#2 \rangle}

\newcommand{\proba}[1]{\mathbb{P}\left( #1 \right)}

\newcommand{\probaLigne}[1]{\mathbb{P}( #1 )}
\newcommandx\probaMarkovTilde[2][2=]
{\ifthenelse{\equal{#2}{}}{{\widetilde{\mathbb{P}}_{#1}}}{\widetilde{\mathbb{P}}_{#1}\left[ #2\right]}}
\newcommand{\probaMarkov}[2]{\mathbb{P}_{#1}\left( #2\right)}
\newcommand{\expe}[1]{\PE \left[ #1 \right]}
\newcommand{\expeW}[2]{\PE^{#1} \left[ #2 \right]}

\newcommand{\expeLigne}[1]{\PE [ #1 ]}

\newcommand{\expeMarkov}[2]{\PE_{#1} \left[ #2 \right]}

\newcommand{\plusinfty}{+\infty}
\newcommand\limn{\lim_{n \to \infty}}
\newcommand{\limsupn}{\limsup_{n\to \plusinfty}}
\newcommand{\liminfn}{\liminf_{n\to \plusinfty}}

\def\ie{\textit{i.e.}}

\def\eqsp{\;}
\newcommand{\coint}[1]{\left[#1\right)}
\newcommand{\ocint}[1]{\left(#1\right]}
\newcommand{\ooint}[1]{\left(#1\right)}
\newcommand{\ccint}[1]{\left[#1\right]}
\newcommand{\cointLigne}[1]{[#1)}

\newcommand{\ccintLigne}[1]{[#1]}

\newcommandx{\weight}[2][2=n]{\omega_{#1,#2}^N}

 \newcommand{\tcr}[1]{{#1}}
  \newcommand{\tcwc}[1]{{#1}}
  \newcommand{\tcrw}[1]{{#1}}

\newcommandx\sequence[3][2=,3=]
{\ifthenelse{\equal{#3}{}}{\ensuremath{\{ #1_{#2}\}}}{\ensuremath{\{ #1_{#2}, \eqsp #2 \in #3 \}}}}
\newcommandx\sequenceD[3][2=,3=]
{\ifthenelse{\equal{#3}{}}{\ensuremath{\{ #1_{#2}\}}}{\ensuremath{( #1)_{ #2 \in #3} }}}

\newcommandx{\sequencen}[2][2=n\in\N]{\ensuremath{\{ #1_n, \eqsp #2 \}}}
\newcommandx\sequenceDouble[4][3=,4=]
{\ifthenelse{\equal{#3}{}}{\ensuremath{\{ (#1_{#3},#2_{#3}) \}}}{\ensuremath{\{  (#1_{#3},#2_{#3}), \eqsp #3 \in #4 \}}}}
\newcommandx{\sequencenDouble}[3][3=n\in\N]{\ensuremath{\{ (#1_{n},#2_{n}), \eqsp #3 \}}}

\def\iid{\text{i.i.d.}}

\def\eg{e.g.}

\newcommand{\opnorm}[1]{{\left\vert\kern-0.25ex\left\vert\kern-0.25ex\left\vert #1 
    \right\vert\kern-0.25ex\right\vert\kern-0.25ex\right\vert}}

\def\generator{\mathcal{A}}

\def\bfe{\mathbf{e}}

\def\Id{\operatorname{Id}}

\newcommandx{\CPE}[3][1=]{{\mathbb E}_{#1}\left[\left. #2 \, \middle \vert \, #3 \right. \right]} %%%% esperance conditionnelle

\newcommandx{\CPELigne}[3][1=]{{\mathbb E}_{#1}[\left. #2 \,  \vert \, #3 \right. ]} %%%% esperance conditionnelle

\newcommandx{\CPVar}[3][1=]{\mathrm{Var}^{#3}_{#1}\left\{ #2 \right\}}
\newcommand{\CPP}[3][]
{\ifthenelse{\equal{#1}{}}{{\mathbb P}\left(\left. #2 \, \right| #3 \right)}{{\mathbb P}_{#1}\left(\left. #2 \, \right | #3 \right)}}

\newcommandx{\osc}[2][1=]{\mathrm{osc}_{#1}(#2)}

\def\Id{\operatorname{Id}}

\def\transpose{\operatorname{T}}

\def\btau{\bar{\tau}}

\def\tW{\tilde{W}}

\def\yt{\tilde{y}}
\def\ty{\yt}

\def\tx{\tilde{x}}

\def\tX{\tilde{X}}
\def\tY{\tilde{Y}}

\def\btau{\bar{\tau}}

\def\btau{\bar{\tau}}

\newcommand\coupling[2]{\Gamma(\mu,\nu)}

\def\tpi{\tilde{\pi}}

\newcommand{\comp}{\mathrm{c}}

\renewcommand{\geq}{\geqslant}
\renewcommand{\leq}{\leqslant}

\def\Leb{\mathrm{Leb}}

\def\Phibf{\mathbf{\Phi}}

\def\bgamma{\bar{\gamma}}
\def\bdelta{\bar{\delta}}
\def\tZ{\tilde{Z}}

\def\varphibf{\boldsymbol{\varphi}}
\def\InftyBound{c_\infty}

\def\bfs{\mathbf{s}}
\def\sbf{\bfs}

\def\bfm{\mathbf{m}}

\def\Ltt{\mathtt{L}}
\def\Ctt{\mathtt{C}}
\def\mtt{\mathtt{m}}

\def\tR{\tilde{R}}

\def\bpg{\overline{p}_{\sigma^2\gamma}}
\def\tw{\tilde{w}}
\def\rmT{\mathrm{T}}
\def\trmT{\tilde{\rmT}}

\def\deltaInf{\bdelta}

\newcommand{\half}{{\nicefrac{1}{2}}}
\def\tpi{\tilde{\pi}}

\newcommandx{\wasserstein}[3][1=\distance,3=]{\mathscr{W}_{#1}^{#3}\left(#2\right)}
\newcommandx{\wassersteinLigne}[3][1=\distance,3=]{\mathscr{W}_{#1}^{#3}(#2)}
\newcommandx{\wassersteinD}[1][1=\distance]{\mathscr{W}_{#1}}
\newcommandx{\wassersteinDLigne}[1][1=\distance]{\mathscr{W}_{#1}}

\def\cbf{\mathbf{c}}
\def\tcbf{\tilde{\cbf}}

\def\lyapD{\mathpzc{V}}
\def\LyapD{\mathpzc{V}}
\def\VlyapD{\mathpzc{V}}

\def\VlyapDs{\mathpzc{V}^*}

\def\lyapDexp{\mathpzc{W}}

\def\VlyapDexp{\mathpzc{W}}

\def\LyapDsexp{\mathpzc{W}^*}
\def\VlyapDsexp{\mathpzc{W}^*}

\def\constMoment{c}
\def\Wn{W^{(n)}}
\def\Wnbf{\mathbf{W}^{(n)}}

\def\wiener{\mathbb{W}}
\def\wienersigma{\mathcal{W}}

\def\Wbf{\mathbf{W}}
\def\mubf{\boldsymbol{\mu}}
\def\Wbfn{\Wbf^{(n)}}

\newcommand{\txts}{\textstyle}

\def\Mrm{\mathrm{M}}
\def\rmM{\mathrm{M}}

\def\rmL{\mathrm{L}}
\def\Wrm{\mathrm{W}}
\def\rmW{\mathrm{W}}
\def\Nrm{\mathrm{N}}
\def\rmN{\mathrm{N}}

\newcommand{\qvar}[1]{\left\langle#1 \right\rangle}

\def\loiGauss{\mathbf{N}}
\def\ybf{\mathbf{y}}
\def\bfy{\mathbf{y}}

\title{Discrete sticky couplings of functional autoregressive processes}
\author[1]{Alain Durmus}
\author[2]{Andreas Eberle}
\author[3]{Aurélien Enfroy}
\author[4]{\\Arnaud Guillin}
\author[5]{Pierre Monmarch{\'e}}
\affil[1]{\small{Université Paris-Saclay, ENS Paris-Saclay, CNRS, Centre Borelli, F-91190 Gif-sur-Yvette, France.}}
\affil[2]{\small{Institute for Applied Mathematics - University of Bonn, Germany.}}
\affil[3]{\small{Samovar, Télécom SudParis, département CITI, TIPIC, Institut Polytechnique de Paris, Palaiseau.}}
\affil[4]{\small{Laboratoire de Math\'ematiques Blaise Pascal - Universit\'e Clermont-Auvergne, France.}}
\affil[5]{\small{LJLL - Sorbonne Universit\'e, France.}}

\begin{document}
\footnotetext[1]{Email: alain.durmus@ens-paris-saclay.fr}
\footnotetext[2]{Email: eberle@uni-bonn.de}
\footnotetext[3]{Email: aurelien.enfroy@ens-paris-saclay.fr}
\footnotetext[4]{Email: arnaud.guillin@uca.fr}
\footnotetext[5]{Email: pierre.monmarche@sorbonne-universite.fr}

\maketitle

\begin{abstract}~
  In this paper, we provide bounds in Wasserstein and total variation distances between the distributions of
   the successive iterates of two functional autoregressive processes
   with isotropic Gaussian noise of the form
   $Y_{k+1} = \rmT_\gamma(Y_k) + \sqrt{\gamma\sigma^2} Z_{k+1}$ and
   $\tilde{Y}_{k+1} = \tilde{\rmT}_\gamma(\tilde{Y}_k) +
   \sqrt{\gamma\sigma^2} \tilde{Z}_{k+1}$.
   % in the limit where the
   % parameter $\gamma \to 0$.
   More precisely, we give   non-asymptotic bounds on
   $\rho(\mathcal{L}(Y_{k}),\mathcal{L}(\tilde{Y}_k))$, where $\rho$ is
   an appropriate weighted Wasserstein distance or a $V$-distance,
   uniformly in the parameter $\gamma$, and on $\rho(\pi_{\gamma},\tilde{\pi}_{\gamma})$,
   where  $\pi_{\gamma}$ and $\tilde{\pi}_{\gamma}$  are the respective
    stationary   measures of the two processes. The class of considered processes encompasses the
   Euler-Maruyama discretization of Langevin diffusions and its
   variants.  The bounds we derive are of order $\gamma$ as $\gamma \to 0$. To obtain
   our results, we  rely on the construction of a discrete sticky Markov chain
   $(W_k^{(\gamma)})_{k \in \nset}$ which bounds  the distance between an appropriate coupling of the two processes.  We then establish stability and quantitative convergence results for this process uniformly on
   $\gamma$. In addition, we show that it converges in
   distribution to the continuous sticky process studied in
   \cite{howitt2007stochastic,eberle:zimmer:2016}. Finally, we
   apply our result to Bayesian inference of ODE parameters and  numerically illustrate them on two particular problems.
 \end{abstract}

% !TeX root = main_imsart.tex

\section{Introduction}
We are interested in this paper in  Markov chains
$(Y_k)_{k \in \nset}$ starting from $y \in \rset^d$ and defined by
recursions of the form
\begin{equation}
  \label{eq:intro_def_Y_k}
  Y_{k+1} = \rmT_{\gamma}(Y_k)  + \sigma \sqrt{\gamma} Z_{k+1} \eqsp,
\end{equation}
where $\sigma >0$, $\gamma \in \ocint{0,\bgamma}$, for some $\bgamma >0$,
$\{\rmT_{\gamma} \,: \, \gamma \in \ocint{0,\bgamma}\}$ is a family of
continuous functions from $\rset^d$ to $\rset^d$ and
$(Z_k)_{k \geq 1}$ is a sequence of \iid~$d$-dimensional standard
Gaussian random variables. Note that the Euler-Maruyama discretization
of overdamped Langevin diffusions or of general Komolgorov processes and its variants belong to this class of
processes and in that setting $\gamma$ corresponds to the
discretization step size. Indeed, the Euler scheme consists in taking
for any $\gamma \in \ocint{0,\bgamma}$,
$\rmT_{\gamma}(y) = y +\gamma b(y)$ for some   $b : \rset^d \to \rset^d$. When $b=-\nabla U$ for some potential $U$ \tcrw{and $\sigma^2=2$}, these methods are
now popular Markov Chain Monte Carlo algorithms to sample from the
target density
$ x \mapsto \rme^{-U(x)}/ \int_{\rset^{d}} \rme^{-U(y)} \rmd y$.
However, in some applications, explictly computing $\nabla U$ is not
an option and further numerical methods must be implemented which come
with additional bias since only approximations of $\nabla U$ can be
used in \eqref{eq:intro_def_Y_k}. In this paper, we precisely study
this additional source of error. In particular, based on
a chain defined by \eqref{eq:intro_def_Y_k}, we consider a second Markov chain
$(\tY_k)_{k \in \nset}$ defined by the recursion
\begin{equation}
    \label{eq:intro_def_tY_k}
  \tY_{k+1} = \tilde{\rmT}_{\gamma}(\tY_k) + \sigma \sqrt{\gamma} \tZ_{k+1} \eqsp,
\end{equation}
where $\{\tilde{\rmT}_{\gamma} \,: \, \gamma \in \ocint{0,\bgamma}\}$
is a family of functions from $\rset^d$ to $\rset^d$ such that for any
$\gamma$, $\tilde{\rmT}_{\gamma}$ is an approximation of
$\rmT_{\gamma}$ in a sense specified below, and $(\tZ_k)_{k \geq 1}$ is a sequence of
\iid~$d$-dimensional standard Gaussian random variables potentially
correlated with $(Z_k)_{k \geq 1}$. % Denote by $R_{\gamma}$ and $\tR_{\gamma}$ the Markov kernel associated with $(Y_k)_{k \in \nset}$ and $(\tY_k)_{k \in\nset}$ respectively. 

We will enforce below conditions that ensure that both 
$(Y_k)_{k\in\nset}$ and $(\tY_k)_{k \geq 1}$ are
geometrically ergodic, and denote by $\pi_\gamma$ and $\tilde \pi_\gamma$ their
 invariant probability measures respectively. If
for any $\gamma >0$, $\tilde{\rmT}_{\gamma}$ is close in some sense to
$\rmT_{\gamma}$, the overall process $(\tY_k)_{k \in\nset}$ can be
seen as a perturbed version of $(Y_k)_{k \in\nset}$, and $\tilde \pi_\gamma$
is expected to be close to $\pi_\gamma$. The main goal of this paper is 
to establish quantitative bounds on the Wasserstein and total variation distance between the finite-time laws of the two processes and
between their equilibria. The study of perturbation of Markov processes has been the subject of many existing works; see \eg~\cite{shardlow:stuart:2000,mitrophanov:2005,johndrow:et:al:2015,rudolf:schweizer:2018,medina:rudolf:schweizer:2020} and the references therein. However, it turns out that these existing results do not apply as such. 
 We pay particular attention to the dependency of these estimates
on $\gamma$. Indeed, in the case of the Euler scheme of a continuous-time diffusion, $\pi_\gamma$ and the law of $Y_{\lfloor t/\gamma\rfloor}$ for some $t>0$ converge to the invariant measure and law at time $t$ of the continuous-time process, and similarly for the perturbed chain. Hence, as $\gamma\rightarrow 0$, our estimates should not degenerate, but rather yield quantitative estimates for the continuous time process. More precisely, the present paper is the discrete-time counterpart of the study conducted by \cite{eberle:zimmer:2016} in the continuous-time case, and as $\gamma$ vanishes we recover estimates that are consistent with those of \cite{eberle:zimmer:2016}.

As in \cite{eberle:zimmer:2016}, our results are based on the
construction of a suitable coupling of the processes, i.e. a
simultaneous construction of a pair $(Y_k,\tilde Y_k)_{k\in\N}$ of
non-independent chains that marginally follow 
\eqref{eq:intro_def_Y_k} and \eqref{eq:intro_def_tY_k} respectively and are
designed to get and stay close to each  other. We use the maximal
reflection coupling for Gaussian laws, namely at each step the two chains are
coupled to merge with maximal probability and, otherwise, we use a reflection (see \Cref{subsec:coupling} below).
 Estimates on the laws of the chains then follow
from the study of $(\|Y_k-\tY_k\|)_{k\in\N}$, which is itself based on
the analysis of a Markov chain $(W_k)_{k\in\N}$ on
$\coint{0,\plusinfty}$ that is such that, by design of the coupling,
$\|Y_k-\tY_k\|_k \leqslant W_k$ for all $k\in\N$.  Thus, the question
of establishing bounds between the laws of two $d$-dimensional Markov
chains is reduced to the study of a single one-dimensional
chain. Besides, together with the Markov property, the auxiliary chain
has some nice features. At first, it is stochastically monotonous,
\ie~if $(W_k')_{k\in\N}$ is a Markov chain associated to the same
Markov kernel as $(W_k)_{k\in\N}$ and such that $W_0\leq W_0'$, then
for any $k \in \nset$,  $W_k'$ is stochastically dominated by  $W_k$,
\ie~for any $t \geq 0$, $\PP(W_k \leq t) \geq \PP(W_k'\leq t)$. Secondly,
$(W_k)_{k\in\N}$ has an atom at
$0$. %The analysis of this auxiliary Markov chain constitutes most of the article.

The main results and main steps of this study are the following. First, we prove that $(W_k)_{k\in\N}$ admits a unique invariant measure and that, independently of $\gamma$, the moments and mass on $\ooint{0,\plusinfty}$ of this equilibrium are small when the difference between $\rmT_\gamma$ and $\tilde{\rmT}_{\gamma}$ is small. Secondly, we establish the geometric convergence of the chain towards its equilibrium, at an explicit rate (stable as $\gamma \to 0$).  Finally, we prove that, as $\gamma\rightarrow 0$, the chain $(W_k)_{k\in\N}$ converges in law to the continuous-time sticky diffusion that played the same role in \cite{eberle:zimmer:2016}. This last part is not necessary to get  estimates on the finite-time and equilibrium laws of \eqref{eq:intro_def_Y_k} and \eqref{eq:intro_def_tY_k} for a given $\gamma>0$, but it sheds some new light on the limit sticky process which, in \cite{eberle:zimmer:2016}, is constructed as the limit of continuous-time diffusions with  
diffusion coefficients that vanish at zero, rather than  discrete-time chains. In some sense, $(W_k)_{k\in\N}$ can be seen as a discretization scheme for the sticky process, see also \cite{bou2020sticky} on this topic. % In particular, the sticky behaviour, namely the fact that the process spends a non-zero fraction of time at zero, while being identically equal to $0$ on no open time interval, can be related to the fact that, for the discrete time chain, the number of consecutive steps at zero is of order $1/\sqrt{\gamma}$, and such sequences are separated by excursions on $\ooint{0,\plusinfty}$ that,  with a probability [of order $1-\mathcal O(\sqrt\gamma)$ ?],  are themselves   of order $1/\sqrt{\gamma}$.

Besides the obvious continuous/discrete time difference between \cite{eberle:zimmer:2016} and the present work, let us emphasize a few other distinctions. First, in \cite{eberle:zimmer:2016}, the one-dimensional sticky process has an explicit invariant measure. This is not the case in our framework, which makes the derivation of the bounds on the moments of the equilibrium a bit more involved.  Secondly, in \cite{eberle:zimmer:2016}, although it is proven that the mass at zero and the first moment of the law of the sticky diffusion converge to their value at equilibrium (which is sufficient to get estimates on the laws of the two initial $d$-dimensional processes),  the question of long-time convergence is not addressed for the sticky diffusion, whereas our long-time convergence results for $(W_k)_{k\geqslant 0}$ together with its convergence as $\gamma\rightarrow \tcr{0}$ furnish an explicit convergence rate for  the sticky diffusion. The proof of the stability of the mass at zero and of the first moment  in \cite{eberle:zimmer:2016} relies on a concave modification of the distance (such as used e.g. in \cite{eberle:2015}), which is contracted by the chain before it hits zero. This method does not apply to, say, the second moment of the process. As a consequence, the results of \cite{eberle:zimmer:2016} only concern the total variation and $\wassersteinD[1]$ Wasserstein distances, while we consider a broader class of distances.

Finally, our theoretical results are illustrated through numerical experiments. In particular, we study the influence of the discretization scheme generally needed to perform 
  Bayesian inference for parameters of Ordinary Differential Equations (ODEs).

\section*{Outline of the work}
\tcr{The present document is organized as follows. We present the main results we obtain in \Cref{subsec:main_result_coupling}. The maximal reflection coupling is used in \Cref{subsec:coupling} to give a coupling of $Y_{k+1}$ and $\tilde{Y}_{k+1}$ whose difference is bounded by a one-dimensional Markov chain. The properties of this chain are described in \Cref{sebsec:auxiliary_Markov_chain}. The convergence of this chain to a continuous process when $\gamma\to 0$ is demonstrated in \Cref{sec:contin_limit}. The numerical illustrations are presented in \Cref{sec:application}. The postponed proofs can be read in \Cref{sec:Postponed_proofs}.}

% pitch:
% \begin{enumerate}
% \item bound for $\rho(R_\gamma^k(z,\cdot), \tR_\gamma^k(z,\cdot))$
% \item implies bounds for stationary distributions
% \item Main difficulties : $\gamma \to 0$;
% \item relies on the construction of discrete sticky
% \item study convergence/ergodicity
% \item main step is to show that it admits an atom 
% \item convergence of the family of discrete sticky
% \item Application to bayesian inference and biased force
% \item 
% \item 
% \item 
% \item 
% \item 
% \item 
% \item 
% \item 
% \item 
% \item 
% \item 
% \end{enumerate}
  \section*{Notation and convention}

We denote by $\mcbb(\rset^d)$, the Borel $\sigma$-field of $\rset^d$
endowed with the Euclidean distance and  by
$\varphibf_{\upsigma^2}$ the density of the one-dimensional Gaussian
distribution with zero-mean and variance $\upsigma^2>0$. In the case
$\upsigma =1$, we simply denote this density by $\varphibf$.  $\Phibf$ denotes the cumulative distribution of the one-dimensional Gaussian distribution with mean $0$ and variance $1$.
 $\Delta_{\rset^d}$ stands for the subset  $  \{(x,x) \in \rset^{2d} \, : \, x \in\rset^d\}$ of $\rset^d$
and for any $\msa \subset \rset^d$, $\msa^{\comp}$  for its
complement.  Let $\mu$ and $\nu$ be two $\sigma$-finite measures on
$(\rset^d,\mcbb(\rset^d))$. If $\nu$ is absolutely continuous with
respect to $\mu$, we write $\nu \ll \mu$. We say that $\nu$ and $\mu$ are equivalent if and only if $\nu \ll \mu$ and $\mu \ll \nu$. 
We denote by $\ceil{\cdot}$ and $\floor{\cdot}$ the floor and ceiling function respectively. For $d,n\in\N^*$, $\mathcal M_{d,n}(\R)$ stands for the set of $d\times n$ real matrices.  We denote by $\rmC^k(\msu,\msa)$ the set of $k$ times continuously differentiable functions from an open set $\msu \subset \rset^m$ to $\msa \subset \rset^p$. We use the convention $\sum_{k=n}^p=0$ and $\prod_{k=n}^p=1$ for $p<n$, $n,p \in \nset$, and $a/0 = \plusinfty$ for $a>0$. 
%We take the convention that $a/0 = \infty$ if $a >0$.

%%% Local Variables:
%%% mode: latex
%%% TeX-master: "main_imsart"
%%% End:

%%% Local Variables:
%%% mode: latex
%%% TeX-master: "main_imsart"
%%% End:

% !TeX root = main_imsart.tex
\section{Sticky reflection coupling}
\label{sec:main_result_coupling}

\subsection{Main result}
\label{subsec:main_result_coupling}
% This section introduces discret sticky reflection coupling. Let
% $\gamma,\sigma>0$ and $\rmT_\gamma, \tilde{\rmT}_\gamma\in
% \rmC(\rset^d,\rset^d)$. We are interested in the long-time behaviour of the
% Markov chain defined by induction on $k\in \mathbb{N}$,
% \begin{align}
% \begin{aligned}
% \underline{X}_{k+1} & =  \rmT_\gamma (\underline{X}_k) + (\sigma^2\gamma)^{\half} Z_{k+1}\\
% \underline{\tilde{X}}_{k+1} & =  \tilde{\rmT}_\gamma (\underline{\tilde{X}}_k) + (\sigma^2\gamma)^{\half} Z_{k+1} \eqsp,
% \end{aligned}
% \end{align}
% where $(Z_k)_{k\in\N}$ is $\iid$ sequences of random variables with standard Gaussian distribution.
The Markov kernels $R_{\gamma}$  associated
with $(Y_k)_{k \in \nset}$ defined in \eqref{eq:intro_def_Y_k} are given for any $y \in \rset^{d}$, $\msa \in \mathcal{B}(\rset^d)$ by
\begin{equation+}
  \label{eq:def_R_gamma}
  R_\gamma(y,\msa) = (2\uppi \sigma^2 \gamma)^{-\nicefrac{d}{2}}\int_{\rset^d} \1_{\msa}(y') \exp\defEns{-\norm[2]{y'-\rmT_{\gamma}(y)}/(2\sigma^2 \gamma)} \rmd y' \eqsp.
\end{equation+}
Note that $\tR_{\gamma}$ associated with $(\tY_k)_{k \in\nset}$ is
given by the same expression upon replacing $\rmT_{\gamma}$ by $\tilde{\rmT}_{\gamma}$.  
We consider the following assumption on the family
$\{\rmT_{\gamma} \, : \, \gamma \in\ocint{0,\bgamma}\}$. This
condition will ensure that $R_{\gamma}$ is geometrically ergodic (see
\Cref{theo:convergence_markov_chain_rsetd}) and it will be important to
derive our main results regarding the distance of $R_{\gamma}^k$ and $\tR_{\gamma}^k$, for $k \in\nset$.
\tcrw{\begin{assumption}
  \label{ass:LipsAnd}
  \begin{enumerate}[label=(\roman*)]
    Assume that $\sup_{\gamma\in \ocint{0,\bgamma}}\gamma^{-1}\normLigne{\rmT_{\gamma}(0)} < \plusinfty$ and for any
    $\gamma \in \ocint{0,\bgamma}$, there exists a non-decreasing function
    $\tau_\gamma :[0,+\infty) \to[0,+\infty)$ satisfying
\item  $  \norm{\rmT_\gamma(x)-\rmT_\gamma(\tilde{x})}\leq \tau_{\gamma}(\norm{x-\tilde{x}})$ for any
    $x, \tilde{x} \in \rset^d$ ;
    \label{ass:LipsAnd_1}
    \item $\tau_{\gamma}(0)= 0$ and there exist $R_1,\Ltt \geq 0$ and $\mtt >0$ such that for any
    $\gamma \in \ocint{0,\bgamma}$,
    \label{ass:LipsAnd_2}
    \begin{equation}
      \label{ass:LipsAnd_eq}
    \sup_{r \in \ooint{0,\plusinfty}} \{\tau_{\gamma}(r)/r \} \leq 1+ \gamma \Ltt \eqsp, \qquad \sup_{r \in \ooint{R_1,\plusinfty}}\{\tau_{\gamma}(r)/r\} \leq 1-\gamma \mtt \eqsp .
\end{equation}
  \end{enumerate}
%   There exist $R_1,\Ltt \geq 0$ and $\mtt >0$ such that for any
%   $\gamma \in \ocint{0,\bgamma}$, there exists a non-decreasing function
%   $\tau_\gamma :[0,+\infty) \to[0,+\infty)$ satisfying $\tau_{\gamma}(0)= 0$, $  \norm{\rmT_\gamma(x)-\rmT_\gamma(\tilde{x})}\leq \tau_{\gamma}(\norm{x-\tilde{x}})$ for any
%   $x, \tilde{x} \in \rset^d$, and
%   \begin{equation}
%       \label{ass:LipsAnd_eq}
%  \sup_{r \in \ooint{0,\plusinfty}} \{\tau_{\gamma}(r)/r \} \leq 1+ \gamma \Ltt \eqsp, \qquad \sup_{r \in \ooint{R_1,\plusinfty}}\{\tau_{\gamma}(r)/r\} \leq 1-\gamma \mtt \eqsp .
% \end{equation}
%  In addition, $\sup_{\gamma\in \ocint{0,\bgamma}}\normLigne{\rmT_{\gamma}(0)} < \plusinfty$. 
% % \begin{equation}
% %  \eqsp.
% % \end{equation}
\end{assumption}}
% \tcr{We use the notation \Cref{ass:LipsAnd}-\ref{ass:LipsAnd_2} to say that we only need the assumptions about $\tau_\gamma$.}
\tcr{Part of our results only deals with objects that exclusively depend on a family of functions $\{\tau_\gamma\,:\,\gamma \in \ocint{0,\bgamma}\}$ and not $\rmT_{\gamma}$, and these results only necessitate conditions on $\{\tau_\gamma\,:\,\gamma \in \ocint{0,\bgamma}\}$ specified by H1-(ii) (note that H1-(i) only concerns $\rmT_{\gamma}$).  That is why, when stating these specific results, we only assume H1-(ii).}
Note that \Cref{ass:LipsAnd} implies that $\bgamma \leq 1/\mtt$. Further, the condition that for any $\gamma \in \ocint{0,\bgamma}$,
$\tau_{\gamma}$ is non-decreasing can be omitted upon replacing in our
study $\tau_{\gamma}$ by the affine majorant
\begin{equation}
  \label{eq:def_tilde_tau}
  \bar{\tau}_{\gamma} : r\mapsto 
  \begin{cases}
    (1+\Ltt \gamma) r & \text{ if $r \in \ccint{0,R_1} $} \eqsp,\\
    (1+\Ltt \gamma) R_1  +  (1-\mtt \gamma)(r-R_1) & \text{ otherwise}   \eqsp.
  \end{cases}
\end{equation}
 Indeed, by definition and \eqref{ass:LipsAnd_eq}, for any $r \in \coint{0,\plusinfty}$,
$\tau_{\gamma}(r) \leq \bar{\tau}_{\gamma}(r)$, therefore for any
$x, \tilde{x} \in \rset^d$,
$ \norm{\rmT_\gamma(x)-\rmT_\gamma(\tilde{x})}\leq
\bar{\tau}_{\gamma}(\norm{x-\tilde{x}})$. In addition, an easy computation leads to setting $R_2= 2R_1(\Ltt +\mtt)/\mtt$,
  \begin{equation}
      \label{ass:LipsAnd_eq_v2}
 \sup_{r \in \ooint{0,\plusinfty}} \{\bar{\tau}_{\gamma}(r)/r  \} \leq 1+\gamma \Ltt \eqsp, \qquad \sup_{r \in \ooint{R_2,\plusinfty}}\{\bar{\tau}_{\gamma}(r)/r\} \leq 1-\gamma \mtt/2 \eqsp.
\end{equation}
Then, $\bar{\tau}_{\gamma}$ satisfies \Cref{ass:LipsAnd} and is non-decreasing.

Note that \Cref{ass:LipsAnd} implies that for any $r \in \coint{0,\plusinfty}$ and $\gamma \in \ocint{0,\bgamma}$, $\tau_{\gamma}(r) \leq (1+\gamma\Ltt)r$, therefore  $\rmT_{\gamma}$ is $(1+\gamma \Ltt)$-Lipschitz. The second condition in \eqref{ass:LipsAnd_eq} ensures that for any $\gamma \in \ocint{0,\bgamma}$, $\rmT_{\gamma}$ is a \textit{contraction at large distances}, \ie~for any $x,\tx \in \rset^d$, $\norm{\rmT_\gamma(x) -
 \rmT_\gamma(\tx)} \leq (1-\gamma \mtt) \norm{x-\tx}$, if $\norm{x -\tx} \geq R_1$.

The assumption
\Cref{ass:LipsAnd} holds  for the Euler scheme applied to diffusions
with scalar covariance matrices, \ie~\eqref{eq:intro_def_Y_k} with
$\rmT_\gamma(x)=x+\gamma b(x)$
% \begin{equation}
%   \label{eq:form_T_Euler_One}
% \text{$\rmT_\gamma(x)=x+\gamma b(x)$} \eqsp,  
% \end{equation}
and a drift function  $b: \rset^d \to \rset^d$, if, for some $\Ltt_b,\mtt_b,R_b>0$, $b$ is $\Ltt_b$-Lipschitz continuous and satisfies
\begin{align}
\ps{x-y}{b(x)-b(y)}\leq -\mtt_b \normLigne{x-y}^2 \eqsp ,
\end{align}
for all $x,y\in \rset^d$ with $\norm{x-y}\geqslant R_b$. Indeed, this implies that for any $x,y \in \rset^d$, \sloppy$\norm{\rmT_{\gamma}(x) -\rmT_{\gamma}(y)} \leq
(1+\Ltt_b\gamma){\norm{x-y}}$ and, provided $\gamma \in \ooint{0,\mtt_b/\Ltt_b^2}$ and $\norm{x-y}\geqslant R_b$, \sloppy
$\norm[2]{\rmT_{\gamma}(x) -\rmT_{\gamma}(y)} \leq
(1-\mtt_b\gamma){\norm[2]{x-y}}$. Therefore, it suffices to consider
$\tau_{\gamma}$ defined by \eqref{eq:def_tilde_tau} with $\Ltt = \Ltt_b$,
$\mtt =\mtt_b/2$ and $R_1=R_b$.

\tcr{Note that other discretization methods have been developed to handle drift functions $b$ that are not necessarily Lipschitz. Some examples include the split-step Euler-Maruyama discretization \cite{MATTINGLY2002185,talay2002stochastic}, tamed Euler-Maruyama scheme  \cite{hutzenthaler2015numerical,BROSSE20193638}, and Markov jump process approximations  \cite{bou2018continuous}.  However, addressing these alternative discretization schemes is beyond the scope of this paper, and we leave them for future work.}

\bigskip

Our results will be stated in term of Wasserstein distances and $V$-norms, whose definitions are the following. Consider a measurable cost function $\cbf : \rset^{2d} \to [0,\infty)$. Then the associated  Wasserstein distance  $\wassersteinD[c]$ is given for two probability measures $\mu,\nu$ on $\R^d$ by
\[\wassersteinD[c](\nu,\mu) \  = \ \inf_{\pi\in\Pi(\nu,\mu)} \int_{\R^{2d}} \cbf(x,y) \pi(\dd x,\dd y)\,, \]
where $\Pi(\nu,\mu)$ is the set of transference plans or couplings between $\nu$ and $\mu$, namely the set of probability measures on $\R^d$ whose first and second $d$-dimensional marginals are  $\nu$ and $\mu$ respectively. In the particular case where $\cbf(x,y) = \1_{\Delta_{\rset^d}^{\comp}}(x,y)$, $\wassersteinD[c]$ is simply the total variation distance $\tvnorm{\cdot}$. For $V: \rset^d \to \coint{1,\plusinfty}$, the choice $\cbf(x,y) = \1_{\Delta_{\rset^d}^{\comp}}(x,y)\{V(x)+V(y)\}$ yields the $V$-norm (see \cite[Theorem 19.1.7]{douc:moulines:priouret:soulier:2018}), i.e. $\wassersteinD[c](\nu,\mu) =\normLigne{\nu-\mu}_{V}$. Finally, for $\cbf(x,y) = \norm[p]{x-y}$ with $p \in \coint{1,\plusinfty}$, $\wassersteinD[c]$  is the $p$-th power of the usual Wasserstein distance of order $p$.
 
 \bigskip

We first show that \Cref{ass:LipsAnd} implies that the Markov kernel
$R_{\gamma}$ is $V_c$-uniformly geometrically
ergodic where for any $c >0$ and $x \in \rset^d$, $  V_c(x) = \exp(c \normLigne[2]{x}) $,
% \begin{equation}
%   \label{eq:def_V_c_exp_ergo_R_gamma}
% \eqsp,
% \end{equation}
with a convergence rate that scales linearly with the step size
$\gamma$. 
\begin{proposition}
  \label{theo:convergence_markov_chain_rsetd}
  Assume \Cref{ass:LipsAnd}. Then, for any
  $\gamma \in \ocint{0,\tcrw{\bgamma}}$, $R_{\gamma}$ admits a unique stationary distribution
  $\pi_{\gamma}$. In addition, there
  exist $c >0$, $\rho \in \coint{0,1}$ and
  $C \geq 0$ such that for any $x \in\rset^d$ and $\gamma \in\ocint{0,\tcrw{\bgamma}}$, $    \normLigne{\updelta_x R_{\gamma}^k  - \pi_{\gamma}}_{V_{c}}  \leq C \rho^{k \gamma} V_c(x)$.
  % \begin{equation}
  %   \normLigne{\updelta_x R_{\gamma}^k  - \pi_{\gamma}}_{V_{c}}  \leq C \rho^{k \gamma} V_c(x) \eqsp.
  % \end{equation}
\end{proposition}
\begin{proof}
The proof of this result follows the same strategy as \cite{debortoli2019convergence} but since we are not interested in sharp constants a more direct proof is postponed to \Cref{sec:theo:convergence_markov_chain_rsetd}.
\end{proof}
% Note that this result can be made quantitative, and other convergence
% results in total variation and Wasserstein distance of order
% $p \in \coint{1,\plusinfty}$ can also be established following the same
% lines as the proof of \cite[Corollary 14]{debortoli2019convergence}. However, these results are out of the scope of the present paper and
% would be simple adaptations of those in \cite{debortoli2019convergence} or \cite{eberle2018quantitative}.

We now consider an assumption which quantifies the
perturbation associated with $\trmT_{\gamma}$ relatively to
$\rmT_{\gamma}$, for $\gamma \in \ocint{0,\bgamma}$.

\begin{assumption}
\label{ass:bound}
There exists $\InftyBound >0$ such that 
$\sup_{x \in \rset^d}\normLigne{\rmT_\gamma(x)-\tilde{\rmT}_\gamma(x)} \leq \gamma
\InftyBound$ for all
$\gamma \in \ocint{0,\bgamma}$. 
\end{assumption}
\begin{example}
  \label{ex:euler}
The assumption \Cref{ass:bound} holds for the Euler scheme applied to diffusions with scalar covariance matrices, \ie~\eqref{eq:intro_def_Y_k} and \eqref{eq:intro_def_tY_k} with
\begin{equation}
  \label{eq:form_T_Euler}
\text{$\rmT_\gamma(x)=x+\gamma b(x)\ $ and $\ \tilde{\rmT}_\gamma(x)=x+\gamma \tilde{b}(x)$} \eqsp,  
\end{equation}
under the condition that
$\sup_{x \in \rset^d}\normLigne{b(x)-\tilde{b}(x)}\leq
\InftyBound$. This setting is exactly the one we introduced to
motivate our study. In particular, in the case where $b = -\nabla U$
for some potential $U$, $\tilde{b}$ may correspond to a numerical
approximation of this gradient. 
\end{example}

Note that compared to $\rmT_{\gamma}$, $\gamma \in \ocint{0,\bgamma}$,
we do not assume any smoothness condition on
$\tilde{\rmT}_{\gamma}$. More precisely, we do not assume that
$\tilde{\rmT}_{\gamma}$ satisfies \Cref{ass:LipsAnd}. Regarding the
ergodicity properties of $\tR_{\gamma}$ associated with $\tilde{\rmT}_{\gamma}$, 
$\gamma \in \ocint{0,\bgamma}$, we have the following result. 
\begin{proposition}
  \label{theo:ergo_tilde_R}
  Assume \Cref{ass:LipsAnd} and \Cref{ass:bound}. Then, for any
  $\gamma \in \ocint{0,\tcrw{\bgamma}}$, $\tR_{\gamma}$ admits a unique
  stationary distribution $\tpi_{\gamma}$. In addition, there exists
  $c >0$ such that for any $\gamma \in \ocint{0,\tcrw{\bgamma}}$, there
  exist $\rho_{\gamma} \in \coint{0,1}$ and $C_{\gamma} \geq 0$ such
  that for any $x \in\rset^d$
  $ \normLigne{\updelta_x \tR_{\gamma}^k - \tpi_{\gamma}}_{V_{c}} \leq
  C_{\gamma} \rho_{\gamma}^{k} V_c(x)$, where $V_c(x) = \exp(c \normLigne[2]{x})$.
\end{proposition}
\begin{proof}
  The proof is postponed to \Cref{sec:proof-theo:ergo_tilde_R}.
\end{proof}

Similarly to \Cref{theo:convergence_markov_chain_rsetd} with respect
to $R_{\gamma}$, \Cref{theo:ergo_tilde_R} implies that $\tR_{\gamma}$
is \break $V_c$-uniformly geometrically ergodic. However in contrast to
\Cref{theo:convergence_markov_chain_rsetd}, the dependency of the rate
of convergence with respect to the step size $\gamma$ is not explicit
anymore since the results and the method employed in
\cite{debortoli2019convergence} or \cite{eberle2018quantitative}
cannot be applied anymore.

Note that \Cref{theo:convergence_markov_chain_rsetd} and
\Cref{theo:ergo_tilde_R} imply that $R_{\gamma}$ and $\tR_{\gamma}$
converge to $\pi_{\gamma}$ and $\tpi_{\gamma}$ respectively in total
variation and Wasserstein metric of any order
$p \in \coint{1,\plusinfty}$. 

%\subsection{The coupling}

%We can now introduce one of the main results the main object and results of this paper.
Based on the two assumptions above, we can now state one of our main results.  Our goal is to quantify the distance between the laws of the iterates of the two chains $(Y_k)_{k\in \nset}$
and $(\tY_k)$, in particular starting from the same initial point $x \in \rset^d$ or at equilibrium. Indeed, remark that, in view of Propositions~\ref{theo:convergence_markov_chain_rsetd}  and \ref{theo:ergo_tilde_R}, letting $k\rightarrow +\infty$ in the next statement yields  quantitative
bounds on $\wassersteinD[\cbf](\pi_{\gamma}, \tpi_{\gamma})$ for any
$\gamma \in \ocint{0,\tcrw{\bgamma}}$.

\begin{theorem}
  \label{theo:main_results}
  Assume \Cref{ass:LipsAnd} and \Cref{ass:bound}  hold and let
\[(\tcbf,\lyapD) \in \{(\1_{(0,+\infty)}, \abs{\cdot}\tcrw{+1}), (\abs{\cdot}, \abs{\cdot}\tcrw{+1}), (\tcrw{\exp(\abs{\cdot})-1},\tcrw{\exp(\abs{\cdot})})\} \eqsp.\]
  % set for any $x,\tx \in \rset^d$, $t \geq 0$, $\cbf(x,\tx) = \1_{\Delta^{\comp}_{\rset^d}}(x,\tx)$ and  $V(t) = t$ or $\cbf(x,\tx) = \norm{x-\tx}$ and  $V(t) = t$ or $\cbf(x,\tx) = \1_{\Delta^{\comp}_{\rset^d}}(x,\tx) \exp(\norm{x-\tx})$ and  $V(t) = \1_{\rset_+^*}(t) \exp(t)$.
  Then, there exist some
explicit constants $C,c \geq 0$, $\rho\in \coint{0,1}$ \tcrw{and $\bgamma_1\in \ocint{0,\bgamma}$} such that for any $k \in\nset$, $\gamma \in \ocint{0,\bgamma_1}$  and  $x,\tx \in \rset^d$,
\begin{equation}
  \wassersteinD[\cbf](\delta_x R_{\gamma}^k, \delta_{\tx}\tR_{\gamma}^k) \leq C \rho^{\gamma k} \lyapD( \norm{x-\tx}) + c \InftyBound \eqsp,
\end{equation}
where $\cbf(x,\tx) = \tcbf(\norm{x-\tx})$.
\end{theorem}

\begin{remark}
It is also possible to treat the case of cost functions $\cbf $ of the form $\cbf(x,y) = \tcbf(\norm{x-y}) \po V(x)+V(y)\pf$ with $\tcbf$ as in \Cref{theo:main_results} and $V$  a positive function, simply by using H\"older's inequality. Indeed, for $p,q>1$ with $1/p+1/q=1$, we can bound
\[ \wassersteinD[\cbf](\nu,\mu) \ \leqslant  \ \po  \wassersteinD[\cbf_p](\nu,\mu)\pf^{1/p}   \po \po \nu\po V^q\pf\pf^{1/q}  + \po \mu\po V^q\pf\pf^{1/q}\pf \eqsp , \]
with $\cbf_p(x,y) = \tcbf^p(\|x-y\|)$. Bounds on the $\wassersteinD[\cbf_p]$ distance can then be established as in \Cref{theo:main_results}, while bounds on expected values of $V^q$, independent of $\gamma$, are classically obtained through Lyapunov arguments (see e.g. the proof of \Cref{theo:convergence_markov_chain_rsetd} in \Cref{sec:theo:convergence_markov_chain_rsetd}). 
\end{remark}

The rest of this section is devoted to the proof of \Cref{theo:main_results}. In particular, we define in the following the main object of this paper.

%\alain{independent}
\subsection{The discrete sticky kernel}
\label{subsec:coupling}

We define a Markovian coupling of the two chains $(Y_k)_{k \in\nset}$ and $(\tilde{Y}_k)_{k \in\nset}$ defined in \eqref{eq:intro_def_Y_k} and \eqref{eq:intro_def_tY_k} by using, at each step, the maximal reflection coupling of the two Gaussian proposals, which is optimal for the total variation distance (i.e that maximizes the probability of coalescence). \tcr{This coupling is widely recognized and possesses numerous favorable properties. For instance, it can be advantageous in the context of the Metropolis-Hastings Algorithm (see \cite{wang2021maximal}) and for remove bias in Markov Chain Monte Carlo methods (see \cite{jacob2020unbiased}).}

Let $(U_k)_{k\geq 1}$ be a sequence of $\iid$ uniform random variables on $\ccint{0,1}$ independent of $(Z_k)_{k \geq 1}$ which we recall is a sequence of \iid~$d$-dimensional standard
Gaussian random variables. We define
the discrete sticky Markov coupling $K_{\gamma}$ of
$R_{\gamma}$ and $\tR_{\gamma}$ as the Markov kernel associated with
the Markov chain on $\rset^{2d}$ given for $k\in \nset$ by 
\begin{equation}
  \label{eq:def_refl_discret}
\begin{aligned}
X_{k+1} & =  \rmT_\gamma(X_k) + (\sigma^2\gamma)^{\half} Z_{k+1}\\
\tilde{X}_{k+1} & =  X_{k+1} B_{k+1}  +(1-B_{k+1}) \rmF_\gamma (X_k,\tilde{X}_k,Z_{k+1}) \eqsp ,
\end{aligned}
\end{equation}
where  $B_{k+1}  =  \1_{[0,+\infty)}(p_\gamma(X_k,\tilde{X}_k, Z_{k+1}) - U_{k+1})$ and
\begin{equation}
%  \label{eq:39}
   \rmF_\gamma (x,\tilde{x},z) = \tilde{\rmT}_\gamma(\tilde{x}) + (\sigma^2\gamma)^{\half} \defEns{ \Id - 2\bfe(x,\tilde{x})\bfe(x,\tilde{x})^{\transpose} } z \eqsp,
\end{equation}
\begin{equation}
\label{eq:def_E_e}
  \rmE(x,\tilde{x})  =  \tilde{\rmT}_\gamma(\tx)-\rmT_\gamma(x) \eqsp, \quad  \bfe(x,\tilde{x}) =
\begin{cases}
\frac{\rmE(x,\tilde{x})}{\norm{\rmE(x,\tilde{x})}} & \text{ if } \rmE(x,\tilde{x}) \not = 0 \\
\bfe_0 & \text{otherwise} \eqsp,
\end{cases}
\end{equation}
\begin{equation}
  p_\gamma(x,\tilde{x},z)  =  \ 1\wedge \parentheseDeux{\frac{\varphibf_{\sigma^2\gamma}\defEns{ \|\rmE(x,\tilde{x})\| - (\sigma^2\gamma)^{\half}\langle \bfe(x,\tilde{x}),z\rangle } }{\varphibf_{\sigma^2\gamma}\defEns{ (\sigma^2\gamma)^{\half}\langle \bfe(x,\tilde{x}),z\rangle}}}\eqsp,
\end{equation}
 where $\bfe_0 \in \rset^d$ is an arbitrary unit-vector, \ie~$\norm{\bfe_0} = 1$,  and  $\varphibf_{\sigma^2\gamma}$ is the density of the one-dimensional Gaussian distribution with mean $0$ and variance $\sigma^2 \gamma$. 
In other words, $K_{\gamma}$  is   given for any $\gamma\in\ocint{0,\bgamma}$, $(x,y)\in \rset^{2d}$ and
$\msa \in \mathcal{B}(\rset^{2d})$ by
\begin{align}
\begin{aligned}
K_\gamma ((x,\tilde{x}),\msa)&=\int_{\rset^{d}} \1_\msa (\rmT_\gamma(x)+(\sigma^2\gamma)^{\half} z, \rmT_\gamma(x)+(\sigma^2\gamma)^{\half} z) p_\gamma(x,\tilde{x}, z)\tcr{\frac{\varphibf(\norm{z})}{(2\uppi)^{(d-1)/2}}} \rmd z\\
& +\int_{\rset^{d}} \1_\msa \parenthese{ \rmT_\gamma(x)+(\sigma^2\gamma)^{\half} z,\rmF_\gamma (x,\tilde{x},z)} (1-p_\gamma(x,\tilde{x},z))\tcr{\frac{\varphibf(\norm{z})}{(2\uppi)^{(d-1)/2}}} \rmd z\eqsp ,
\end{aligned}
\end{align}
\tcr{where $\varphibf=\varphibf_1$ is the density of the one-dimensional Gaussian distribution with mean $0$ and variance $1$.} In words, from the initial conditions $(x,\tx)$, this coupling works as follows: first, a Gaussian variable $Z_{k+1}$ is drawn for the fluctuations of $X_{k+1}$. Then, $\tX_{k+1}$ is made equal to $X_{k+1}$ with probability $p_{\tcr{\gamma}}(x,\tx,Z_{k+1})$ and, otherwise, the random variable $\tZ_{k+1}$ which determines the fluctuations of $\tX_{k+1}$ with respect to its average $ \tilde{\rmT}_\gamma(\tx)$ is given by the orthogonal reflection of $Z_{k+1}$ in the direction $ \tilde{\rmT}_\gamma(\tx)-\rmT_\gamma(x)$.
It is well known that for any $(x,\tx) \in \rset^d$,
$K_{\gamma}((x,\tx),\msa\times \rset^d) = R_{\gamma}(x,\msa)$ and
$K_{\gamma}((x,\tx), \rset^d\times \msa) = \tR_{\gamma}(\tcr{\tilde{x}},\msa)$, see \eg~\cite[Section 3.3]{bubley:dyer:jerrum:1998}, \cite[Section 4.1]{durmus:moulines:2019:supp}, \cite{eberle2018quantitative} or \cite{debortoli2019convergence}.
We coin the term \emph{sticky} for the Markov coupling  $K_{\gamma}$ after \tcrw{\cite{eberle:zimmer:2016,watanabe1971_II} and since it aims at each stage $k+1\in\nset$ to merge $X_{k}$ and $\tX_{k}$ with maximal probability under the constraint that these two processes must be Markov chains associated with $R_{\gamma}$  and $\tR_{\gamma}$ respectively, with prescribed initial conditions.  However, unlike other works dealing with couplings, which also use  the term sticky, $K_{\gamma}((x,x),\Delta_{\rset^d}) < 1$ for any $x \in\rset^d$ if $\InftyBound >0$ and therefore $\PP(X_{k+1} \neq \tX_{k+1},X_k = \tX_k|(X_k,\tX_k)) = \1_{\Delta_{\rset^d}}(X_k,\tX_k)K_{\gamma}((X_k,X_k),\Delta_{\rset^d}^{\complement}) >0$. }

The starting point of our analysis is the next result, which will enable to compare the coupling difference process $\norm{X_{k+1}-\tilde{X}_{k+1}}$ with a Markov chain on $[0,+\infty)$. Define $(G_k)_{k \geq 1}$ for any $k \geq 1$ by
\begin{equation}
  \label{eq:def_z_k}
  G_{k} = \psLigne{\bfe(X_{k-1},\tX_{k-1})}{Z_{k}} \eqsp,
\end{equation}
where $\bfe$ is given by \eqref{eq:def_E_e}. 
  For any $a \geq 0$, $g \in \rset$, $u\in\ccint{0,1}$ and $\gamma \in \ocint{0,\bgamma}$ define 
\begin{align}
\label{eq:def_H}
  \mathscr{H}_\gamma\parentheseLigne{a,g,u}= \1_{[0,+\infty)}(u-\bpg(a,g))\parenthese{a-2(\sigma^2\gamma)^{\half}g}\eqsp,
\end{align}
where
\begin{equation}
  \label{eq:def_bpg}
\bpg(a,g)=  \ 1\wedge \frac{\varphibf_{\sigma^2\gamma}\parenthese{ a- (\sigma^2\gamma)^{\half}g } }{\varphibf_{\sigma^2\gamma}\parenthese{ (\sigma^2\gamma)^{\half}g}} \eqsp .
\end{equation}
\tcrw{Note that $\mathscr{H}_\gamma$ is non-negative function as proven in \Cref{lem:H_increase}.}

\begin{proposition}
\label{lem:G_inequality}
Assume \Cref{ass:LipsAnd}-\ref{ass:LipsAnd_1} and \Cref{ass:bound}  hold. Then for any $\gamma \in \ocint{0,\bgamma}$,  $k\in \mathbb{N}$, almost surely, we have 
\begin{equation}
  \label{eq:lemma_comparison_discrete}
\normLigne{X_{k+1}-\tilde{X}_{k+1}}\leq \mathscr{G}_\gamma\parentheseLigne{\normLigne{X_k-\tilde{X}_k},G_{k+1},U_{k+1}} \eqsp ,
\end{equation} 
where $(X_k,\tilde{X}_k)_{k \in\nset}$ are defined by  \eqref{eq:def_refl_discret}, and for any $w\in \coint{0,\plusinfty}$, $g\in \rset$ and $u\in \ccint{0,1}$, 
\begin{align}
\mathscr{G}_\gamma\parentheseLigne{w,g,u}= \mathscr{H}_{\gamma}(\tau_{\gamma}(w)+\gamma \InftyBound,g,u) \eqsp.
\end{align}
In addition, for any $g \in \rset^d$ and $u\in \ccint{0,1}$, $w\mapsto \mathscr{G}_\gamma(w,g,u)$ is non-decreasing.
\end{proposition}

\begin{proof}
  The proof is postponed to \Cref{sec:proof:lem:G_inequality}.
\end{proof}

Consider now the stochastic process $(W_k)_{k\in \mathbb{N}}$ starting from $\normLigne{X_0-\tilde{X}_0}$ and defined by induction on $k$ as follows,
\begin{align}
W_{k+1} &=\mathscr{G}_\gamma (W_k,G_{k+1},U_{k+1}) \nonumber\\
& =
\begin{cases}
  \tau_{\gamma}(W_k)+\gamma \InftyBound - 2 \sigma \sqrt\gamma G_{k\tcr{+1}} & \text{ if }U_{k+1} \geqslant \bpg(\tau_{\gamma}(W_k)+\gamma \InftyBound,G_{k\tcr{+1}})\\
0 & \text{ otherwise.} 
\end{cases}\label{eq:def_r}
\end{align}
By definition \eqref{eq:def_z_k} and \eqref{eq:def_E_e}, an easy induction implies that $(G_k)_{k \geq 1}$ and $(U_k)_{k \geq 1}$
are independent, $(G_k)_{k \geq 1}$ are \iid~standard Gaussian random variables and $(U_k)_{k \geq 1}$ are \iid~uniform random variables on $\ccint{0,1}$. Therefore, $(W_k)_{k \in\nset}$ is a Markov chain with Markov kernel  $Q_\gamma$  defined for $w \in \coint{0,\plusinfty}$ and
$\msa \in \mcb{[0,+\infty)}$ by
\begin{align}
\label{def_q_gamma}
  & Q_{\gamma}(w,\msa)   \ =\  \updelta_{0}(\msa) \int_{\rset} \bpg(\tau_{\gamma}(w) + \gamma \InftyBound , g)\varphibf(g)  \rmd g  
\\
 \quad  +\int_{\rset} & \1_{\msa}\parenthese{\tau_{\gamma}(w)   +\gamma \InftyBound-2 \sigma \gamma^{1/2} g} \{1- \bpg\parenthese{\tau_{\gamma}(w) + \gamma \InftyBound , g} \} \varphibf(g)  \rmd g \eqsp, 
\end{align}
where $\varphibf$ is the density of the standard Gaussian distribution on $\rset$.

By \Cref{lem:G_inequality}, we have almost surely for any $k\in \mathbb{N}$, 
\begin{equation}
\label{eq:r_inequality}
\normLigne{X_k-\tilde{X}_k}\leq W_k \eqsp .
\end{equation}
Another consequence of \Cref{lem:G_inequality} is that $Q_\gamma$ is stochastically monotonous (see
  \eg~\cite{lund:meyn:tweedie:96} or
  \cite{roberts:tweedie:2000}), more precisely if $(W_k)_{k\in\N}$ and $(\tilde W_k)_{k\in\N}$ are two chains given by \eqref{eq:def_r} with the same variables $(G_k,U_k)_{k\in\nset}$ with $W_0 \leqslant \tilde W_0$, then almost surely $W_k\leqslant \tilde W_k$ for all $k\in\N$. This nice property will be used several times in the analysis of this chain.

The main consequence of \eqref{eq:r_inequality} is the following result. 
\begin{corollary}
  \label{coro:bound_sticky}
Assume \Cref{ass:LipsAnd}-\ref{ass:LipsAnd_1} and \Cref{ass:bound}  hold.  Let $\cbf : \rset^{2d} \to [0,+\infty)$ of the form
  $\cbf(x,y) = \tcbf(\norm{x-y})$ for some non-decreasing function
  $\tcbf : [0,+\infty) \to [0,+\infty)$, $\tcbf(0)=0$. For any 
  $x,\tx \in \rset^d$ and $k\in \mathbb{N}$,
\begin{equation}
  \wassersteinD[\cbf](\updelta_x R_{\gamma}^k,\updelta_{\tx}\tR_{\gamma}^k) \leq \int_{\rset^{2d}} \cbf(y,\ty) K_{\gamma}^k((x,\tx), \rmd (y,\ty)) \leq
 \int_0^{+\infty}\tcbf(\tw) Q_{\gamma}^k(\normLigne{x-\tx}, \rmd \tw)   \eqsp.
\end{equation} 
\end{corollary}
\begin{proof}
Let $k\in \mathbb{N}$. By  \eqref{eq:r_inequality} and since $\tcbf$ is non-decreasing, we get almost surely $\tcbf(\normLigne{X_k- \tilde{X}_k}) \leq \tcbf(W_k)$. Taking the expectation concludes the proof. 
\end{proof}

From \Cref{coro:bound_sticky}, the question to get bounds on
$ \wassersteinD[\cbf](\updelta_x
R_{\gamma}^k,\updelta_{\tx}\tR_{\gamma}^k)$ boils down to the study of the Markov kernel $Q_{\gamma}$ on $[0,+\infty)$, which is the main part of our work.

\subsection{Analysis of the auxiliary Markov chain}
\label{sebsec:auxiliary_Markov_chain}
 We start with a Lyapunov/drift result.
\begin{proposition}
\label{prop:drift_v_norm}
Assume \Cref{ass:LipsAnd}-\ref{ass:LipsAnd_2}. Then for any $w\geqslant 0$,
\begin{equation}
Q_\gamma \VlyapDs_1(w)\leq (1-\gamma\mtt)\VlyapDs_1(w)\1_{(R_1,+\infty)}(w)+(1+\gamma\Ltt)\VlyapDs_1(w)\1_{(0,R_1]}(w)+\gamma \InftyBound \eqsp ,
\end{equation}
where $Q_{\gamma}$ is defined by \eqref{def_q_gamma} and  for any \tcrw{$w\in\rset_+$, $\VlyapDs_1(w)=w$} .
\end{proposition}
\begin{proof}
  The proof is postponed to \Cref{sec:proof_prop:drift_v_norm}. 
\end{proof}
\Cref{prop:drift_v_norm} implies in particular that
 for any $w\in \tcrw{\rset_+}$,
\tcrw{\begin{equation}
Q_\gamma \VlyapDs_1(w)\leq (1-\gamma\mtt)\VlyapDs_1(w)+\gamma[(\Ltt+\mtt)R_1 + \InftyBound] \eqsp. 
\end{equation}}
Then, a straightforward induction shows that for any $k \in\nset$,
\begin{equation}
Q_\gamma^k \VlyapDs_1(w)\leq (1-\gamma\mtt)^k\VlyapDs_1(w)+[(\Ltt+\mtt)R_1 + \InftyBound]/\mtt \eqsp,
\end{equation}
and therefore by \Cref{coro:bound_sticky} taking $\tcbf(t) =t$,
\begin{equation}
  \label{eq:conseq_drift_v_norm}
  \wassersteinD[1](\updelta_x R_{\gamma}^k,\updelta_{\tx}\tR_{\gamma}^k) \leq (1-\gamma\mtt)^k\norm{x-\tx}+[(\Ltt+\mtt)R_1 + \InftyBound]/\mtt  \eqsp. 
\end{equation}
However, this result is not sharp as $k \to \plusinfty$. Indeed, in
the case $\InftyBound =0$, $R_{\gamma}= \tR_{\gamma}$ and by
\Cref{theo:convergence_markov_chain_rsetd}, it holds that
$ \wassersteinD[1](\updelta_x
R_{\gamma}^k,\updelta_{\tx}\tR_{\gamma}^k) \to 0$ as
$k \to \plusinfty$, while the right-hand side of  \eqref{eq:conseq_drift_v_norm} converges to $(\Ltt+\mtt)R_1 /\mtt  \neq 0$. In particular, that is why adapting existing results, such as the one established in \cite{rudolf:schweizer:2018}, is not an option here. We need to refine our results in order to fill this gap. To this end, we need to analyze more precisely the long-time behavior of $Q_{\gamma}$. A
first step is to show that it is ergodic.

\begin{proposition}
\label{propo:existence_mes_statio}
Assume \Cref{ass:LipsAnd}-\ref{ass:LipsAnd_2}. For any
  $\gamma\in \ocint{0,\overline{\gamma}}$, $Q_\gamma$ admits a unique
  invariant probability measure $\mu_{\gamma}$ and is geometrically
  ergodic. In addition, $\mu_{\gamma}(\{0\}) >0$ and $\mu_{\gamma}$ is absolutely continuous  with respect to  the measure   $\updelta_{0} + \Leb$ on $([0,+\infty),\mcb{[0,+\infty)})$. Finally, in the case $\InftyBound \neq 0$, $\mu_{\gamma}$ and   $\updelta_{0} + \Leb$ are equivalent. 
\end{proposition}

\begin{proof}
  The proof is postponed to \Cref{sec:proof_existence_mes_statio}.
\end{proof}

\begin{corollary}
  \label{coro_ergo_sticky_control_wasser}
Assume \Cref{ass:LipsAnd} and \Cref{ass:bound}  hold.  Let $\cbf : \rset^{2d} \to [0,+\infty)$ of the form
  $\cbf(x,y) = \tcbf(\norm{x-y})$ for some non-decreasing function
  $\tcbf : [0,+\infty) \to [0,+\infty)$, $\tcbf(0)=0$. For any 
  $x,\tx \in \rset^d$ and $k\in \mathbb{N}$,
  \begin{equation}
    \label{eq:coro_ergo_sticky_control_wasser}
  \wassersteinD[\cbf](\updelta_x R_{\gamma}^k,\updelta_{\tx}\tR_{\gamma}^k) \leq
\int_0^{+\infty} \tcbf(\tw) \{Q_{\gamma}^k(\normLigne{x-\tx}, \cdot)  - \mu_{\gamma}\}(\rmd \tw) + \mu_{\gamma}(\tcbf)   \eqsp.
\end{equation} 
where $\mu_{\gamma}$ is the stationary distribution of $Q_{\gamma}$ given by \eqref{def_q_gamma}. In particular, if $x = \tx$, \sloppy$ \wassersteinD[\cbf](\updelta_x R_{\gamma}^k,\updelta_{x}\tR_{\gamma}^k) \leq \mu_{\gamma}(\tcbf)$.
\end{corollary}
\begin{proof}
  The proof of \eqref{eq:coro_ergo_sticky_control_wasser} is a
  consequence of \Cref{propo:existence_mes_statio} and
  \Cref{coro:bound_sticky}. The last statement follows from the fact
  that $Q_{\gamma}$ is stochastically monotonous. Indeed, by \Cref{lem:G_inequality},   
   for any  $w ,\tw \in \coint{0,\plusinfty}$, $w \leq \tw$, and $a \in \coint{0,\plusinfty}$,
  $Q_{\gamma}(w,\ccint{0,a}) \geq
  Q_{\gamma}(\tilde{w},\ccint{0,a})$. Therefore, for any
  $a \in \coint{0,\plusinfty}$, $w \mapsto Q_{\gamma}(w,\ccint{0,a})$ is
  non-increasing on $[0,+\infty)$ and for any non-increasing bounded function $f$,
  $Q_{\gamma} f(w) \geq Q_{\gamma}f(\tw)$ for any $w,\tw \in \coint{0,\plusinfty}$,
  $w \leq \tw$. As a result, a straightforward induction shows that
  for any $k \in \nset$, $w ,\tw \in \coint{0,\plusinfty}$, $w \leq \tw$, and
  $a \in \coint{0,\plusinfty}$,
  $Q_{\gamma}^k(w,\ccint{0,a}) \geq
  Q_{\gamma}^k(\tilde{w},\ccint{0,a})$. Then, we obtain $Q^k_\gamma(0,\ccint{0,a}) \geq \int_0^{+\infty} \mu_{\gamma}(\rmd w) Q_{\gamma}^k(w,\ccint{0,a}) = \mu_{\gamma}(\ccint{0,a})$. Since $\tcbf$ is non-decreasing on $[0,+\infty)$, we get $Q^{\tcrw{k}}_{\gamma}\tcbf(0) \leq \mu_{\gamma}(\tcbf)$, which combined with \eqref{eq:coro_ergo_sticky_control_wasser} completes the proof. 
\end{proof}

\Cref{coro_ergo_sticky_control_wasser} then naturally brings us to
derive moment bounds for the stationary distribution $\mu_{\gamma}$,
$\gamma \in \ocint{0,\bgamma}$ and quantitative convergence
bounds for $Q_{\gamma}$ to $\mu_{\gamma}$. Our next results address these two problems. 

\begin{theorem}
\label{theo:bound_moment_rset_star_invariant_mes}
Assume \Cref{ass:LipsAnd}-\ref{ass:LipsAnd_2}.
For any $\bdelta \in \ocint{0,\{ \Ltt^{-1} \wedge (\sigma\rme^{-1}/\InftyBound)^2\}}$ and  $\gamma\in \ocint{0,\bgamma}$,
\begin{equation}
  \label{eq:bound_moment_rset_star_invariant_mes}
 \int_{[0,+\infty)} w \,  \mu_\gamma(\rmd w)\leq \InftyBound\constMoment_1  \eqsp, \qquad \mu_\gamma((0,+\infty))\leq \InftyBound \constMoment_2 \eqsp ,
\end{equation}
where $\mu_{\gamma}$ is the stationary distribution of $Q_{\gamma}$ given by \eqref{def_q_gamma}, and, considering $\zeta$ given below in \eqref{eq:def_zeta},
\begin{align}
    \label{eq:def_delta_bar_lambda_1_eq:bound_moment_rset_star_invariant_mes}
  \constMoment_1&=\eta_1 R_1 (1+\Ltt/\mtt)+1/\mtt  \eqsp, \\
  \constMoment_2& =\rme^{(\deltaInf+\bgamma)\Ltt}\parentheseLigne{\constMoment_1(1+\bgamma \Ltt)/\deltaInf^{1/2}+[\deltaInf+\bgamma]^{1/2}}/(\sqrt{2\uppi}\sigma)+2\zeta[\deltaInf+\bgamma]^{1/2}\rme^{3(\deltaInf+\bgamma)\Ltt}/\sigma^3  \eqsp, \\
  \label{eq:def_delta_eq:bound_moment_rset_star_invariant_mes}
  \eta_1&=\left. [\deltaInf+\bgamma]^{1/2}\parentheseDeux{\frac{\tcrw{2}\zeta\rme^{3(\deltaInf+\bgamma)\Ltt}}{\sigma^3}+\frac{\rme^{(\deltaInf+\bgamma)\Ltt}}{2\sqrt{2\uppi}\sigma}}\middle/
\Phibf\parenthese{-\frac{(1+\bgamma\Ltt)R_1+(\deltaInf+\bgamma)\InftyBound}{2\deltaInf^{1/2}\sigma\rme^{-(\deltaInf+\bgamma) \Ltt}}}
        \right.\eqsp .
\end{align}
\end{theorem}

\begin{proof}
  The proof is postponed to \Cref{sec:proof-crefth}.
\end{proof}

\begin{theorem}
  \label{coro:exp_moment}
  Assume \Cref{ass:LipsAnd}-\ref{ass:LipsAnd_2}.
For any  $a >0$  and $\gamma\in \ocint{0,\bgamma_1}$,
\begin{equation+}
  \label{eq:bound_moment_exp_rset_star_invariant_mes}
  \int_{\tcrw{\rset_+}}  \tcrw{\parentheseDeux{\exp(a w)-1}}  \, \rmd \mu_\gamma( w)\leq \InftyBound\constMoment_3 \eqsp,
\end{equation+}
where \tcrw{$\bgamma_1$ and} $\constMoment_3$ \tcrw{are} explicitly given in the proof and $\mu_{\gamma}$ is the stationary distribution of $Q_{\gamma}$ given by \eqref{def_q_gamma}.
\end{theorem}
\begin{proof}
The proof is postponed to \Cref{sec:proof:coro:exp_moment}.
\end{proof}

% A quick study of the bounds provided by
% \Cref{theo:bound_moment_rset_star_invariant_mes} and
% \Cref{coro:exp_moment} on the following example yields that an appropriate choice for the
% parameter $\delta$ is
% $(1/2)\wedge \Ltt^{-1}$ and in that case $\bdelta  = \delta$ as $\InftyBound \to 0$.
% \alain{discussion sharpness two ar}
We now specify the convergence of
$Q_{\gamma}$ to $\mu_{\gamma}$ for any $\gamma \in \ocint{0,\bgamma}$.

\begin{theorem}
  \label{theo:convergence_Q_gamma}
Assume \Cref{ass:LipsAnd} and \Cref{ass:bound}  hold. There exist
  explicit constants $\rho \in \coint{0,1}$, $C \geq 0$ \tcrw{and $\bgamma_1\in \ocint{0,\bgamma}$} such that
  for any $\gamma \in \ocint{0,\bgamma_{\tcrw{1}}}$, $w \geqslant 0$, 
  \begin{equation}
    \label{eq:theo:convergence_Q_gamma}
    \Vnorm[\VlyapD]{\updelta_w Q_{\gamma}^k - \mu_{\gamma}} \leq C\rho^{\gamma k} \VlyapD(w) \eqsp,
  \end{equation}
  where $\VlyapD(w) = 1+ \tcrw{w}$ or $\VlyapD(w) = \exp(a \tcrw{w})$, for $a >0$. 
\end{theorem}

\begin{proof}
The proof is postponed to \Cref{sec:proof:theo:convergence_Q_gamma}.
\end{proof}

Combining the results of \Cref{coro_ergo_sticky_control_wasser},
\Cref{theo:bound_moment_rset_star_invariant_mes},
\Cref{coro:exp_moment} and \Cref{theo:convergence_Q_gamma} allows to
address the main questions raised in this section and prove \Cref{theo:main_results}.

\paragraph*{Discussion on the bounds provided by \Cref{theo:bound_moment_rset_star_invariant_mes}}
In this paragraph, we discuss how the constants $c_1,c_2$ given in \Cref{theo:bound_moment_rset_star_invariant_mes} \tcr{behave} with respect to the parameters $R_1,\Ltt, \mtt$ in the limit $\InftyBound \to 0$ and $\bgamma \to 0$. For ease of presentation, we also only consider the case $\sigma=1$. 

\begin{enumerate}[label=(\arabic*)]
\item First consider the case $R_1=0$. As $\mtt \to 0$,  $c_1,c_2 $ are of order $\mtt^{-1}$ and
 $ 1/[\mtt \bdelta^{1/2}] + \bdelta^{1/2}$ respectively for
$\bdelta \in \ocint{0,\{ \Ltt^{-1} \wedge
  (\sigma\rme^{-1}/\InftyBound)^2\}}$. Since $\Ltt$ can be taken arbitrarily small (as $R_1=0$), choosing  $\bdelta = \mtt^{-1}$, we 
obtain that $c_2$ is of order $\mtt^{-1/2}$. Note that the
dependency of $c_1,c_2$ with respect to $\mtt$ is sharp; see \Cref{ex:ar_proc} below. 
\item We now consider the case $R_1 \geq 1$, $\Ltt =0$. Note that in
  this case $\bdelta$ can be chosen arbitrarily  in $\ooint{0,1}$. Then,
  for some universal constants $C_1,C_2,C_3$,
  $\eta_1 \geq  C_1\bdelta^{1/2} / \Phibf\{ C_2 R_1/\bdelta^{1/2} + C_3
  \InftyBound \bdelta^{1/2} \}$. Therefore, taking
  $\bdelta = \mtt^{-1} \vee R_1^{2}$, we get that for some
  universal constants $D_1,D_2,E \geq 0$,
  $c_1 \leq  D_1[(R_1\vee \mtt^{-1/2}) + \mtt^{-1}]$,
  $c_2 \leq E \mtt^{-1/2} \vee R_1$. Note that the bound of
  $c_2$ with respect to $R_1$ and $\mtt$ is consistent with the
  results obtained in \cite{eberle:zimmer:2016} (see \cite[Lemma
  1]{eberle:zimmer:2016}) for the stationary distributions of
  continuous sticky processes.  Note that it is shown in \cite[Example
  2]{eberle:zimmer:2016} that this bound is sharp with respect to
  $R_1$ and $\mtt$.
\item In the case $R_1\wedge \Ltt \geq 1$, taking $\bdelta = \Ltt^{-1}$ since we are in the regime $\InftyBound \to 0$ \tcr{and $\bgamma \to 0$}, we get that up to logarithmic term and using $\bgamma \leq \Ltt^{-1}$, $c_1,c_2$ are smaller than $ C \exp[\rme^{4}(R_1\Ltt^{\half} +\InftyBound)^2]$ for some universal constant $C \geq 0$. The  estimate for $c_2$ is also consistent with  \cite[Lemma 1]{eberle:zimmer:2016} which holds for stationary distributions of continuous sticky processes. 
\end{enumerate}

\begin{example}
  \label{ex:ar_proc}
Consider the particular example of two auto-regressive
processes for which $\rmT_{\gamma}(y) = (1-\varrho \gamma) y$ and $\trmT_{\gamma}(y) = (1-\varrho \gamma) y + \gamma \varrho a$ for $\gamma \in \ooint{0,\varrho^{-1}}$ and $a,\varrho >0$. Then, on the one hand, \Cref{ass:LipsAnd} and \Cref{ass:bound} are satisfied with $R_1=0$, $\mtt = \varrho$ and $\InftyBound = \varrho a$ which lead to $c_1 \InftyBound \sim a$ and $c_2 \InftyBound \sim C a / \varrho^{1/2}$, as $\varrho \to 0$, for some universal constant $C \geq 0$. On the other hand, an easy computation (see \eg~\cite{JMLR:v20:18-173}) shows that the stationary distributions $\pi_{\gamma}$ and $\tpi_{\gamma}$ provided by \Cref{theo:convergence_markov_chain_rsetd} and \Cref{theo:ergo_tilde_R} are $\loiGauss(0,\varrho^{-1}(2-\gamma \varrho \gamma)^{-1})$ and $\loiGauss(a,\varrho^{-1}(2-\gamma \varrho \gamma)^{-1})$ respectively. Therefore, we get $\wassersteinD[1](\pi_{\gamma}, \tpi_{\gamma}) = a $ and $\tvnorm{\pi_{\gamma}- \tpi_{\gamma}} \sim C a / \varrho^{1/2}$ as $\varrho \to 0$.% and we get up to universal constants $c_1$ and $c_2$.  
\end{example}

%%% Local Variables:
%%% mode: latex
%%% TeX-master: "main_imsart"
%%% End:

\section{Continuous-time limit}
\label{sec:contin_limit}

In the case where $\rmT_{\gamma}$ and $\tilde{\rmT}_{\gamma}$ are
specified by \eqref{eq:form_T_Euler}, then under appropriate
conditions on $b$ and $\tilde{b}$, it can be shown, see \eg~\cite[Proposition 25]{debortoli2019convergence}, that for any $T \geq 0$ and $x \in \rset^d$,
\begin{equation}
  \label{eq:convergence_discre_to_cont}
  \lim_{m \to \plusinfty} \{\Vnorm[V]{\updelta_x R_{T/m}^{m}-\updelta_{x}P_T}+ \Vnorm[V]{\updelta_x \tilde{R}_{T/m}^{m}-\updelta_{x}\tilde{P}_T}\} = 0\eqsp,
\end{equation}
for some measurable function $V : \rset^d \to \coint{1,\plusinfty}$ and where  $(P_t)_{t \geq 0}$ and $(\tilde{P}_t)_{t \geq 0}$ are the Markov semigroup corresponding \tcrw{to the diffusions $\rmd \mathbf{X}_t=b(\mathbf{X}_t)\tcr{\rmd t}+\sigma\rmd B_t$ and $\rmd \tilde{\mathbf{X}}_t=\tilde{b}(\tilde{\mathbf{X}}_t)\tcr{\rmd t}+\sigma\rmd B_t$, where $(B_t)_{t\in \rset_+}$ is a standard $d$-dimensional Brownian motion}. Then, this naturally implies convergence in total variation and also Wasserstein distance of order $p$ if $\inf_{x \in \rset^d\tcrw{\setminus\{0\}}} \{V(x)/\norm[p]{x}\} >0$. 
As a consequence, results of \Cref{sec:main_result_coupling} immediately transfer to the continuous-time processes. More precisely, let  $\cbf : \rset^{2d} \to [0,+\infty)$ of the form
  $\cbf(x,y) = \tcbf(\norm{x-y})$ for some non-decreasing function
  $\tcbf : [0,+\infty) \to [0,+\infty)$, $\tcbf(0)=0$. If \eqref{eq:convergence_discre_to_cont} holds and  $\sup_{x,y \in \rset^d}\{\cbf(x,y)/\{V(x)+V(y)\}\tcrw{\}} < \plusinfty$, we get \tcrw{if $\wassersteinD[\cbf]$ satisfies} the triangle inequality that for any $x,\tx\in\rset^d$, $T > 0$,
  $  \wassersteinD[\cbf](\updelta_x P_T ,\updelta_{\tx}\tilde{P}_T)  \leq \liminf_{m \to \plusinfty} \wassersteinD[\cbf](\updelta_x R_{T/m}^{m} , \updelta_{\tx} \tilde{R}_{T/m}^{m})$. Then, results of \Cref{sec:main_result_coupling} can be applied implying if \Cref{ass:LipsAnd} and \Cref{ass:bound} holds,  that for any $x,\tx \in \rset^d$ there exist $C_1,C_2\geq 0$ such that for any $T \geq 0$,  $ \wassersteinD[\cbf](\updelta_x P_T ,\updelta_{\tx}\tilde{P}_T) \leq C_1\rho^T + C_2 \InftyBound$. We therefore generalize the result provided in  \cite{eberle:zimmer:2016} which is specific to the total variation distance. We do not give a specific statement for this result which is mainly technical and is not the main subject of this paper.   Instead, the goal of this section is to study the continuous-time limit of the coupling \eqref{eq:def_refl_discret} (and not only of its marginals) toward  some continuous-time sticky diffusion.  % more precisely, the convergence of the chain $(W_k)_{k\geqslant 0}$ with transition operator $Q_\gamma$ toward the one-dimensional sticky diffusion considered in \cite{eberle:zimmer:2016}.

More precisely, let  $(\gamma_n)_{n \in\nset}$ be  a sequence of step sizes such
 that $\lim_{n \to \plusinfty} \gamma_n = 0$\tcrw{, $\gamma_n\leq \bgamma$} and $w_0 \geqslant 0$. Then, consider the sequence of Markov chains
 $\{(\Wn_k)_{k \in\nset}\, : \, n \in \nset\}$ for any $n \in \nset$,
 $(\Wn_k)_{k \in\nset}$ is the Markov chain defined by
 \eqref{eq:def_r} with $\Wn_0 = w_0$, $\gamma = \gamma_n$ and
 therefore associated with the Markov kernel $Q_{\gamma_n}$. 
Let  $\{(\Wnbf_t)_{t \in \ooint{0,\plusinfty}}\, : \, n \in \nset\}$ be the continuous linear interpolation of  $\{(\Wn_k)_{k \in\nset}\, : \, n \in \nset\}$, \ie~the sequence of continuous processes
 defined for any
 $n \in \nset$, $t \in \ooint{0,\plusinfty}$ by 
 \begin{equation}
   \label{eq:def_wnbf}
   \Wnbf_t = \Wn_{\floor{t/\gamma_n}} + \{\Wn_{\ceil{t/\gamma_n}}-\Wn_{\floor{t/\gamma_n}}\}\{t/\gamma_n - \floor{t/\gamma_n}\} \eqsp. 
 \end{equation}
 Note that for any $k \in \nset$ and $h \in \ccint{0,\gamma_n}$,
 $ \Wnbf_{k\gamma_n+h} = W_{k} +(h/\gamma_n)\{\Wn_{k+1}-\Wn_{k}\}$.  We
 denote by $\wiener = \rmC([0,+\infty),\rset)$ 
 endowed with the uniform topology on compact sets, $\wienersigma$ its
 corresponding $\sigma$-field  and $(\Wrm_t)_{t \geq 0}$ the
 canonical process defined for any $t \in \ooint{0,\plusinfty}$ and $\omega \in\wiener$ by $\Wrm_t(\omega) = \omega_t$.
 Denote by $(\wienersigma_t)_{t \geq 0}$ the filtration associated with $(\Wrm_t)_{t \geq 0}$. Note that  $\{(\Wnbf_t)_{t \in \ooint{0,\plusinfty}}\, : \, n \in \nset\}$ is a sequence of $\wiener$-valued random variables. The main result of this section concerns the convergence in distribution of this sequence. 

 We consider the following assumption on the function $\tau_{\gamma}$.
 \begin{assumptionA}
   \label{ass:sticky_cont}
There exists a function
   $\kappa : [0,+\infty)\to \tcrw{\rset}$ such that for any $\gamma \in \ocint{0,\bgamma}$, 
   $\tau_{\gamma}(w) = w+\gamma\kappa(w)$ and $\kappa(0)=0$. In
   addition, $\kappa$ is $\Ltt_{\kappa}$-Lipschitz: for any
   $w_1,w_2 \in \ooint{0,\plusinfty}$, $\absLigne{\kappa(w_1)-\kappa(w_2)} \leq \Ltt_{\kappa} \abs{w_1-w_2}$. 
 \end{assumptionA}
 This is not a restrictive condition since, under \Cref{ass:LipsAnd}, up to a possible modification of  $\tau_{\gamma}$, it is always possible to ensure \Cref{ass:sticky_cont}.

 Under \Cref{ass:sticky_cont}, we consider a sticky process \cite{watanabe:1971,watanabe1971_II,eberle:zimmer:2016}, which solves the stochastic differential equation
 \begin{equation}
   \label{eq:def_sticky_cont}
   \rmd \Wbf_t =  \{\kappa (\Wbf_t) + \InftyBound\} \rmd t + 2 \sigma \1_{(0,+\infty)} (\Wbf_t) \rmd B_t \eqsp,
 \end{equation}
 where $(B_t)_{t \geq 0}$ is a one-dimensional Brownian motion. Note
 that for any initial distribution $\mubf_0$ on
 $(\rset,\mcbb(\rset^d))$, \eqref{eq:def_sticky_cont} admits a unique
 weak solution by \cite[\tcrw{Lemma 17, Theorem 22}]{eberle:zimmer:2016}. 
%  for any 
%  $\varphi \in \rmC^{\infty}(\rset, \rset)$, with compact support,
%  \begin{equation}
%    \label{eq:martingale_problem_sticky}
% \parenthese{\varphi(\Wrm_t) - \varphi(\Wrm_0) - \int_{0}^t \generator \varphi(\Wrm_u) \rmd u }_{t \geq 0} \text{ is a $(\wienersigma_t)_{t \geq 0}$-martingale} \eqsp,
% \end{equation}

The main result of this section is the following.
\begin{theorem}
  \label{theo:continuous_limit}
  Assume \Cref{ass:sticky_cont}. Then, the sequence
  $\{(\Wbfn_t)_{t \geq 0} \, : \, n \in \nset\}$ defined by \eqref{eq:def_wnbf} converges in distribution
  to the solution $(\Wbf_t)_{t \geq 0}$ of the SDE \eqref{eq:def_sticky_cont}. 
\end{theorem}

The proof of this theorem follows the usual strategy employed to show
convergence of a sequence of continuous processes to a Markov
process.  A first step is to show that under \Cref{ass:sticky_cont},
$\{(\Wbfn_t)_{t \geq 0} \, : \, n \in \nset\}$ is uniformly
bounded in $\mrl^q$ for some $q \geq 2$, on $\ccint{0,T}$ for any $T \geq 0$.

\begin{proposition}
  \label{propo:moment_conv_continuous}
  Assume \Cref{ass:sticky_cont}. Then for any $T \geq 0$, there exists
  $C_T \geq 0$ such that  $\sup_{n \in \nset} \expeLigne{\sup_{t \in \ccint{0,T}}  \{\Wbfn_t\}^4} \leq C_T$ where $(\Wbfn_t)_{t \geq 0}$ is defined by \eqref{eq:def_wnbf}. 
\end{proposition}
\begin{proof}
The proof is postponed to \Cref{sec:proof-crefpr_moment_cont}. 
\end{proof}
  
Then, we are able to obtain the tightness of the sequence of stochastic processes \sloppy$\{(\Wnbf_t)_{t \geq 0}  \, : \, n \in \nset\}$. 
\begin{proposition}
  \label{propo:tight_convergence_continuous}
  Assume \Cref{ass:sticky_cont}. Then,  $\{(\Wbfn_t)_{t \geq 0} \, : \, n \in \nset\}$ is tight in $\wiener$. 
\end{proposition}
\begin{proof}
  The proof is postponed to \Cref{sec:proof_tight_cont}.
\end{proof}

Denote for any $n \in \nset$, $\mubf_n$ the distribution of $(\Wnbf_t)_{t \geq 0}$ on $\wiener$.
Then, by Prohorov's Theorem \cite[Theorem 5.1,5.2]{billingsley:1999},
$(\mubf_n)_{n \in \nset}$ admits a limit
point.  If we now show that every limit point associated with
$\{(\Wbfn_t)_{t \geq 0} \, : \, n \in \nset\}$ is a solution of the SDE \eqref{eq:def_sticky_cont} using
that $\{(\Wbfn_t)_{t \geq 0} \, : \, n \in \nset\}$ is tight
again and since \eqref{eq:def_sticky_cont}  admits a unique weak solution, the proof of \Cref{theo:continuous_limit} will be  completed. To establish this result, we use the characterization of solutions of SDEs through martingale problems. More precisely by \cite[Theorem 1.27]{cherny:engelbert:2005}, the
 distribution $\mubf$ on $\wiener$ of $(\Wbf_t)_{t \geq 0}$, solution of \eqref{eq:def_sticky_cont}, is the
 unique solution to the martingale problem associated with $\mubf_0$,
 the drift function $w \mapsto \kappa (w) + \InftyBound$ and the
 variance function $2 \sigma \1_{(0,+\infty)}$, \ie~it is the unique probability
 measure satisfying on the filtered probability space 
 $(\wiener,\wienersigma, (\wienersigma_t)_{t \geq 0}, \mubf)$:
 \begin{enumerate}[wide, labelwidth=!, labelindent=0pt,label=(\alph*)]
 \item  the
 distribution of $\Wrm_0$ is $\mubf_0$;
\item the processes $(\Mrm_t)_{t\geq 0}$, $(\Nrm_t)_{t \geq 0}$ defined for any $t \geq 0$ by
  \begin{equation}
    \label{eq:def_Mrm_Nrm}
    \Mrm_t = \Wrm_t - \Wrm_0 - \int_{0}^{t} \{ \InftyBound + \kappa(\Wrm_u)\} \rmd u \eqsp, \qquad 
    \Nrm_t = \Mrm_t^2 - 4 \sigma^2 \int_{0}^t \1_{(0,+\infty)}(\Wrm_u) \rmd u \eqsp,
  \end{equation}
are $(\wienersigma_t)_{t \geq 0}$-local martingales. 
\end{enumerate}
In other words, it corresponds in showing that $(\Mrm_t)_{t \geq 0}$ is a $(\wienersigma_t)_{t \geq 0}$-local martingales and by \cite[Theorem 1.8]{revuz:yor:1994} identifying its quadratic variation $(\qvar{\rmM}_t)_{t \geq 0}$ as the process \sloppy$(4 \sigma^2 \int_{0}^t \1_{(0,+\infty)}(\Wrm_u) \rmd u)_{t \geq 0}$. 
Therefore, \Cref{theo:continuous_limit} is a direct consequence of the following result.

\begin{theorem}
  \label{theo:martingale_prob_martinM_N}
Assume \Cref{ass:sticky_cont}. Let $\mubf_{\infty}$ be a limit point of $(\mubf_n)_{n\in\nset}$. Then, the two processes $(\Mrm_t)_{t \geq 0}$ and $(\Nrm_t)_{t \geq 0}$ defined by \eqref{eq:def_Mrm_Nrm} are $(\wienersigma_t)_{t \geq 0}$-martingales on  $(\wiener,\wienersigma, (\wienersigma_t)_{t \geq 0}, \mubf_{\infty})$. 
\end{theorem}
\begin{proof}
  The proof is postponed to \Cref{sec:proof_theo:martingale_prob_martinM_N}.
\end{proof}

%%% Local Variables:
%%% mode: latex
%%% TeX-master: "main_imsart"
%%% End:

% !TeX root = main.tex
\section{An application in Bayesian statistics: parameter estimation in an ODE}
\label{sec:application}

\subsection{Setting and verifying the assumptions}
Consider an ordinary differential equation ODE on $\R^n$ of the form
\begin{equation}\label{eq:EDOappli}
  \dot x_{\theta}(t) \ = \ f_{\theta}(x_{\theta}(t),t)\,, \qquad x_{\theta}(0)=x_0 \in \R^n\,,
\end{equation}
where $\{f_{\theta} \, : \, \theta \in \rset^d\}$ is a family of function from $\R^n \times [0,+\infty)$ to $ \R^n$ parametrized by some parameter $\theta \in \R^d$. In all this section $x_0 \in \rset^n $ is assumed to be fixed  and we consider the following assumption.
\begin{assumptionAO}\label{ass:ODE}
For all $\theta\in\R^d$ there exists a unique solution of \eqref{eq:EDOappli} defined for all positive times, which we denote  by $(x_{\theta}(t))_{t\geqslant 0}$. In addition, the functions  $(\theta,x,t) \in \R^d\times\R^n\times[0,+\infty) \mapsto f_{\theta}(x,t)$ and $(\theta,t) \in \R^d\times[0,+\infty) \mapsto x_{\theta}(t)$  are continuously differentiable. 
  %For $T>0$, there exist a compact $\Omega_T$ of $\R^n$ such that for all $\theta\in\R^d$ and $t\in[0,T]$, $x_\theta(t) \in \Omega_T$.
 \end{assumptionAO}
 In fact the continuous differentiability of $(\theta,t)   \mapsto x_{\theta}(t)$ is a consequence of the one of 
 $(\theta,x,t) \mapsto f_{\theta}(x,t)$, see \eg~\cite[Theorem 4.D]{zeidler:1986}.
 
To fix ideas, throughout this section, we will repeatedly discuss the following case of a logistic equation.
\begin{example}\label{exemple:logistique}
For $r\in \mathrm C^1(\R,\R_+)$, set $f_\theta(x)=x(1-r(\theta)x)$ for any $\theta,x\in\R$, so that \eqref{eq:EDOappli} reads
\[\dot x_{\theta}(t) \ = \ x_{\theta}(t) \po 1 - r(\theta) x_{\theta}(t)\pf\eqsp, \qquad x_{\theta}(0)=x_0\,,\]
with $x_0\geqslant 0$.
In this example, \Cref{ass:ODE} holds and,  $r$ and $x_0$ being positive, for all $\theta\in\R$, the solution of \eqref{eq:EDOappli} is such that $x_\theta(t) \in [0,\rme^t x_0]$ for all $t\geqslant 0$. Indeed, $x \mapsto x(1-r(\theta)x)$ is locally Lipschitz continuous, which yields  existence and uniqueness of a maximal solution. Since $0$ is always an equilibrium, solutions stay positive, from which $\tcr{\dot x_{\theta}}(t) \leqslant x_{\theta}(t)$ for all $t\geqslant 0$, implying that $x_{\theta}(t)\leqslant e^t x_0$ for all $t\geqslant 0$. This also implies non-explosion, hence the solution is defined on $[0,+\infty)$.
\end{example}

 We consider the problem of estimating $\theta$ based on some observation of a trajectory of the ODE. 
 More precisely, for $T>0$, $N\in\N^*$, $(t_1,\ldots,t_N) \in \rset^N$, $0 < t_1<\dots<t_N=T$, the statistical model corresponding to the observation $\mathbf{y}=(y_i)_{i\in\{1,\ldots,N\}}\in(\R^n)^N$ is given by
 \begin{equation}
   \label{eq:stat_model_ODE_I}
   y_i \ = \ x_{\theta}(t_i) + \varepsilon_i \eqsp,
 \end{equation}
for $\theta \in \R^d$ and  where $(\varepsilon_i)_{i\in\{1,\ldots,N\}}$ are \tcr{\iid}~random variables  on $\R^n$ distributed according to some known positive density $\varphi_{\varepsilon}$ with respect to the Lebesgue measure. Given a prior distribution with positive density $\pi_0$ on $\R^d$, the a posteriori distribution for this model admits a positive density $\pi$ with respect to the Lebesgue measure which is characterized by the potential $U$ given by (up to an additive constant)
\[
  -\log\pi  (\theta) = U(\theta)  = - \ln \pi_0(\theta) -  \sum_{i=1}^N \ln \varphi_\varepsilon\po y_i - x_{\theta}(t_i)\pf \eqsp. 
\]
% up to a non-relevant additive constant.
We consider the following assumption on $\pi_0$ and $\varphi_{\varepsilon}$ setting $-\log(\pi_0) = U_0$.
\begin{assumptionAO}\label{ass:ODEbounded2}
  The functions $\pi_0$ and $\varphi_{\varepsilon}$ are twice continuously differentiable and there exist
  $\mtt_U >0$, $\Ltt_U,R_U\geq 0$ such that $\na U_0$ is $\Ltt_U$-Lipschitz continuous and  for any $\theta,\tilde\theta\in\R^d$ with $\normLigne{\theta-\tilde\theta}\geqslant R_U$,
\[\psLigne{\theta-\tilde\theta}{\na U_0(\theta)-\na U_0(\tilde\theta)} \ \geqslant \ \mtt_U \normLigne{\theta-\tilde\theta}^2\,.\]
%There exist $\bar h>0$ and $\Omega\subset \R^b\times\mathcal M_{d,n}(\R)$ such that for all $i\in\{1,\ldots,N\}$, $\theta\in\R^d$ and $h\in(0,\bar h)$, $z_\theta(t_i) ,\Phi_i^h(\theta)\in\Omega$. There exist $\Ltt_F,\Ltt_F'>0$ such that for all $z,\tilde z\in\Omega$, $t\geqslant 0$, $\theta,\tilde\theta\in\R^d$,
%\begin{equation}\label{eq:FdemoODE}
%\norm{F_{\theta}(z,t) - F_{\tilde \theta}(\tilde z,t)} \ \leqslant \ \Ltt_F  \norm{\theta-\tilde\theta} + \Ltt_F'\norm{z-\tilde z}\,.
%\end{equation} 
%Moreover, $\ln \varphi_\varepsilon$ is $\mathrm C^2$.
\end{assumptionAO}

In practice, expectations with respect to the posterior distribution can be approximated by ergodic means of the Unadjusted Langevin Algorithm (ULA), namely the Markov chain 
\begin{equation}
  \label{eq:def_ULA_ODE}
  X_{k+1} \ = \ X_k-\gamma \na U(X_k) + \sqrt{2\gamma} Z_{k+1} \eqsp,
\end{equation}
where $\gamma>0$ and $(Z_k)_{k\in\N}$  are \tcr{\iid}~standard Gaussian variables. The long-time convergence of this algorithm and the \tcr{asymptotic} bias on the invariant measure due to the time discretization are well understood, see e.g. \cite{dalalyan2017theoretical,durmus2017nonasymp,durmus2019high,dalalyan2019user,debortoli2019souk,durmus2021asymptotic} and references therein. However, in the present case, it is not possible to sample this Markov chain, as the  exact computation of 
\begin{equation}
  \label{eq:def_poten_ODE_U}
  \nabla U(\theta) \ = \ - \nabla_{\theta} \ln \pi_0(\theta) +  \sum_{i=1}^N \nabla_{\theta} x_{\theta}(t_i)  \nabla_x \ln \varphi_\varepsilon\po y_i - x_{\theta}(t_i)\pf \eqsp,
\end{equation}
is not possible in most cases because of the term involving $x_{\theta}$ and  $\nabla_{\theta} x_{\theta}$. Here $\nabla_{\theta}$ and $\nabla_x$ denote the gradient operator with respect to $\theta$ and $x$ respectively.
Therefore, only approximations of these two functions can be used in place of $( x_{\theta}(t_i) ,\nabla_{\theta} x_{\theta}(t_i) )_{i\in\cco 0,N\ccf}$, which leads to an additional discretization bias. Our results based on the sticky coupling yields a quantitative bound on this error (with respect to the ideal ULA above). Let us detail this statement.

First,  remark that $t \mapsto z_\theta(t)=( x_{\theta}(t) ,\nabla_{\theta} x_{\theta}(t) )$ solves
\begin{equation}\label{ed:ODEz}
\dot z_\theta(t) \ = \ F_{\theta}(z_\theta(t),t)\qquad z_\theta(0)=z_0=(x_0,0)
\end{equation} 
on $\R^n\times\mathcal M_{d,n}(\R)$   with for any $x \in \rset^n$, $\mathbf{A} \in \mathcal M_{d,n}(\R)$, $\theta \in \rset^d$, $t\geq 0$,
\begin{equation}
\label{eq:def_F_theta}
F_{\theta}((x,\mathbf{A}),t) \ = \ \po f_\theta (x,t), \nabla_\theta f_{\theta}(x,t) + \mathbf{A}  \na_x f_{\theta}(x,t)\pf\,.
\end{equation}
Provided $f_\theta$, $\na_\theta f_\theta$ and $\na_x f_\theta$ are computable, in practice this ODE can be approximated by standard numerical schemes. For instance, a basic explicit Euler discretization with time-step $h>0$ is given by
\begin{equation}
\tilde z_\theta^h(0)=z_0\,,\qquad \tilde z_\theta^h((k+1)h) \  = \ \tilde z_\theta^h(kh) + h F_{\theta}\po\tilde z_\theta^h(kh),kh\pf\qquad \forall k\in\N
\end{equation}
and
\begin{equation}\label{eq:EulerEDO}
\tilde z_\theta^h(t) = \tilde z_\theta^h(kh) + (t-kh) F_{\theta}\po\tilde z_\theta^h(kh),kh\pf\,,\qquad t\in [kh,(k+1)h) \eqsp.
\end{equation} 
 To establish the consistency of this approximation (with some uniformity in $\theta$), we consider the following condition.
\begin{assumptionAO}\label{ass:ODEbounded}
  There exist $\Ltt_F,\Ltt_F',\mathtt{C}_F,\delta>0$ and a compact set $\msk\subset \R^n\times\mathcal M_{d,n}(\R)$  such that the following holds.   For all $t\in[0,T]$ and $\theta\in\R^d$, the ball centered at $ z_\theta(t)$ and radius $\delta$ is included in $\msk$. Moreover, for all $z,\tilde z\in\msk$, $t,s\in[0,T]$ and $\theta,\tilde\theta\in\R^d$, $\norm{F_\theta(z,t)}\leqslant \mathtt \mathtt{C}_F$ and
\begin{equation}\label{eq:FdemoODE}
\norm{F_{\theta}(z,t) - F_{\tilde \theta}(\tilde z,s)} \ \leqslant \ \Ltt_F  \normLigne{\theta-\tilde\theta} + \Ltt_F'\po \normLigne{z-\tilde z}+|t-s|\pf\,.
\end{equation} 
\end{assumptionAO}
\begin{proposition}\label{prop:Euler}
  Assume \Cref{ass:ODE} and  \Cref{ass:ODEbounded}. There exist $\bar h,C>0$ such that for all $h\in(0,\bar h]$, $\theta\in \R^d$, and $t \in \ccint{0,T}$,
   $\tilde z_\theta^h(t)\in\msk$ and $\sum_{i=1}^N\norm{z_\theta(t_i) - \tilde z_\theta^h(t_i)} \ \leqslant \ C h$,
where $z_\theta$ solves \eqref{ed:ODEz} and $\tilde z_\theta^h$ is given by \eqref{eq:EulerEDO}.
\end{proposition}
\begin{proof}
The proof is postponed to \Cref{subsec:postponed-application}.
\end{proof}
\begin{example}[Continuation of \Cref{exemple:logistique}]
  \label{exemple:logistique_2}
  Let us check for instance that  \Cref{ass:ODEbounded} is satisfied for \Cref{exemple:logistique} provided that $r$ is twice continuously differentiable on $[0,+\infty)$ with, for some $\Ltt_r,\Ltt_r',\Ltt_r''>0$,
  \begin{equation}
    \label{eq:condition_r_logistique}
    \text{ $r,r'$ and $r''$  uniformly bounded respectively by $\Ltt_r$, $\Ltt'_r$ and $\Ltt_r''$  \eqsp.}
  \end{equation}
  We may consider for example $r: \theta \mapsto  a_1 \theta^2/(\theta^2+a_2)$ for $a_1,a_2 \in \ooint{0,\plusinfty}$.
   Recall that $f_\theta(x)=x(1-r(\theta)x)$ and $x_\theta(t) \in [0,\rme^tx_0]$ for all $\theta\in\R$ and all $t\in[0,T]$, so that for any $t \geq 0$ and \tcrw{$\theta \in\rset$},
\begin{equation}
  |\partial_\theta f_\theta(x_\theta(t))| = |r'(\theta) x^2_\theta(t)| \leqslant \Ltt_r' \rme^{2t}x_0^2 \eqsp. 
%  \partial_x f_\theta(x_\theta(t))| &= | 1 - 2r(\theta) x_\theta(t)|  \leqslant  1+2\Ltt_r \rme^{t} x_0 \eqsp.
\end{equation}
% for all $t\geqslant 0$. 
 Notice that $1/r(\theta)$ is an equilibrium of the equation, so that it cannot be crossed by other solutions. Hence, on the one hand, if $1 \leqslant r(\theta) x_0$ then $x_\theta$ is non-increasing (in particular $x_\theta(t)\leqslant x_0$ for all $t\geqslant 0$) while, on the other hand, if $1 \geqslant r(\theta) x_0$, then  $1 \geqslant r(\theta) x_\theta(t)$ for all $t\geqslant 0$.   In both cases, we get that for all  $t\geqslant 0$, 
\[|\partial_x f_\theta(x_\theta(t))| = | 1 - 2r(\theta) x_\theta(t)|  \leqslant  1+2\Ltt_r x_0\,. \]
 Combining the two previous bounds,
%   \[\partial_t \norm{\partial_\theta x_\theta(t)}^2 \ \leqslant \ 2 \norm{\partial_\theta x_\theta(t)} \po \Ltt_r/r_*^2 + \norm{\partial_\theta x_\theta(t)}  \pf \]
% Then, we get
\[
  \abs{\partial_\theta x_\theta(t)} \ \leqslant \ \int_0^t \po \Ltt_r' \rme^{2t}x_0^2 + \po 1+2\Ltt_r x_0\pf  \abs{\partial_\theta x_\theta(s)}\pf \dd s \eqsp,
\] 
and thus by Grönwall's inequality, $\abs{\partial_\theta x_\theta(t)}  \leqslant     M_t  = \Ltt_r'  x_0^2t \rme^{ \po 3+2\Ltt_r x_0\pf t} $,
for all $t\in[0,T]$ and $\theta\in\R$. Then, for any $\delta>0$, \Cref{ass:ODEbounded} is satisfied with  $\msk=[-\delta,\rme^Tx_0+\delta]\times[-M_T-\delta,M_T+\delta]$.  Since $\msk$ is compact and as, by \eqref{eq:def_F_theta},
  for any $x \in \tcrw{\rset}$, $\mathbf{A} \in \mathcal M_{\tcrw{1,1}}(\R)$, $\theta \in \tcrw{\rset}$, $t\geq 0$,
 \[F_\theta((x,\mathbf{A}),t) \ = \ \po x(1 - r(\theta)x)\ , \  \tcrw{-}r'(\theta) x^2 + \mathbf{A}(1 - 2r(\theta) x)\pf \,,\]
 then \eqref{eq:FdemoODE} easily follows from the condition \eqref{eq:condition_r_logistique}.
\end{example}
% Moreover, for $z=(x,\nu)\in\msk$,
%\[\|F_\theta(z,t)\| \ \leqslant \ |f_\theta(x)|+|\partial_\theta f_\theta(x)|+|\nu||\partial_x f_\theta(x)| \ \leqslant \ \frac{1}{r_*} + \frac{\Ltt_r}{r_*^2} + M \]
%and
%\begin{align}
%\|F_\theta(z,t)-F_{\tilde\theta}(\tilde z,s)\| \ \leqslant &\  |f_\theta(x)-f_{\tilde\theta}(\tilde x)|+|\partial_\theta f_\theta(x)-\partial_\theta f_{\tilde \theta}(\tilde x)|\\
%& \ +|\nu-\tilde \nu||\partial_x f_\theta(x)| + |\tilde \nu||\partial_x f_\theta(x)-\partial_x f_{\tilde \theta}(\tilde x)|\\
%\ \leqslant &\ \tilde x^2 |r(\theta)-r(\tilde\theta)| + \po r(\theta)+|r'(\theta)|\pf |x^2-\tilde x^2|\\
%& \ +|\nu-\tilde \nu| + 2|\tilde \nu| \po |r(\theta)-r(\tilde \theta)||\tilde x|\pf
%\end{align}

% Assuming that there exists $K>0$ such that $F_{\theta}$ is $K$-Lipschitz continuous for all $\theta \in \R^d$, we get by classical results on the Euler scheme that there exists $C>0$ such that for all $\theta\in\R^d$ and all $h>0$,
% \[\sup_{i\in\{1,\ldots,N\}}\norm{z_\theta(t_i) - \tilde z_\theta^h(t_i)} \ \leqslant \ C h\,.\]
\tcrw{We aim to consider general discretization schemes with higher orders than the Euler discretization defined in (31). Therefore, in the following we consider a solver $\Psi^h:\R^d \rightarrow (\R^n\times\mathcal M_{d,n}(\R))^N$ for $h > 0$ satisfying the condition:} % namely $(\Psi_1^h(\theta),\dots,\Psi_N^h(\theta))$ is a computable approximation of $(z_\theta(t_1),\dots,z_\theta(t_N))$.
\begin{assumptionAO}\label{ass:ODEsolver}
  There exist $\bar h,\mathtt{C}_{\Psi},\alpha>0$ such that, for any $\theta\in\R^d$ and $h\in(0,\bar h]$,
  \begin{equation}
    \label{eq:36}
    \sum_{i=1}^N\normLigne{z_\theta(t_i) - \Psi_i^h(\theta)} \ \leqslant \ \mathtt{C}_{\Psi} h^{\alpha}\eqsp,
  \end{equation}
  where $z_{\theta}$ is a solution of \eqref{ed:ODEz} and $\Psi_i^h : \rset^d \to \rset^n \times \mathcal{M}_{d,n}(\rset)$ is the $i$-th component of $\Psi^h$.
%\[\,.\]
\end{assumptionAO}
When \Cref{ass:ODEbounded} and \Cref{ass:ODEsolver} are both satisfied,  without loss of generality, we assume  furthermore that $\bar h$ is sufficiently small so that $\mathtt \mathtt{C}_{\Psi} \bar h^\alpha \leqslant \delta$. This implies that $\Psi_i^h(\theta) \in \msk$ for all $\theta\in\R^d$, $i\in\{1,\ldots,N\}$ and $h\in (0,\bar h]$.

Writing $\Psi_i^h(\theta) = (\tilde x_{\theta}^h(t_i),G_\theta^h(t_i))$, we can consider for any $\theta \in \rset^d$,
\begin{equation}
  \label{eq:def_tilde_b_ODE}
  \tilde b_h(\theta) \ = \  - \nabla_{\theta} \ln \pi_0(\theta) +  \sum_{i=1}^N G_{\theta}^h(t_i)\cdot  \nabla_x  \ln \varphi_\varepsilon \po y_i - \tilde x_{\theta}^h(t_i)\pf \eqsp,
\end{equation}
as an approximation of $\nabla U$ \eqref{eq:def_poten_ODE_U}
% , lead us
% to consider the the Markov chain on $\R^d$ with transitions
% \[\tX_{k+1} \ = \ \tX_k-\gamma \tilde b_h(\tX_k) + \sqrt{2\gamma} \tZ_{k+1}\,.\]
Remark that, now, in contrast to $b(\theta)=\na U (\theta)$, it is possible in practice to evaluate $\tilde b_h(\theta)$ for $\theta\in\R^d$, provided $\nabla_x  \ln \varphi_\varepsilon $ and $\nabla_{\theta} \ln \pi_0$ can be evaluated. % , which for instance is the case if the prior and the noise distributions are Gaussian laws.
We  now assess the error due to the use of $\tilde b_h$ in place of the exact gradient in \eqref{eq:def_ULA_ODE} by verifying that the assumption of \Cref{sec:main_result_coupling} are satisfied.
For $\gamma,h>0$ and $\theta\in\R^d$, denote
\begin{equation}\label{eq:TtTODE}
\rmT_{\gamma}(\theta) \ = \ \theta-\gamma\na U(\theta)\,,\qquad \tilde{\rmT}_{\gamma,h}(\theta) \ = \ \theta-\gamma\tilde b_h(\theta)\,.
\end{equation}
% In order to verify \Cref{ass:LipsAnd} and \Cref{ass:bound}, in addition to the previous conditions on the ODE \eqref{eq:EDOappli} and the solver, we make the following assumption on the prior and noise distributions.
When  \Cref{ass:ODEbounded2} and \Cref{ass:ODEbounded}  are both satisfied, there exist   $\mathtt{C}_{\sbf},\Ltt_\sbf>0$ such that for all $i\in\{1,\ldots,N\}$, the function $\sbf_i$ given for any $x \in \rset^n$ and $\mathbf{A} \in \mathcal{M}_{d,n}(\rset)$, by
\begin{equation}
  \label{eq:def_s_i}
  \sbf_i(x,\mathbf{A}) = \mathbf{A} \nabla_x \ln \varphi_\varepsilon\po y_i - x\pf
\end{equation}
is bounded by $\mathtt{C}_{\sbf}$ and $\Ltt_{\sbf}$-Lipschitz continuous on $\msk$.

\begin{proposition}\label{prop:EDOmain}
Under \Cref{ass:ODE}, \Cref{ass:ODEbounded2},  \Cref{ass:ODEbounded} and \Cref{ass:ODEsolver}, for any $h \in (0,\bar h)$, the functions $\rmT_\gamma$ and $\tilde{\rmT}_{\gamma,h}$ given by \eqref{eq:TtTODE} satisfy for any  \tcrw{$\bar{\gamma}\in\ooint{0,\mtt/\Ltt^2}$}, \Cref{ass:LipsAnd} and \Cref{ass:bound}, with
\begin{equation}
c_\infty = \mathtt{C}_{\Psi} \Ltt_{\sbf}  h^{\alpha}\eqsp,\quad R_1 = \frac{\tcrw{4}N\Ctt_{\sbf}}{\mtt_U}\vee R_U\eqsp,\quad \mtt = \frac{\mtt_U}{\tcrw{4}} \eqsp, \quad \Ltt  = \Ltt_U + \Ltt_{\sbf}\Ltt_F \sum_{i=1}^N t_i \rme^{\Ltt_F't_i} \eqsp. 
\end{equation}
\end{proposition}
\begin{proof}
  The proof is postponed to \Cref{subsec:postponed-application}.
\end{proof}
Under the conditions of \Cref{prop:EDOmain} and using the results of \Cref{sec:main_result_coupling}, we get that the Markov chains $(X_k)_{k\in\N}$ and $(\tilde X_k)_{k\in\N}$ associated to $\rmT_\gamma$ and $\tilde{\rmT}_{\gamma,h}$ given by  \eqref{eq:TtTODE} have unique invariant measure $\pi_\gamma$ and $\tilde \pi_{\gamma,h}$, and that there exist $\bar \gamma,\bar h,C>0$ such that for all $\gamma\in(0,\bar \gamma]$ and $h\in(0,\bar h]$,
 \begin{equation}
   \label{eq:bound_tv_h_ode}
   \tvnorm{\pi_\gamma - \tilde \pi_{\gamma,h}} \ \leq C h \eqsp. 
 \end{equation}
\begin{example}[Continuation of \Cref{exemple:logistique}]
As a conclusion, consider the logistic case of \Cref{exemple:logistique} with the Euler scheme \eqref{eq:EulerEDO}, assuming that $r\in \mathrm{C}^2(\R,[0,+\infty))$ satisfies \eqref{eq:condition_r_logistique}.   Then, by \Cref{exemple:logistique} and  \Cref{exemple:logistique_2}, \Cref{ass:ODE},  \Cref{ass:ODEbounded}  and \Cref{ass:ODEsolver} hold. Assuming moreover that $\pi_0$ and $\varphi_{\varepsilon}$ are Gaussian, then \Cref{ass:ODEbounded2} also holds and we obtain \eqref{eq:bound_tv_h_ode}.
\end{example}

\subsection{Numerical results}
\tcr{Note that we can simulate the reflection coupling (see \cite{jacob2020unbiased,wang2021maximal}) as the sticky process to illustrate our results. However, in this section we only focus on our result Theorem 4, for which there is no need to simulate the reflection coupling and the sticky process.}
More precisely, we illustrate our findings on two particular ODEs. First, we consider the ODE associated with the Van der Pol oscillator corresponding to the second order ODE:
\begin{equation}
  \label{eq:VdP}
  \ddot{x}_{\theta}(t) - \theta(1-x_{\theta}(t)^2)\dot{x}_{\theta}(t) + x_{\theta}(t) =0 \eqsp,
\end{equation}
where $\theta \in \rset$ is the parameter to infer. It corresponds to \eqref{eq:EDOappli} with $f_{\theta}(x_1,x_2) = (x_2, \theta(1-x_1^2)x_2 -x_1)$. We generate synthetic data solving \eqref{eq:VdP} using the 4th\tcr{-order} Runge-Kutta method for $T=10$ and $\theta=1$. We then select $(x_{\theta}(t_i))_{i=1}^{25}$ for $(t_i)_{i=1}^{25}$ uniformly chosen in $\ccint{0,T}$. The observations  $\mathbf{y} = (y_i)_{i=1}^{25}$ are obtained from $(x_{\theta}(t_i))_{i=1}^{25}$ adding \iid~zero-mean Gaussian noise with variance $0.5$. We consider the corresponding statistical model \eqref{eq:stat_model_ODE_I} where $(\varepsilon_i)_{i=1}^{25}$ are \iid~zero-mean Gaussian random variables with variance $0.5$. We consider as prior $\pi_0$, the zero-mean Gaussian distribution with variance $0.5$. We then use ULA with $\gamma = 10^{-2}$, for which the gradient is estimated using the Euler method with the time steps $h \in \{ 0.05,0.01,0.001\}$.
\Cref{fig:VdP_hist} \tcr{illustrates the convergence of} histograms corresponding to the different Markov chains after $10^5$ iterations with a burn-in of  $10^4$ steps. Gaussian kernel density approximation of these histograms are estimated and used as proxy for the density of the invariant distributions $\pi_{\gamma,h}$ of the Markov chain $(\tilde{X}_k)_{k \in\nset}$ associated to $\tilde{\rmT}_{\gamma,h}$ given by  \eqref{eq:TtTODE}. To obtain a proxy for the density of $\pi_{\gamma}$, the stationary distribution of $(X_k)_{k \in\nset}$ associated to $\rmT_{\gamma}$, we use the same procedure but using the Euler method with $h = 0.0001$. We then estimate the total variation between $\pi_{\gamma}$ and $\pi_{\gamma,h}$ for $h \in \{0.005,0.004,0.003,0.0025,0.0015,0.001,0.00075,0.0005\}$ using numerical integration. The corresponding results over $10$ replications are reported in \Cref{fig:VdP_tv}. We can observe that the total variation distance linearly decreases with $h$ which supports our findings. 

\tcr{Note that we use time steps smaller than 0.005 to ensure that we are operating in the asymptotic regime $\gamma \to 0$ and obtain a clear linear relation between the time step $h$ and the estimated total variation. We employed larger time steps in \Cref{fig:VdP_hist} to ensure that the histograms are not too closely packed for visibility.}
  \begin{figure}
    \centering
    \begin{subfigure}{0.48\columnwidth}
      \includegraphics[width=\linewidth]{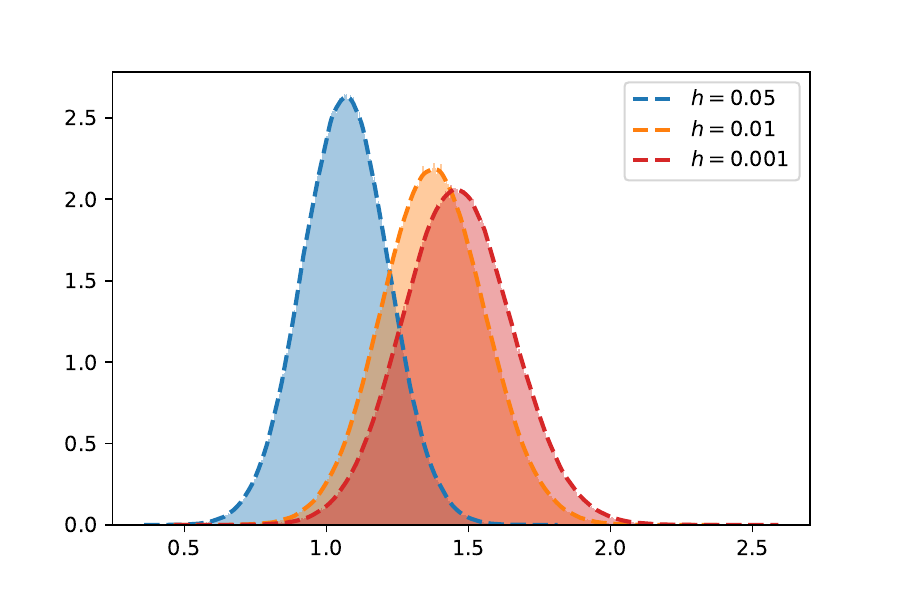}
\caption{Empirical histograms and corresponding KDE for different time steps $h$} \label{fig:VdP_hist}
\end{subfigure}
\hfill
    \begin{subfigure}{0.48\columnwidth}
      \includegraphics[width=\linewidth]{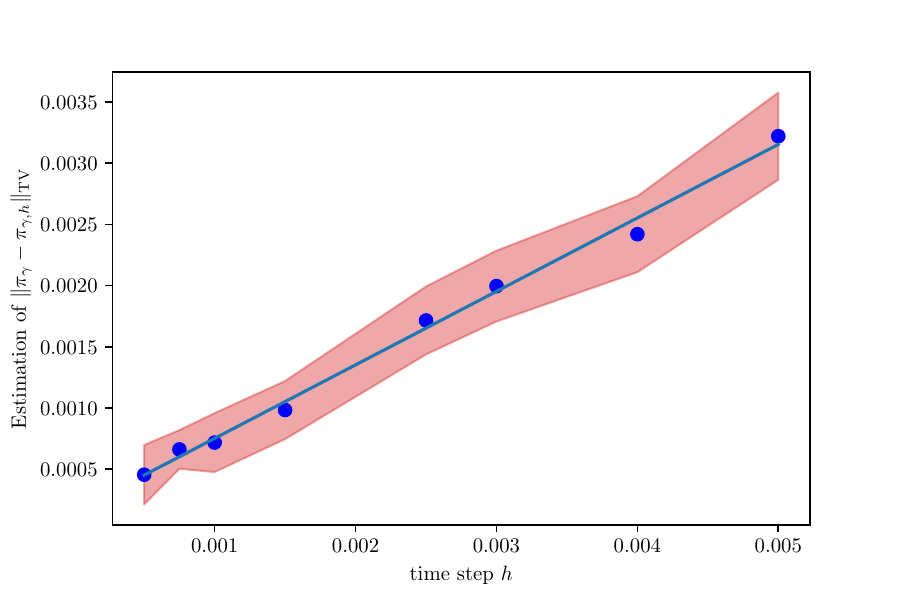}
      \vspace{1pt}
      \caption{Numerical estimation of $\tvnorm{\pi_{\gamma} - \pi_{\gamma,h}}$}
      \label{fig:VdP_tv}
    \end{subfigure}
    \caption{Numerical illustrations for the Van der Pol oscillator \eqref{eq:VdP}}
  \label{fig:VdP}
  \end{figure}

For our second experiment, we consider the Lotka-Volterra model
describing the evolution of the population of two interacting
biological species denoted by $t \mapsto x_{\theta}(t)  = (u_{\theta}(t),v_{\theta}(t))$. The dynamics of
these two populations are assumed to be governed by the system of
equations given by:
\begin{equation}
\label{eq:lv_system}
	\dot{ u}_{\theta}(t) = (\upalpha -\upbeta v_{\theta}(t))u_{\theta}(t) \eqsp,  \qquad
	\dot{v}_{\theta}(t) = (-\upgamma + \updelta u_{\theta}(t))v_{\theta}(t) \eqsp,
%u(0) = u_{0},\ v(0) = v_{0}.
      \end{equation}
      where $\theta = (\upalpha,\upbeta,\upgamma,\updelta)$ is the
      parameter to infer.  We follow the same methodology presented in
      \cite[Chapter~16]{mcelreath2020statistical}.  For this
      experiment, we consider an other statistical model as previously
      and generate synthetic data $\bfy = (y_i)_{i=1}^{50}$
      accordingly and associated with observation times
      $(t_i)_{i=1}^{50}$ uniformly spaced in $\ccint{0,T}$ for $T=10$
      and the true parameter $\theta_0= (0.6,0.025,0.8,0.025)$. More
      precisely, for any $i \in\{1,\ldots,50\}$, $y_i = (u_i^y,v_i^y)$
      with $u_i^y = u_{\theta}(t_i) \rme^{\varepsilon_{u,i} }$,
      $v_i^y = v_{\theta}(t_i) \rme^{\varepsilon_{v,i} }$ and
      $(\varepsilon_{u,j},\varepsilon_{v,j})_{j=1}^{50}$ are \iid~one
      dimensional zero-mean Gaussian random variables with covariance matrix $\mathrm{I}_2$. The prior $\pi_0$ is set to be the Gaussian
      distribution on $\rset^4$ with means $(1,0.05,1,0.05)$ and
      standard deviations $(0.5,0.05,0.5,0.05)$. The posterior
      distribution is then given by
      $\pi(\theta | \ybf ) \propto \exp(-U(\theta))$, where
\begin{equation}
 U(\theta) =  -\log \pi_0(\theta) +  \sum_{i=1}^{50} \parenthese{ \frac{(\log{u_i^y} - \log{u_{\theta}(t_i)})^2 + (\log{v_i^y} - \log{v_{\theta}(t_i)})^2}{2 \varsigma^2}  } \eqsp.
\end{equation}
We then use ULA with $\gamma = 5\times 10^{-5}$, for which the gradient is estimated using the Euler method. We focus here on the second component of the chain. The results for the other components are similar.  
\Cref{fig:LV_hist} represents the histograms for the second component corresponding to the different Markov chains after $10^7$ iterations with a burn-in of  $10^3$ steps. Gaussian kernel density approximation of these histograms are estimated and used as proxy for the marginal density of the invariant distributions $\pi_{\gamma,h}$ of the Markov chain $(\tilde{X}_k)_{k \in\nset}$ associated to $\tilde{\rmT}_{\gamma,h}$ given by  \eqref{eq:TtTODE}. To obtain a proxy for the marginal density of $\pi_{\gamma}$, the stationary distribution of $(X_k)_{k \in\nset}$ associated to $\rmT_{\gamma}$, we use the same procedure but using the Euler method with $h = 0.0001$. We then estimate the total variation between $\pi_{\gamma}$ and $\pi_{\gamma,h}$ for $h \in \{k \times 10^{-2} \,: \, k \in \{1,\ldots,10\}\}$ using numerical integration. The corresponding results over $10$ replications are reported in \Cref{fig:LV_tv}. We can observe that the total variation distance still linearly decreases with $h$. 
  \begin{figure}
    \centering
    \begin{subfigure}{0.48\columnwidth}
      \includegraphics[width=\linewidth]{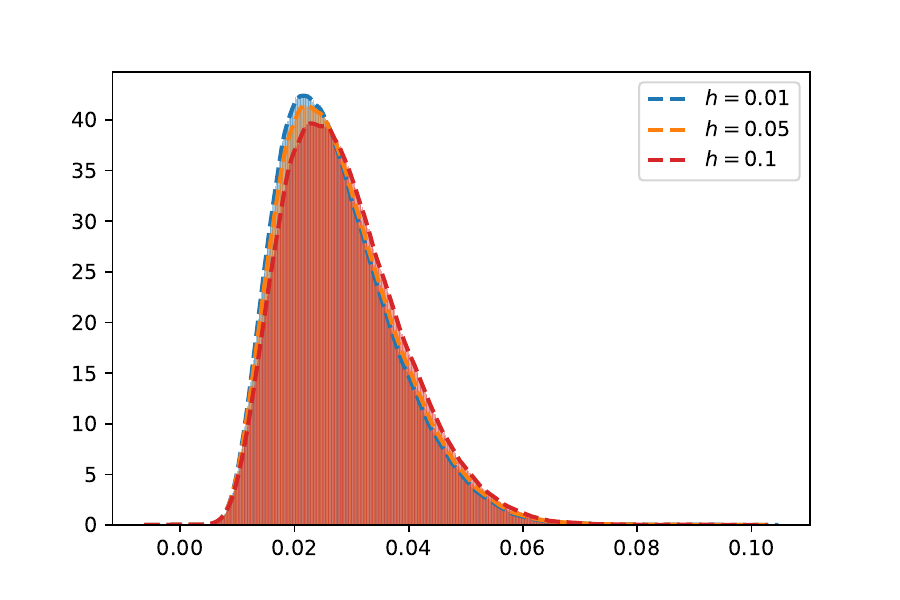}
\caption{Empirical histograms and corresponding KDE for different time steps $h$} \label{fig:LV_hist}
\end{subfigure}
\hfill
    \begin{subfigure}{0.48\columnwidth}
      \includegraphics[width=\linewidth]{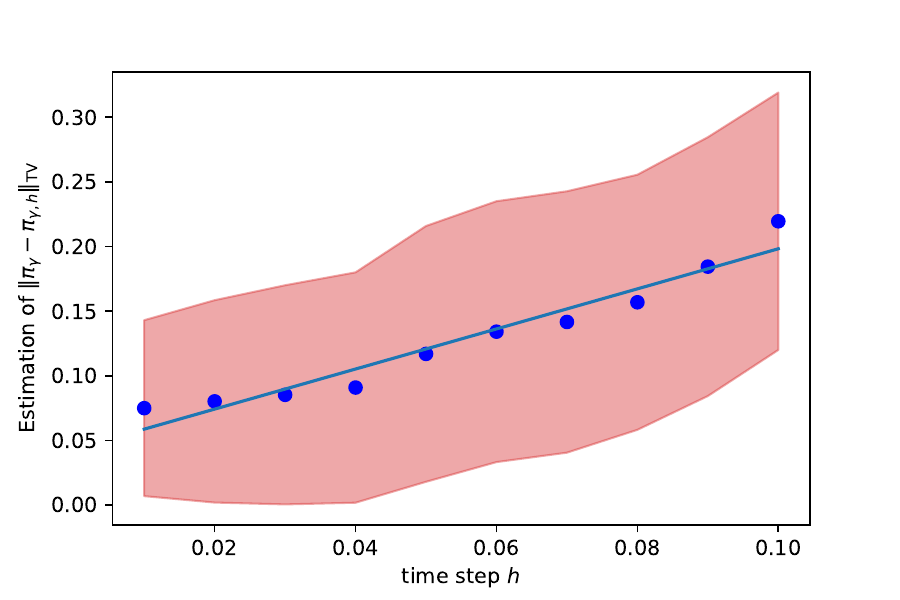}
      \vspace{1pt}
      \caption{Numerical estimation of $\tvnorm{\pi_{\gamma} - \pi_{\gamma,h}}$}
      \label{fig:LV_tv}
    \end{subfigure}
    \caption{Numerical illustrations for the Lotka-Volterra model  \eqref{eq:lv_system}}
  \label{fig:LV}
  \end{figure}

  %%% Local Variables:
%%% mode: latex
%%% TeX-master: "main_imsart"
%%% End:

\section{Postponed proofs}
\label{sec:Postponed_proofs}

\subsection{Proof of \Cref{lem:G_inequality}}
\label{sec:proof:lem:G_inequality}
The proof is based on this technical lemma.
\begin{lemma}
  \label{lem:H_increase}
  For any $g \in \rset$, $u\in\ccint{0,1}$, $\gamma \in \ocint{0,\bgamma}$ \tcwc{and $c\in \coint{0,\plusinfty}$,
  \begin{equation}
    \label{eq:H_positive}
    \mathscr{H}_\gamma\parentheseLigne{c,g,u}\geq 0 \eqsp ,
  \end{equation}
  in addition, for any} $a,b \in \coint{0,\plusinfty}$, $a\leq b$,
  \begin{equation}
  \label{eq:H_increase}
      \mathscr{H}_\gamma\parentheseLigne{a,g,u} \leq   \mathscr{H}_\gamma\parentheseLigne{b,g,u} \eqsp .
  \end{equation}
\end{lemma}
\begin{proof}
Let $u\in \ccint{0,1}$, $g\in \R$ and $a,b \in \coint{0,\plusinfty}$, $a \leq b$. We first prove that for any $c\in \R_+$ such that $c - 2(\sigma^2\gamma)^{\half}g<0$, 
\begin{equation}
\label{eq:H_pos}
\mathscr{H}_\gamma\parentheseLigne{c,g,u}= 0 \eqsp,
\end{equation}
which implies that for any $c\in \R_+$, $\mathscr{H}_\gamma\parentheseLigne{c,g,u}\geq 0$.

\tcwc{If $c - 2(\sigma^2\gamma)^{\half}g<0$ since $c\geq 0$, we have  $0\leq c < 2(\sigma^2\gamma)^{\half}g$. This implies that $-(\sigma^2\gamma)^{\half}g\leq c -(\sigma^2\gamma)^{\half}g< (\sigma^2\gamma)^{\half}g$ which gives $\varphibf_{\sigma^2\gamma}(c -(\sigma^2\gamma)^{\half}g)\geq \varphibf_{\sigma^2\gamma}((\sigma^2\gamma)^{\half}g)$. Finally, $\overline{p}_{\gamma}(c,g)=1$ and therefore \eqref{eq:H_pos} holds.}

We now show \eqref{eq:H_increase}.
It is straightforward by \eqref{eq:H_pos} if $0> a - 2(\sigma^2\gamma)^{\half}g$.
If $0\leq a - 2(\sigma^2\gamma)^{\half}g$. 
By using $t \mapsto \varphibf_{\sigma^2\gamma}(t)$ is decreasing on $[0,+\infty)$, we obtain
\begin{equation}
\label{eq:proba_Inequality}
\bpg(a,g)\geq \bpg(b,g) \eqsp .
\end{equation}
Then, \eqref{eq:H_increase} follows from  \eqref{eq:def_H} and \eqref{eq:proba_Inequality}.

\end{proof}
\begin{proof}[Proof of \Cref{lem:G_inequality}]
Let $k\in \mathbb{N}$ .
By \Cref{ass:LipsAnd}, \Cref{ass:bound} and the  triangle inequality, for any $x,\tilde{x}\in \R^d$,
\begin{align}
\label{eq:E_inequality}
\norm{\rmE(x,\tilde{x})}\leq \tau_{\gamma}(\norm{x-\tilde{x}})+\gamma \InftyBound \eqsp .
\end{align}

By using \eqref{eq:def_refl_discret}, and \eqref{eq:def_H} we have,
\begin{align}
\normLigne{X_{k+1}-\tilde{X}_{k+1}}
&=(1-B_{k+1})\normLigne{-\rmE(X_k,\tilde{X}_k)+2(\sigma^2\gamma)^{\half}e(X_k,\tilde{X}_k)e(X_k,\tilde{X}_k)^{\transpose}Z_{k+1}}\\
&=(1-B_{k+1})\normLigne{-\normLigne{\rmE(X_k,\tilde{X}_k)}e(X_k,\tilde{X}_k)+2(\sigma^2\gamma)^{\half}G_{k+1} e(X_k,\tilde{X}_k)}\\
&=(1-B_{k+1})\abs{\normLigne{\rmE(X_k,\tilde{X}_k)}-2(\sigma^2\gamma)^{\half}G_{k+1}}\\
%&=\abs{\mathscr{H}_\gamma(\normLigne{\rmE(X_k,\tilde{X}_k)},G_{k+1},U_{k+1})}\\
&=\mathscr{H}_\gamma(\normLigne{\rmE(X_k,\tilde{X}_k)},G_{k+1},U_{k+1})\eqsp .
\end{align}
This gives \eqref{eq:lemma_comparison_discrete} when combined with \eqref{eq:E_inequality}.
Finally, the last statement follows from \Cref{lem:H_increase} and \Cref{ass:LipsAnd} ensuring that $\tau_{\gamma}$ is non-decreasing on $[0,+\infty)$.
\end{proof}

\subsection{Proof of \Cref{prop:drift_v_norm}}
\label{sec:proof_prop:drift_v_norm}

The proof is an easy consequence of this technical lemma.
\begin{lemma}
  \label{lem:norm_1_eq}
For any $w \in \coint{0,\plusinfty}$ and $\gamma \in \ocint{0,\bgamma}$, we have that
  \begin{equation}
    \label{eq:7}
    Q_\gamma \VlyapDs_1(w) = \tau_{\gamma}(w) + \gamma \InftyBound \eqsp. 
  \end{equation}
\end{lemma}
\begin{proof}
For any $w\in\R$, we have,
\begin{align}
&Q_\gamma \VlyapDs_1(w)= \int_\R (1-\bpg(\tau_{\gamma}(w) + \gamma \InftyBound,g))\parenthese{\tau_{\gamma}(w) + \gamma \InftyBound-2(\sigma^2\gamma)^{\half}g}\varphibf(g)\rmd g\\
&=\int_\R (\tau_{\gamma}(w) + \gamma \InftyBound-2g)\parenthese{\varphibf_{(\sigma^2\gamma)^{\half}}(g)-\varphibf_{(\sigma^2\gamma)^{\half}}(g)\wedge \varphibf_{(\sigma^2\gamma)^{\half}}(\tau_{\gamma}(w) + \gamma \InftyBound-g)}\rmd g\\
&=\int_{-\infty}^{(\tau_{\gamma}(w) + \gamma \InftyBound)/2} (\tau_{\gamma}(w) + \gamma \InftyBound-2g)\parenthese{\varphibf_{(\sigma^2\gamma)^{\half}}(g)- \varphibf_{(\sigma^2\gamma)^{\half}}(\tau_{\gamma}(w) + \gamma \InftyBound-g)}\rmd g \eqsp. 
\end{align}
By using change of variable $g\to\tcwc{\tau_{\gamma}(w) + \gamma \InftyBound}-g$ we have,
\begin{align}
&\int_{-\infty}^{(\tau_{\gamma}(w) + \gamma \InftyBound)/2} (\tau_{\gamma}(w) + \gamma \InftyBound-2g)\parenthese{\varphibf_{(\sigma^2\gamma)^{\half}}(g)- \varphibf_{(\sigma^2\gamma)^{\half}}(\tau_{\gamma}(w) + \gamma \InftyBound-g)}\rmd g\\
&\qquad\qquad=\frac{1}{2}\int_\R (\tau_{\gamma}(w) + \gamma \InftyBound-2g)\parenthese{\varphibf_{(\sigma^2\gamma)^{\half}}(g)- \varphibf_{(\sigma^2\gamma)^{\half}}(\tau_{\gamma}(w) + \gamma \InftyBound-g)}\rmd g\\
&\qquad\qquad=\tau_{\gamma}(w) + \gamma \InftyBound \eqsp .
\end{align}
\end{proof}

\begin{proof}[Proof of \Cref{prop:drift_v_norm}]
By \Cref{lem:norm_1_eq} and \Cref{ass:LipsAnd}-\ref{ass:LipsAnd_2}, for any $w \in \coint{0,\plusinfty}$,
\begin{align}
Q_\gamma \VlyapDs_1(w)&=\tau_{\gamma}(w)+\gamma c_\infty\\
&\leq (1-\gamma\mtt)\VlyapDs_1(w)\1_{(R_1,\infty)}(w)+(1+\gamma\Ltt)\VlyapDs_1(w)\1_{(0,R_1]}(w)+\gamma \InftyBound \eqsp .
\end{align}
This completes the proof.  
\end{proof}

\subsection{Proof of \Cref{propo:existence_mes_statio}}
\label{sec:proof_existence_mes_statio}

  We first establish that $Q_\gamma$ admits a unique invariant
  probability measure $\mu_{\gamma}$ and is geometrically ergodic.  To
  that end, we show that $Q_{\gamma}$ is \ref{item:proof_prop_9_irreducible_aperiodique} irreducible and
  aperiodic, \ref{item:proof_prop_9_compact_small} any compact set of $[0,+\infty)$ is small and \ref{item:proof_prop_9_drift} there
  exists $\tcwc{1>}\lambda >0$ and $b \geq 0$ such that
  $Q_{\gamma} \VlyapD(w) \leq \lambda \VlyapD(w) + b$ for any $w \in \coint{0,\plusinfty}$ with
  $\VlyapD(w) =w+1$.  The proof then follows from \cite[Corollary 14.1.6,
  Theorem 15.2.4]{douc:moulines:priouret:soulier:2018}.

  \begin{enumerate}[align=left, leftmargin=*, labelindent=\parindent, listparindent=\parindent, labelwidth=0pt, itemindent=!, labelindent=0pt,label=(\alph*),noitemsep,nolistsep]
  \item  \label{item:proof_prop_9_irreducible_aperiodique} Let $\msk$ be a compact set.  Then for any
  $w\in \msk$ we have
\begin{equation}
\label{eq:bound_q_gamma_eta_msk}
  Q_\gamma(w,\{0\}) \geq  \int_{\ccint{-1,1}} \bpg(\tau_{\gamma}(w) + \gamma \InftyBound , g)\varphibf(g)  \rmd g \geq  \eta_{\msk}\eqsp ,
\end{equation}
where using \Cref{ass:LipsAnd}-\ref{ass:LipsAnd_2}
\begin{multline}
  \eta_{\msk} =  \inf_{(r,g)\in \msk\times\ccint{-1,1}}\bpg(\tau_{\gamma}(\tcwc{r})+\gamma \InftyBound,g) \int_{\ccint{-1,1}} \varphibf(g) \rmd g \\
\geq \inf_{(a,g)\in \ccint{0,M} \times\ccint{-1,1}}\bpg(a,g)
  \int_{\ccint{-1,1}} \varphibf(g) \rmd g \eqsp,
\end{multline}
and $M = (1+\gamma\Ltt) \sup( \msk ) +\gamma \InftyBound$.  Note that
since $(a,g)\to \bpg(a,g)$ is a continuous positive function, and
$\ccint{0,M}\times\ccint{-1,1}$ is compact, $\eta_{\msk} >0$.
Therefore $\{0\}$ is an accessible $(1, \updelta_{0})$-small set
and $Q_\gamma$ is irreducible. In addition, $Q_{\gamma}(0,\{0\}) >0$
which implies that $Q_{\gamma}$ is strongly aperiodic. 

\item \label{item:proof_prop_9_compact_small} Let now $\msc$ be a compact set, we show that $\msc$ is
small. By \eqref{eq:bound_q_gamma_eta_msk}, for 
  $\msa\in \mathcal{B}([0,+\infty) )$ and $w \in \coint{0,\plusinfty}$,
\begin{align}
  Q_\gamma^2(w,\msa)\geq\int_\rset \1_{\{0\}}(\tw)Q_\gamma(\tw,\msa)Q_\gamma(w,\rmd \tw)
  \geq \eta_{\msc} Q_\gamma(0,\msa) \eqsp .
\end{align}
 Therefore $\msc$ is a $(2,Q_\gamma(0,\cdot))$-small set.
\item \label{item:proof_prop_9_drift} In addition, by \tcwc{\Cref{prop:drift_v_norm}} we
  have, for any $w \in \coint{0,\plusinfty}$,
\tcwc{\begin{align}
   Q_{\gamma}\VlyapD(w) &\leq 1+(1-\gamma\mtt)w\1_{(R_1,+\infty)}(w)+(1+\gamma\Ltt)w\1_{(0,R_1]}(w)+\gamma \InftyBound\\
                                             & \leq (1-\gamma\mtt)\VlyapD(w)+\gamma R_1(\mtt+ \Ltt)+\gamma \InftyBound+\gamma \mtt \eqsp.
\end{align}}
\end{enumerate}
The proof of the first part of the proposition is complete.

We now establish the second part. 
Let $\msa\in \mathcal{B}(\rset)$ such that $(\updelta_{0} + \Leb)(\msa)=0$. Then $0\not\in\msa$ and $\Leb(\msa)=0$ therefore for any $w\in\rset$,
\begin{equation}
 Q_{\gamma}(w,\msa) \leq \frac{1}{\sqrt{2\uppi}}\int_{\rset} \1_{\msa}\parenthese{\tau_{\gamma}(w)+\gamma \InftyBound-2 \sigma \gamma^{1/2} g}    \rmd g=0\eqsp .
\end{equation}
It follows that $\mu_\gamma(\msa)=\mu_\gamma Q_\gamma(\msa)=0$ and  $\mu_\gamma\ll(\updelta_{0} + \Leb)$. 

Since for any $w \in \coint{0,\plusinfty}$, $Q_{\gamma}(w,\{0\}) >0$, $\updelta_{\{0\}}$ is an irreducibility measure, and by \cite[Theorem 9.2.15]{douc:moulines:priouret:soulier:2018}, $\mu_\gamma$ is a maximal irreducibility measure for $Q_\gamma$, $\updelta_{0}\ll\mu_\gamma$ implying that $\mu_{\gamma}(\{0\}) >0$.

In the case $\InftyBound \neq 0$, \tcwc{first remark that for any $w \in \coint{0,\plusinfty}$ and $g\in \rset$, if $\tau_{\gamma}(w)   +\gamma \InftyBound-2 \sigma \gamma^{1/2} g>0$, then $\absLigne{\tau_{\gamma}(w)   +\gamma \InftyBound- \sigma \gamma^{1/2} g}>\absLigne{\sigma \gamma^{1/2} g}$ therefore $\bpg\parenthese{\tau_{\gamma}(w) + \gamma \InftyBound , g}<1$ by \eqref{eq:def_bpg}. Then by \eqref{def_q_gamma}, for any $\msa \in \mcb{[0,+\infty)}$, $\Leb(\msa) >0$ and $w \in \coint{0,\plusinfty}$, $Q_{\gamma}(w,\msa)\geq Q_{\gamma}(w,\msa\setminus\{0\}) >0$} and therefore $\Leb$ is an irreducibility measure. Applying  \cite[Theorem 9.2.15]{douc:moulines:priouret:soulier:2018} again, we get that $\updelta_{0} + \Leb\ll\mu_\gamma$. This completes the proof since we have already shown that $\mu_\gamma\ll(\updelta_{0} + \Leb)$.

\subsection{Proof of \Cref{theo:bound_moment_rset_star_invariant_mes}}
\label{sec:proof-crefth}

All technical results are moved to \Cref{sec:technical-results-1}.
The main idea is to analyze how the mass of $\ooint{0,\plusinfty}$ evolves with the sequence of distributions defined by $Q^n(w,\cdot)$ for $w \in\coint{0,\plusinfty}$. To do this, we first establish the following drift condition.
Define
$(\alpha_k)_{k \geq 1}$, $(\beta_k)_{k \geq 1}$ for any
$k \geq 1$ by
\begin{equation}
  \label{eq:def_alpha_beta}
%&D_1=1 \eqsp , \qquad D_{l+1}=-\sum_{k=1}^l \parentheseDeux{1-2\Phibf\parenthese{-\frac{\alpha_k}{2\beta_k}}}D_{l-k+1}\\
%&C_1=0 \eqsp , \qquad C_{l+1}=C_l -D_{l+1} \\
%&\alpha_1=\gamma \InftyBound \eqsp , \qquad \alpha_{k+1}=\gamma \InftyBound+\frac{\alpha_k}{1+\gamma \Ltt}\\
%&\beta_1=\sigma\sqrt{\gamma}\eqsp , \qquad \beta_{k+1}=\sqrt{\sigma^2\gamma+\parenthese{\frac{\beta_k}{1+\gamma\Ltt}}^2}\\
%\label{eq:def_alpha}
%&\alpha_k=\frac{1}{(1+\gamma\Ltt)^{k-1}}\parenthese{\gamma\InftyBound-\frac{\InftyBound(1+\gamma\Ltt)}{L}}+\frac{\InftyBound(1+\gamma\Ltt)}{L}\\
%&\beta^2_k=\frac{1}{(1+\gamma\Ltt)^{2(k-1)}}\parenthese{\sigma^2\gamma-\frac{\sigma^2(1+\gamma\Ltt)^2}{L(2+\gamma L)}}+\frac{\sigma^2(1+\gamma\Ltt)^2}{L(2+\gamma L)} \label{eq:def_beta}\\
\alpha_k=\gamma\InftyBound\sum_{i=0}^{k-1}(1+\gamma\Ltt)^{-i} \eqsp, \qquad 
\beta^2_k=\gamma\sigma^2\sum_{i=0}^{k-1} (1+\gamma\Ltt)^{-2i} \eqsp. 
\end{equation}

\begin{lemma}
\label{lem:post_induction}
Assume \Cref{ass:LipsAnd}-\ref{ass:LipsAnd_2}. Let $\bdelta \in \ocint{0,\{ \Ltt^{-1} \wedge (\sigma\rme^{-1}/\InftyBound)^2\}}$, with the convention $1/0=\plusinfty$.  For any
  $\gamma \in \ocint{0,\bgamma}$, $n\in \{0,\ldots,n_{\gamma}\}$, $n_{\gamma} = \floorLigne{\bdelta/\gamma}$, and $w\in \coint{0,\plusinfty}$, it holds 
\begin{equation}
\label{eq:induction_step}
\int_{[0,+\infty)} \1_{(0,+\infty) }(\tilde{w})
Q^{n+1}_\gamma(w,\rmd \tilde{w})\leq 1-2\Phibf\parenthese{-\frac{\tau_\gamma(w)+\alpha_{n+1}}{2\beta_{n+1}}}+\zeta\sum_{k=1}^n \frac{\gamma\alpha_k}{\beta_k^3} \eqsp,
\end{equation}
where $(\alpha_k)_{k\geq 1},(\beta_k)_{k\geq 1}$ are defined in
\eqref{eq:def_alpha_beta},  where $\Phibf$ are the density and the cumulative
distribution function of the one-dimensional Gaussian distribution
with mean $0$ and variance $1$ and
\begin{equation}
  \label{eq:def_zeta}
  \zeta = \tcwc{2}\textstyle{(1+\bgamma \Ltt)^2\sigma^2(2\sqrt{2\uppi})^{-1}\parentheseDeux{\sup_{t \geq 0} \{t^2\Phibf(-t)\}+ 1/8}} \eqsp. 
\end{equation}
\end{lemma}

\begin{proof}
Let $\gamma \in\ocint{0,\bgamma}$, $w \in \coint{0,\plusinfty}$.   Note that for any $n\in \{0,\ldots,n_{\gamma}\}$, by
\eqref{eq:def_alpha_beta} and \Cref{lem:bound_alpha_beta},
\begin{equation}
  \label{eq:bound_alpha_beta_proof_induc}
  \alpha_n/(2\beta_n) \leq 1 \eqsp. 
\end{equation}
Then, by \Cref{lem:0_proba}, \eqref{eq:induction_step} holds for $n=0$. Assume it holds for $n \in \{0,\ldots,n_{\gamma}-1\}$. Then, we get
\begin{align}
%\label{eq:induction_step}
\int_{[0,+\infty)} \1_{(0,+\infty)}(\tilde{w})
Q^{n+1}_\gamma(w,\rmd \tilde{w})&\leq \int_{[0,+\infty)}\parentheseDeux{1-2\Phibf\parenthese{-\frac{\tau_\gamma(w)+\alpha_{n}}{2\beta_{n}}}} Q_{\gamma}(w,\tilde{w})\\
&\qquad\qquad\qquad\qquad\qquad\qquad\qquad\qquad+\zeta\sum_{k=1}^{n-1} \frac{\gamma\alpha_k}{\beta_k^3} \eqsp.
\end{align}
The proof is then concluded by a straightforward induction using \Cref{lem:technique_lemma_induction} and \eqref{eq:bound_alpha_beta_proof_induc}. 
\end{proof}

From \Cref{lem:post_induction}, we can have the following bound on the mass $\mu_{\gamma}$ at $\ooint{0,R}$ for $R \geq  0$.

\begin{theorem}
  \label{theo:bound_mu_0_R_1}
Assume \Cref{ass:LipsAnd}-\ref{ass:LipsAnd_2}. Let $R \geq 0$. 
For any $\gamma\in \ocint{0,\bgamma}$ and $\bdelta \in \ocint{0,\{ \Ltt^{-1} \wedge (\sigma\rme^{-1}/\InftyBound)^2\}}$, $\mu_{\gamma}(\ooint{0,R}) \leq \eta_R  \InftyBound$, where
\begin{equation+}
  \label{eq:def_eta_R}
    \eta_R=\left. [\deltaInf+\bgamma]^{1/2}\parentheseDeux{\frac{\tcwc{2}\zeta\rme^{3(\deltaInf+\bgamma)\Ltt}}{\sigma^3}+\frac{\rme^{(\deltaInf+\bgamma)\Ltt}}{2\sqrt{2\uppi}\sigma}}\middle/
\Phibf\parenthese{-\frac{(1+\bgamma\Ltt)R+(\deltaInf+\bgamma)\InftyBound}{2\deltaInf^{1/2}\sigma\rme^{-(\deltaInf+\bgamma) \Ltt}}}
        \right.\eqsp .
\end{equation+}
\end{theorem}
\begin{proof}
  Let $\bdelta \in \ocint{0,\{ \Ltt^{-1} \wedge (\sigma\rme^{-1}/\InftyBound)^2\}}$ and  $\gamma \in \ocint{0,\bgamma}$. Set $n_{\gamma} = \floorLigne{\deltaInf/\gamma}$. Note that $\deltaInf \leq \gamma(n_{\gamma}+1) \leq \deltaInf +\bgamma$. 
  By \Cref{lem:post_induction}, \Cref{propo:existence_mes_statio}, integrating \eqref{eq:induction_step} with respect to $\mu_{\gamma}$ and using that $\tau_{\gamma}(0) = 0$, $\Phibf(-t) \leq 1/2$ for any $t \geq 0$, gives
  \begin{align}
      \label{eq:end_proof_mass_0_0}
    & \qquad \mu_\gamma((0,+\infty))\leq \int_\rset \defEns{1-2\Phibf\parenthese{-\frac{\tau_\gamma(w)+\alpha_{n_\gamma+1}}{2\beta_{n_\gamma+1}}} }\rmd \mu_{\gamma}( w)+\zeta \sum_{k=1}^{n_\gamma} \frac{\gamma\alpha_k}{\beta_k^3}\\
    &\leq 1-2\Phibf\parenthese{-\frac{\alpha_{n_\gamma+1}}{2\beta_{n_\gamma+1}}}+\zeta \sum_{k=1}^{n_\gamma} \frac{\gamma\alpha_k}{\beta_k^3}\\
    & \qquad \qquad + 2\int_{(0,+\infty)}\defEns{\Phibf\parenthese{-\frac{\alpha_{n_\gamma+1}}{2\beta_{n_\gamma+1}}}-\Phibf\parenthese{-\frac{\tau_\gamma(w)+\alpha_{n_\gamma+1}}{2\beta_{n_\gamma+1}}}} \rmd \mu_{\gamma}( w) 
    \\
     \label{eq:end_proof_mass_0}
&\leq 1-2\Phibf\parenthese{-\frac{\alpha_{n_\gamma+1}}{2\beta_{n_\gamma+1}}}+\mu_\gamma((0,+\infty))+\zeta \sum_{k=1}^{n_\gamma} \frac{\gamma\alpha_k}{\beta_k^3}\\
&\qquad \qquad-2\int_{\ooint{0,\tcwc{R}}}\Phibf\parenthese{-\frac{\tau_\gamma(w)+\alpha_{n_\gamma+1}}{2\beta_{n_\gamma+1}}} \rmd \mu_{\gamma}( w) \eqsp .
% &\qquad\qquad\qquad\qquad\qquad\qquad\qquad+2\int_{\ooint{0,R_1}}\Phibf\parenthese{-\frac{\tau_\gamma(w)+\alpha_{n_\gamma+1}}{2\beta_{n_\gamma+1}}} \rmd \mu_{\gamma}( w) \eqsp .
\end{align}
Rearranging the terms yields
\begin{equation}
\label{eq:end_proof_mass_0_1}
2\int_{\ooint{0,\tcwc{R}}} \Phibf\parenthese{-\frac{\tau_\gamma(w)+\alpha_{n_\gamma+1}}{2\beta_{n_\gamma+1}}}\rmd \mu_{\gamma}( w)\leq 1-2\Phibf\parenthese{-\frac{\alpha_{n_\gamma+1}}{2\beta_{n_\gamma+1}}}+\zeta \sum_{k=1}^{n_\gamma} \frac{\gamma\alpha_k}{\beta_k^3} \eqsp .
\end{equation}
In addition, by \Cref{ass:LipsAnd}-\ref{ass:LipsAnd_2} using \Cref{lem:bound_alpha_beta} and $t \mapsto \Phibf(-t)$ is decreasing on $\rset$, we have
\begin{align}
 \int_{\ooint{0,\tcwc{R}}} \Phibf\parenthese{-\frac{\tau_\gamma(w)+\alpha_{n_\gamma+1}}{2\beta_{n_\gamma+1}}}\rmd \mu_{\gamma}( w)
&\geq
  \Phibf\parenthese{-\frac{(1+\gamma\Ltt)\tcwc{R}+\alpha_{n_\gamma+1}}{2\beta_{n_\gamma+1}}}\mu_\gamma((0,\tcwc{R})) \\
  \label{eq:end_proof_mass_0_2}
&\geq \Phibf\parenthese{-\frac{(1+\bgamma\Ltt)\tcwc{R}+(\deltaInf+\bgamma)\InftyBound}{2\deltaInf^{1/2}\sigma\rme^{-(\deltaInf+\bgamma) \Ltt}}}\mu_\gamma((0,\tcwc{R}))\eqsp,
\end{align}
Using that $t \mapsto 1-2\Phibf(-t)$ is $\sqrt{2/\uppi}$-Lipschitz and combining \eqref{eq:end_proof_mass_0_1}, \eqref{eq:end_proof_mass_0_2} and  \Cref{lem:bound_alpha_beta} we get that
\begin{multline}
  \label{eq:bound_mu_gamma_0_r_One}
\Phibf\parenthese{-\frac{(1+\bgamma\Ltt)\tcwc{R}+(\deltaInf+\bgamma)\InftyBound}{2\deltaInf^{1/2}\sigma\rme^{-(\deltaInf+\bgamma) \Ltt}}}\mu_\gamma((0,\tcwc{R})) \leq \alpha_{n_\gamma+1}/({2\sqrt{2\uppi}\beta_{n_\gamma+1}})+ \zeta \gamma \sum_{k=1}^{n_\gamma} \{\alpha_k/\beta_k^3\} \\
  \leq \InftyBound [\deltaInf+\bgamma]^{1/2}\parentheseDeux{\frac{\tcwc{2}\zeta\rme^{3(\deltaInf+\bgamma)\Ltt}}{\sigma^3}+\frac{\rme^{(\deltaInf+\bgamma)\Ltt}}{2\sqrt{2\uppi}\sigma}} \eqsp,
\end{multline}
which implies that $\mu_\gamma((0,\tcwc{R})) \leq \tcwc{\eta_R} \InftyBound$ and completes the proof. 
\end{proof}

Now we can easily complete the proof of \Cref{theo:bound_moment_rset_star_invariant_mes}.

\begin{proof}[Proof of \Cref{theo:bound_moment_rset_star_invariant_mes}]
  Let $\bdelta \in \ocint{0,\{ \Ltt^{-1} \wedge (\sigma\rme^{-1}/\InftyBound)^2\}}$ and  $\gamma \in \ocint{0,\bgamma}$. By \Cref{prop:drift_v_norm} and using $\mu_{\gamma}$ is invariant for $Q_{\gamma}$, we obtain 
\begin{equation}
\int_\rset w \eqsp \rmd \mu_{\gamma}( w)\leq (1-\gamma\mtt)\int_{\coint{R_1,\plusinfty}} w \eqsp \rmd \mu_{\gamma}( w) +(1+\gamma\Ltt)\int_{\ooint{0,R_1}} w \eqsp \rmd \mu_{\gamma}( w) +\gamma\InftyBound \eqsp.
\end{equation}
Then, rearranging the terms in this inequality yields 
\begin{align}
\int_{\coint{R_1,\plusinfty}} w \eqsp \rmd \mu_{\gamma}( w) & \leq R_1\mu_\gamma((0,R_1))\Ltt/\mtt+ \InftyBound/\mtt \\
  \int_{[0,+\infty)} w \eqsp \rmd \mu_{\gamma}( w)& \leq R_1\mu_\gamma((0,R_1))(1+\Ltt/\mtt)+ \InftyBound/\mtt \eqsp,
\end{align}
which, combined with \Cref{theo:bound_mu_0_R_1} applied to $R\leftarrow R_1$,  concludes the proof of the first inequality in
\eqref{eq:bound_moment_rset_star_invariant_mes}. Finally, by
\eqref{eq:end_proof_mass_0_0}, using that $t \mapsto 1-2\Phibf(-t)$ is $\sqrt{2/\uppi}$-Lipschitz, we have 
\begin{multline}
  \mu_\gamma((0,+\infty))\leq (\sqrt{2\uppi}\beta_{n_\gamma+1})^{-1}\int_{[0,+\infty)} \{(1+\bgamma \Ltt)w + \alpha_{n_\gamma+1}\} \rmd \mu_{\gamma}(w) + 
\zeta \gamma\sum_{k=1}^{n_\gamma} \{\alpha_k/\beta_k^3\} \\
\leq  (\InftyBound \constMoment_1 (1+\bgamma \Ltt)  +\alpha_{n_\gamma+1})/(\sqrt{2\uppi}\beta_{n_\gamma+1})+\zeta \gamma\sum_{k=1}^{n_\gamma} \{\alpha_k/\beta_k^3\} \eqsp .
\end{multline}
This finishes the proof using $n_{\gamma} = \floorLigne{\deltaInf/\gamma}$, $\deltaInf \leq \gamma(n_{\gamma}+1) \leq \deltaInf +\bgamma$ and \Cref{lem:bound_alpha_beta}. 
\end{proof}

\subsubsection{Technical results}
\label{sec:technical-results-1}

\begin{lemma}
\label{lem:0_proba}
For any $w\in \rset$,
\begin{equation}
Q_\gamma(w,\{0\})=2\Phibf\parenthese{-\frac{\tau_{\gamma}(w) + \gamma \InftyBound}{2\sigma\sqrt{\gamma}}} \eqsp ,
\end{equation}
where $Q_\gamma$ is defined by \eqref{def_q_gamma} and $\Phibf$ is the cumulative distribution of the one-dimensional Gaussian distribution with mean $0$ and variance $1$.
\end{lemma}
\begin{proof}
Let $w\in \rset$. By \eqref{eq:def_r} and the change of variable $g\to\sigma\sqrt{\gamma}g$, we get
\begin{align}
&Q_\gamma(\tcwc{w},\{0\})%&=\int_{\rset} \bpg(\tau_{\gamma}(w) + \gamma \InftyBound,g)\varphibf_1(g) \rmd g\\
=\int_{\rset} \parenthese{1\wedge \frac{\varphibf_{\sigma^2\gamma}\parenthese{ \tau_{\gamma}(w) + \gamma \InftyBound- \sigma\sqrt{\gamma}g } }{\varphibf_{\sigma^2\gamma}\parenthese{ \sigma\sqrt{\gamma}g}}}\varphibf(g) \rmd g\\
&=\int_{\rset}\varphibf_{\sigma^2\gamma}\parenthese{ g}\wedge \varphibf_{\sigma^2\gamma}\parenthese{ \tau_{\gamma}(w) + \gamma \InftyBound- g }  \rmd g=\int_{\rset}\varphibf_{\sigma^2\gamma}\parenthese{ g}\wedge \varphibf_{\sigma^2\gamma}\parenthese{ g-\tau_{\gamma}(w) \tcwc{-} \gamma \InftyBound }  \rmd g\\
&=\int_{-\infty}^{(\tau_{\gamma}(w) + \gamma \InftyBound)/2} \varphibf_{\sigma^2\gamma}\parenthese{ g-\tau_{\gamma}(w) \tcwc{-} \gamma \InftyBound }  \rmd g+\int_{(\tau_{\gamma}(w) + \gamma \InftyBound)/2}^{+\infty} \varphibf_{\sigma^2\gamma}\parenthese{ g}  \rmd g\\
&=\int_{-\infty}^{-(\tau_{\gamma}(w) + \gamma \InftyBound)/2} \varphibf_{\sigma^2\gamma}\parenthese{ g }  \rmd g+\int_{(\tau_{\gamma}(w) + \gamma \InftyBound)/2}^{+\infty} \varphibf_{\sigma^2\gamma}\parenthese{ g}  \rmd g=2\Phibf\parenthese{-\frac{\tau_{\gamma}(w) + \gamma \InftyBound}{2\sigma\sqrt{\gamma}}} \eqsp,
\end{align}
and the lemma follows.
\end{proof}

%\begin{lemma}
%\label{lem:Bernoulli}
%For any $\zeta,a>0$, and $t,x\in \coint{0,\plusinfty}$,
%\begin{align}
%&\int_\rset \varphibf_{\zeta^2}(g)\parenthese{1-1\wedge\frac{\varphibf_{\zeta^2}(t-g)}{\varphibf_{\zeta^2}(g)}}\parentheseDeux{1-2\Phibf\parenthese{-\frac{t+x-2g}{2a}}}\rmd g\\
%&\qquad\qquad\qquad\qquad\qquad\qquad\qquad\qquad\qquad\qquad\leq 1-2\Phibf\parenthese{-\frac{t+x}{2(\zeta^2+a^2)}} \eqsp .
%\end{align}
%\end{lemma}
%\begin{proof}
%Let $\zeta,a>0$, and $t,x\in \coint{0,\plusinfty}$. By \cite[Lemma 20]{durmus2019high}, we have
%\begin{align}
%&\int_\rset \varphibf_{\zeta^2}(g)\parenthese{1-1\wedge\frac{\varphibf_{\zeta^2}(t+x-g)}{\varphibf_{\zeta^2}(g)}}\parentheseDeux{1-2\Phibf\parenthese{-\frac{t+x-2g}{2a}}}\rmd g\\
%&\qquad\qquad\qquad\qquad\qquad\qquad\qquad\qquad\qquad\qquad= 1-2\Phibf\parenthese{-\frac{t+x}{2(\zeta^2+a^2)}} \eqsp .
%\end{align}
%\end{proof}

\begin{lemma}
\label{lem:null_term}
Let $\sigma^2,\gamma >0$. For any $t \geq 0$ \tcwc{and $a>0$}, we have  
\begin{equation}
\int_\rset \parentheseDeux{1-2\Phibf\parenthese{-\frac{t-2\sigma \gamma^{1/2} g}{2a}}}\bpg(t,g)\varphibf(g)\rmd g=0 \eqsp,
\end{equation}
where $\bpg$ is defined by \eqref{eq:def_bpg}, $\varphibf$ and
$\Phibf$ are the density and the cumulative distribution function of the one-dimensional
Gaussian distribution with mean $0$ and variance $1$ respectively.
\end{lemma}
\begin{proof}
  Using the changes of variable $g \mapsto \sigma \gamma^{1/2}g$ and \tcwc{$g \mapsto t-g$}, we obtain
\begin{align}
  &\int_\rset \parentheseDeux{1-2\Phibf\parenthese{-\frac{t-2\sigma \gamma^{1/2} g}{2a}}}\bpg(t,g)\varphibf(g)\rmd g\\
  &=\int_\rset \parentheseDeux{1-2\Phibf\parenthese{-\frac{t-2\sigma \gamma^{1/2} g}{2a}}}\defEns{ 1\wedge \frac{\varphibf_{\sigma^2\gamma}\parenthese{ \tcwc{t}- \sigma\sqrt{\gamma}g } }{\varphibf_{\sigma^2\gamma}\parenthese{ \sigma\sqrt{\gamma}g}}}\varphibf_{}(g)\rmd g\\
  &=\int_{-\infty}^{t/2} \parentheseDeux{1-2\Phibf\parenthese{-\frac{t-2 g}{2a}}}  \varphibf_{\sigma^2\gamma}\parenthese{ t-  g } \rmd g+\int_{t/2}^{+\infty} \parentheseDeux{1-2\Phibf\parenthese{-\frac{t-2 g}{2a}}}  \varphibf_{\sigma^2\gamma}\parenthese{ g } \rmd g\\
  % &=\int_{-\infty}^{-t/2} \parentheseDeux{1-2\Phibf\parenthese{-\frac{-t-2 g}{2a}}}  \varphibf_{\sigma^2\gamma}\parenthese{ g } \rmd g+\int_{t/2}^{+\infty} \parentheseDeux{1-2\Phibf\parenthese{-\frac{t-2 g}{2a}}}  \varphibf_{\sigma^2\gamma}\parenthese{ g } \rmd g\\
  &\tcwc{=\int_{t/2}^{+\infty} \parentheseDeux{1-2\Phibf\parenthese{-\frac{-t+2 g}{2a}}}  \varphibf_{\sigma^2\gamma}\parenthese{ g } \rmd g +\int_{t/2}^{+\infty} \parentheseDeux{1-2\Phibf\parenthese{-\frac{t-2 g}{2a}}}  \varphibf_{\sigma^2\gamma}\parenthese{ g } \rmd g
    \eqsp. }
\end{align}
Using that $s\in\rset$,  $1-2\Phibf\parenthese{s} =-\parentheseDeux{1-2\Phibf\parenthese{-s}} $ completes the proof.
\end{proof}

\begin{lemma}
\label{lem:proba_form}
Let $\sigma^2,\gamma >0$. For any $t, s \geq 0$ and $a >0$,
\begin{multline}
  \int_\rset \parentheseDeux{1-2\Phibf\parenthese{-\frac{t+s-2\sigma\sqrt{\gamma} g}{2a}}}\bpg(t,g)\varphibf(g)\rmd g\\
  = 2\proba{\sigma\sqrt{\gamma} G \geq t/2, -s-t\leq 2a \tilde{G}-2\sigma\sqrt{\gamma} G\leq s-t }\eqsp,
\end{multline}
where $G,\tilde{G}$ are two independent one-dimensional standard
Gaussian random variables, $\bpg$ is defined by \eqref{eq:def_bpg},
$\varphibf$ and $\Phibf$ are the density and the cumulative
distribution function of the one-dimensional Gaussian distribution
with mean $0$ and variance $1$ respectively.
\end{lemma}
\begin{proof}
    Using the changes of variable $g \mapsto \sigma \gamma^{1/2}g$, $g \mapsto g-t$, we get
\begin{align}
  &  \int_\rset \parentheseDeux{1-2\Phibf
    \parenthese{-\frac{t+s-2\sigma\sqrt{\gamma} g}{2a}}}\bpg(t,g)\varphibf(g)\rmd g \\
  &\qquad \qquad  = \int_{-\infty}^{-t/2} \parentheseDeux{1-2\Phibf\parenthese{-\frac{-t+s-2 g}{2a}}}  \varphibf_{\sigma^2\gamma}\parenthese{ g } \rmd g \\
  &\qquad \qquad \qquad 
   \qquad \qquad  +\int_{t/2}^{+\infty}
    \parentheseDeux{1-2\Phibf\parenthese{-\frac{t+s-2 g}{2a}}}  \varphibf_{\sigma^2\gamma}\parenthese{ g } \rmd g\\
  \label{eq:proof_gaussian_distr_follow}
  & \qquad \qquad  = 2 [\probaLigne{\sigma \gamma^{1/2} G \geq t/2}  -A-B]
\end{align}
\begin{equation}
A   =\int_{-\infty}^{-t/2}\Phibf\parenthese{-\frac{-t+s-2 g}{2a}}  \varphibf_{\sigma^2\gamma}\parenthese{ g } \rmd g \eqsp, \, B= \int_{t/2}^{+\infty}
\Phibf\parenthese{-\frac{t+s-2 g}{2a}}  \varphibf_{\sigma^2\gamma}\parenthese{ g } \rmd g  \eqsp. 
\end{equation}
In addition, we have since $(-G,-\tilde{G})$ has the same distribution than $(G,\tilde{G})$, 
\begin{multline}
  A  = \proba{\sigma\sqrt{\gamma} G\leq -\frac{t}{2}, \tilde{G}\leq-\frac{-t+s-2 \sigma\sqrt{\gamma}G}{2a} } \\= \proba{\sigma\sqrt{\gamma} G\geq \frac{t}{2}, \tilde{G}\geq\frac{-t+s+2 \sigma\sqrt{\gamma}G}{2a} } \eqsp ,
\end{multline}
and 
\begin{equation}
  B  = \proba{\sigma\sqrt{\gamma} G\geq \frac{t}{2}, \tilde{G}\leq-\frac{t+s-2 \sigma\sqrt{\gamma}G}{2a} } \eqsp .
\end{equation}
Therefore, we obtain 
\begin{align}
  &A+B\\
  &=\proba{\sigma\sqrt{\gamma} G\geq t/2}-\proba{\sigma\sqrt{\gamma} G\geq t/2, -\frac{t+s-2 \sigma\sqrt{\gamma}G}{2a}\leq\tilde{G}\leq\frac{-t+s+2 \sigma\sqrt{\gamma}G}{2a} }\\
   &=\proba{\sigma\sqrt{\gamma} G\geq t/2}-\proba{\sigma\sqrt{\gamma} G \geq t/2, -s-t\leq 2a \tilde{G}-2\sigma\sqrt{\gamma} G\leq s-t }
\end{align}
Plugging this expression in \eqref{eq:proof_gaussian_distr_follow} concludes the proof.  % \rem{P : le résultat de ce lemme englobe le lemme précédent (cas $s=0$) et sa démo ne l'utilise pas, du coup le lemme précédent est-il très utile ?}
\end{proof}

\begin{lemma}
\label{lem:Phi_inequality}
For any $a\in \rset$, $b\in \ccint{0,1}$, it holds
\begin{equation}
\Phibf\parenthese{a+b}-\Phibf\parenthese{a-b}\geq 1-2\Phibf\parenthese{-b} -a^2b\exp(-b^2/2)/\sqrt{2\uppi} \eqsp,
\end{equation}
where $\Phibf$ are \tcwc{the cumulative
distribution function} of the one-dimensional Gaussian distribution
with mean $0$ and variance $1$.
\end{lemma}
\begin{proof}
Define $\psi : \rset \times \ccint{0,1} \to \rset$  for any $a\in \rset$, $b\in \ccint{0,1}$ by
\begin{equation}
\psi(a,b)=\Phibf\parenthese{a+b}-\Phibf\parenthese{a-b}- 1+2\Phibf\parenthese{-b} +a^2b\exp(-b^2/2)/\sqrt{2\uppi} \eqsp. 
\end{equation}
We show that $\psi(a,b) \geq  0$ for any $a \in \rset$ and
$b \in \ccint{0,1}$.  Using that $1-\Phibf(-t) = \Phibf(t)$ for any
$t \in \rset$, we get that $\psi(a,b)=\psi(-a,b)$ and therefore we
only need to consider the case $a \leq 0$ and $b \in \ccint{0,1}$. In
addition, for any $b \in \ccint{0,1}$, $\psi(0,b) = 0$ and thus, it is sufficient to establish
that for any $b \in \ccint{0,1}$,  $a \mapsto \psi(a,b)$ is non-increasing on $\rset_-$.

For any $a \leq 0$ and
$b \in \ooint{0,1}$, we have using that
$\sinh(t) = \int_0^t \cosh(s)\rmd s \leq t \cosh(t)$ and
$\rme^{-t^2/2} \cosh(t) \leq 1$ for any $t \in \coint{0,\plusinfty}$,
\begin{align}
\sqrt{2\uppi}\exp(b^2/2)\frac{\partial \psi}{\partial a}(a,b)&=  2 \exp(-a^2/2)\sinh(-ab)+2ab\\
&< -2ab[\exp(-a^2b^2/2)\cosh(ab)-1]\leq 0 \eqsp. 
\end{align}
By continuity, it also holds for $a \leq 0$ and $b \in \ccint{0,1}$ which concludes the proof. 
\end{proof}

  \begin{lemma}
\label{lem:technique_lemma_induction_v2}
Assume \Cref{ass:LipsAnd}-\ref{ass:LipsAnd_2}. For any $w\in \coint{0,\plusinfty}$ and $\alpha,\beta\in \coint{0,\plusinfty}$, $\beta >0$,
\begin{align}
&  \int_{(0,+\infty)} \parentheseDeux{1-2\Phibf\parenthese{-\frac{\tau_{\gamma}(\tilde{w})+\alpha}{2\beta}}}Q_\gamma(w,\rmd \tilde{w})  \leq 1-2\Phibf\parenthese{-\frac{\tau_\gamma(w)+\gamma \InftyBound+\alpha/(1+\gamma\Ltt)}{2\sqrt{\sigma^2\gamma+\beta^2/(1+\gamma\Ltt)^2}}} \\
  & \qquad - \int_\rset \parentheseDeux{1-2\Phibf\parenthese{-\frac{\uppsi_{\gamma}(w)-2\sigma\sqrt{\gamma} g}{2\beta/(1+\gamma\Ltt)}}}\bpg(\tau_\gamma(w)+\gamma \InftyBound,g))\varphibf(g) \rmd g \\
  &  \qquad \leq  1-2\Phibf\parenthese{-\frac{\tau_\gamma(w)+\gamma \InftyBound+\alpha/(1+\gamma\Ltt)}{2\sqrt{\sigma^2\gamma+\beta^2/(1+\gamma\Ltt)^2}}} \eqsp,
\end{align}
 where $\Phibf$ are the density and the cumulative
distribution function of the one-dimensional Gaussian distribution
with mean $0$ and variance $1$. 
\end{lemma}
\begin{proof}
  Let $\alpha, \beta \geq 0$, $\beta >0$. By \eqref{def_q_gamma}, we have 
\begin{align}
&  \int_{(0,+\infty)} \parentheseDeux{1-2\Phibf\parenthese{-\frac{\tau_{\gamma}(\tilde{w})+\alpha}{2\beta}}}Q_\gamma(w,\rmd \tilde{w}) \\
&=\int_\rset \parentheseDeux{1-2\Phibf\parenthese{-\frac{\tau_{\gamma}(\tau_{\gamma}(w)+\gamma \InftyBound-2\sigma\sqrt{\gamma}g)+\alpha}{2\beta}}}(1-\bpg(\tau_\gamma(w)+\gamma \InftyBound,g))\varphibf(g) \rmd g\eqsp.
\end{align}
By \Cref{ass:LipsAnd}-\ref{ass:LipsAnd_2}, $t \mapsto 1-2\Phibf(-t)$ is increasing,  we have setting $\uppsi_{\gamma}(w) = \tau_\gamma(w)+\gamma \InftyBound+\alpha/(1+\gamma\Ltt)$,
\begin{equation}
  1-2\Phibf\parenthese{-\frac{\tau_{\gamma}(\tau_\gamma(w)+\gamma \InftyBound-2\sigma\sqrt{\gamma}g)+\alpha}{2\beta}} 
   \leq  1-2\Phibf\parenthese{-\frac{\uppsi_{\gamma}(w)-2\sigma\sqrt{\gamma} g}{2\beta/(1+\gamma\Ltt)}} \eqsp. 
\end{equation}
Using \cite[Lemma 20]{durmus2019high}, \Cref{lem:null_term} and \Cref{lem:proba_form}, we get 
\begin{align}
&\int_\rset \parentheseDeux{1-2\Phibf\parenthese{-\frac{\tau_{\gamma}(\tau_\gamma(w)+\gamma \InftyBound-2\sigma\sqrt{\gamma}g)+\alpha}{2\beta}}}(1-\bpg(\tau_\gamma(w)+\gamma \InftyBound,g))\varphibf(g) \rmd g\\
&\leq\int_\rset \parentheseDeux{1-2\Phibf\parenthese{-\frac{\uppsi_{\gamma}(w)-2\sigma\sqrt{\gamma} g}{2\beta/(1+\gamma\Ltt)}}}(1-\bpg(\uppsi_{\gamma}(w) ,g))\varphibf(g) \rmd g\\
&\qquad - \int_\rset \parentheseDeux{1-2\Phibf\parenthese{-\frac{\uppsi_{\gamma}(w)-2\sigma\sqrt{\gamma} g}{2\beta/(1+\gamma\Ltt)}}}\bpg(\tau_\gamma(w)+\gamma \InftyBound,g))\varphibf(g) \rmd g\\
    \label{eq:proof_induciton_0}
&\leq  1-2\Phibf\parenthese{-\frac{\uppsi_{\gamma}(w)}{2\sqrt{\sigma^2\gamma+\beta^2/(1+\gamma\Ltt)^2}}}  \eqsp,
\end{align}
which completes the proof. 
\end{proof}

\begin{lemma}
\label{lem:technique_lemma_induction}
Assume \Cref{ass:LipsAnd}-\ref{ass:LipsAnd_2}. For any $w\in \coint{0,\plusinfty}$ and $\alpha,\beta\in \coint{0,\plusinfty}$, \tcwc{$\beta>0$,} such that $\alpha/(2\beta)\leq 1$,
\begin{multline}
  \int_{[0,+\infty)} \parentheseDeux{1-2\Phibf\parenthese{-\frac{\tau_{\gamma}(\tilde{w})+\alpha}{2\beta}}}Q_\gamma(w,\rmd \tilde{w})\\
  \leq 1-2\Phibf\parenthese{-\frac{\tau_\gamma(w)+\gamma \InftyBound+\alpha/(1+\gamma\Ltt)}{2\sqrt{\sigma^2\gamma+\beta^{\tcwc{2}}/(1+\gamma\Ltt)^2}}} +\zeta\frac{\gamma\alpha}{\beta^3}\eqsp,
\end{multline}
where
 $\zeta$  is defined in \eqref{eq:def_zeta}.
\end{lemma}
\begin{proof}
  Let $\alpha, \beta \geq 0$ such that $\alpha/(2\beta)\leq 1$. By \tcr{\Cref{lem:0_proba}} and \Cref{lem:technique_lemma_induction_v2}, we have 
  \begin{align}
    &  \int_{\tcr{[}0,+\infty)} \parentheseDeux{1-2\Phibf\parenthese{-\frac{\tau_{\gamma}(\tilde{w})+\alpha}{2\beta}}}Q_\gamma(w,\rmd \tilde{w})  \leq 1-2\Phibf\parenthese{-\frac{\tau_\gamma(w)+\gamma \InftyBound+\alpha/(1+\gamma\Ltt)}{2\sqrt{\sigma^2\gamma+\beta^2/(1+\gamma\Ltt)^2}}} \\ &\qquad+\parentheseDeux{1-2\Phibf\parenthese{-\frac{\tau_{\gamma}(0)+\alpha}{2\beta}}}2\Phibf\parenthese{-\frac{\tau_\gamma(w)+\gamma \InftyBound}{2\sigma\sqrt{\gamma}}} \\
  & \qquad - \int_\rset \parentheseDeux{1-2\Phibf\parenthese{-\frac{\uppsi_{\gamma}(w)-2\sigma\sqrt{\gamma} g}{2\beta/(1+\gamma\Ltt)}}}\bpg(\tau_\gamma(w)+\gamma \InftyBound,g))\varphibf(g) \rmd g \\
% &  \int_{[0,+\infty)} \parentheseDeux{1-2\Phibf\parenthese{-\frac{\tau_{\gamma}(\tilde{w})+\alpha}{2\beta}}}Q_\gamma(w,\rmd \tilde{w}) \\
% &=
% 1-2\Phibf\parenthese{-\frac{\tau_\gamma(w)+\gamma \InftyBound+\alpha/(1+\gamma\Ltt)}{2\sqrt{\sigma^2\gamma+\beta^2/(1+\gamma\Ltt)^2}}}\\
% &\qquad+\parentheseDeux{1-2\Phibf\parenthese{-\frac{\tau_{\gamma}(0)+\alpha}{2\beta}}}\int_\rset \bpg(\tau_{\gamma}(w)+\gamma \InftyBound,g)\varphibf(g) \rmd g \\
% &=
% 1-2\Phibf\parenthese{-\frac{\tau_\gamma(w)+\gamma \InftyBound+\alpha/(1+\gamma\Ltt)}{2\sqrt{\sigma^2\gamma+\beta^2/(1+\gamma\Ltt)^2}}}\\
%      \label{eq:proof_induciton_0_0}
&   \leq 1-2\Phibf\parenthese{-\frac{\tau_\gamma(w)+\gamma \InftyBound+\alpha/(1+\gamma\Ltt)}{2\sqrt{\sigma^2\gamma+\beta^2/(1+\gamma\Ltt)^2}}} +\parentheseDeux{1-2\Phibf\parenthese{-\frac{\alpha}{2\beta}}}2\Phibf\parenthese{-\frac{\tau_\gamma(w)+\gamma \InftyBound}{2\sigma\sqrt{\gamma}}} \\
    & \qquad  -2 \proba{\{2 \sigma\sqrt{\gamma} G \geq \tau_\gamma(w)+\gamma\InftyBound\} \cap \msa} \eqsp,
      \label{eq:proof_induciton_0_0}
\end{align}
where we used $\tau_{\gamma}(0)=0$ and set 
\begin{equation}
  \label{eq:def_mca_proof_induciton}
  \msa=\defEns{-\tau_\gamma(w)-\gamma\InftyBound-\frac{\alpha}{1+\gamma\Ltt}\leq \frac{2\beta\tilde{G} }{1+\gamma \Ltt} -2\sigma \sqrt{\gamma} G\leq -\tau_\gamma(w)-\gamma\InftyBound+\frac{\alpha}{1+\gamma\Ltt}} \eqsp,
\end{equation}
and $G,\tilde{G}$ are two independent one-dimensional standard
Gaussian random variables.

Define $\uptheta_{\gamma} : [0,+\infty) \times \rset \to \rset$ for any
$w \in \coint{0,\plusinfty}$ and $g \in \rset$ by
$\uptheta_{\gamma}(w,g) =(2\beta)^{-1}(1+\gamma\Ltt)[-\tau_\gamma(w)-\gamma\InftyBound+\sigma\sqrt{\gamma}g]$.  Then, using that
$$\msa=\defEns{\uptheta_{\gamma}(w,G)-\alpha/(2\beta)\leq \tilde{G}\leq
\uptheta_{\gamma}(w,G)+\alpha/(2\beta)} \eqsp,$$ we have
\begin{align}
&\proba{\{2 \sigma\sqrt{\gamma} G \geq \tau_\gamma(w)+\gamma\InftyBound\} \cap \msa}\\&=\int_{\rset} \1_{[0,+\infty)}\parenthese{g-\frac{\tau_\gamma(w)+\gamma\InftyBound}{2\sigma\sqrt{\gamma}}}\parentheseDeux{\Phibf\parenthese{\uptheta_{\gamma}(w,g)+\frac{\alpha}{2\beta}}-\Phibf\parenthese{\uptheta_{\gamma}(w,g)-\frac{\alpha}{2\beta}}}\varphibf(g)\rmd g \eqsp.
\end{align}

Since $\alpha/(2\beta)\leq 1$ by \Cref{lem:Phi_inequality} we have, for any $a\in \rset$,
\begin{equation}
\Phibf\parenthese{a+\frac{\alpha}{2\beta}}-\Phibf\parenthese{a-\frac{\alpha}{2\beta}}\geq 1-2\Phibf\parenthese{-\frac{\alpha}{2\beta}}-\frac{a^2\alpha}{2\sqrt{2\uppi}\beta}\rme^{-\alpha^2/(8\beta^2)} \eqsp,
\end{equation}
which implies 
\begin{multline}
\proba{\{2 \sigma\sqrt{\gamma} G \geq \tau_\gamma(w)+\gamma\InftyBound\} \cap \msa} \geq \Phibf\parenthese{-\frac{\tau_\gamma(w)+\gamma \InftyBound}{2\sigma\sqrt{\gamma}}}\parenthese{1-2\Phibf\parenthese{-\frac{\alpha}{2\beta}}}
\\    -\int_{\rset} \1_{[0,+\infty)}\parenthese{g-\frac{\tau_\gamma(w)+\gamma\InftyBound}{2\sigma\sqrt{\gamma}}}\frac{\uptheta^2_{\gamma}(g,w)\alpha}{2\sqrt{2\uppi}\beta}\rme^{-\alpha^2/(8\beta^2)}\varphibf(g)\rmd g \eqsp. 
\end{multline}
Therefore, we obtain using that $\PE[\1_{[0,+\infty)}(G) G^2] = 1/2$,
\begin{align}
&\parenthese{1-2\Phibf\parenthese{-\frac{\alpha}{2\beta}}}  \Phibf\parenthese{-\frac{\tau_\gamma(w)+\gamma \InftyBound}{2\sigma\sqrt{\gamma}}}-\proba{\{2 \sigma\sqrt{\gamma} G \geq \tau_\gamma(w)+\gamma\InftyBound\} \cap \msa}\\
&\leq \int_{\rset} \1_{[0,+\infty)}\parenthese{g-\frac{\tau_\gamma(w)+\gamma\InftyBound}{2\sigma\sqrt{\gamma}}}\frac{\uptheta^2_{\gamma}(g,w)\alpha}{2\sqrt{2\uppi}\beta}\rme^{-\alpha^2/(8\beta^2)}\varphibf(g)\rmd g\\
&\leq \frac{\alpha\gamma}{\beta^3}\frac{(1+\gamma\Ltt)^2\sigma^2}{2\sqrt{2\uppi}}\rme^{-\alpha^2/(8\beta^2)}\Bigg[\parenthese{\frac{\tau_\gamma(w)+\gamma\InftyBound}{2\sigma\sqrt{\gamma}}}^2\Phibf\parenthese{-\frac{\tau_\gamma(w)+\gamma\InftyBound}{2\sigma\sqrt{\gamma}}}\\
&\qquad-\frac{1}{\sqrt{2\uppi}}\frac{\tau_\gamma(w)+\gamma\InftyBound}{2\sigma\sqrt{\gamma}}\exp\parenthese{-\parenthese{\frac{\tau_\gamma(w)+\gamma\InftyBound}{2\sigma\sqrt{\gamma}}}^2\middle/2}+1/8\Bigg] \eqsp.
\end{align}
Plugging this inequality  in \eqref{eq:proof_induciton_0_0} concludes the proof.
\end{proof}

\begin{lemma}
  \label{lem:bound_alpha_beta}
For any $\gamma>0$, $k\geq 1$, we have 
\begin{align}
& k\gamma\InftyBound \rme^{-k \gamma \Ltt} \leq \alpha_k
                                                              \leq k\gamma\InftyBound \eqsp, 
\qquad  (k\gamma)^{1/2}\sigma \rme^{-k \gamma \Ltt} \leq \beta_k
                                                              \leq (k\gamma)^{1/2}\sigma \eqsp, \\
&  [(k\gamma)^{1/2}\InftyBound/\sigma] \rme^{-k \gamma \Ltt} \leq \alpha_k/\beta_k
  \leq [(k\gamma)^{1/2}\InftyBound/\sigma]\rme^{k \gamma \Ltt} \\
    \label{eq:sum_alpha_beta_cube_0}
& [\InftyBound\gamma^{1/2}/(\sigma^3k^{1/2})]\rme^{-k\gamma\Ltt} \leq \gamma\alpha_k/\beta_k^3\leq [\InftyBound\gamma^{1/2}/(\sigma^3k^{1/2})]\rme^{3k\gamma\Ltt} \\
  \label{eq:sum_alpha_beta_cube}
&\textstyle{  \gamma \sum_{i=1}^{k-1} \{\alpha_i/\beta_i^3\}\leq [2\InftyBound(k\gamma)^{1/2}/\sigma^3]\rme^{3k\gamma\Ltt}}\eqsp,
\end{align}

%\begin{equation}
%\sqrt{2+\gamma \Ltt}\frac{\sqrt{\gamma k}\InftyBound}{2\sigma\sqrt{\rme^{k\gamma\Ltt}+1}}\leq \frac{\alpha_k}{2\beta_k}\leq \frac{\sqrt{\gamma k}\InftyBound}{2\sqrt{2}\sigma}\rme^{k\gamma\Ltt/2}\sqrt{2+\gamma \Ltt}
%\end{equation}
%and 
%\begin{equation}
%\sigma^2\gamma k \rme^{-2L}\leq \beta_k^2\leq \sigma^2\gamma k
%\end{equation}
%\begin{equation}
%\Phibf\parenthese{-\frac{\alpha_k}{2\beta_k}}-\Phibf\parenthese{-\frac{\alpha_{k-1}}{2\beta_{k-1}}}\leq -\frac{\InftyBound}{4\sigma}\sqrt{\frac{\gamma}{k}}\sqrt{2+\gamma \Ltt}\frac{(1+\gamma \Ltt)^{k-1}}{\sqrt{2\pi}\rme^{\gamma k \Ltt}}\rme^{-\InftyBound^2\rme^{2\Ltt}(2+\gamma \Ltt)/(16\sigma^2)}
%\end{equation}
where $(\alpha_k)_{k\geq 1},(\beta_k)_{k\geq 1}$ are defined in
\eqref{eq:def_alpha_beta}.
\end{lemma}
\begin{proof}
Let $k\geq 1$. 
Using for any $i\in \mathbb{N}$, $\rme^{-i\gamma\Ltt}\leq (1+\gamma\Ltt)^{-i}\leq 1$, we have 
\begin{equation}
  \label{eq:bound_alpha_beta_proof}
k\gamma\InftyBound\rme^{-k\gamma\Ltt}\leq\gamma\InftyBound\sum_{i=0}^{k-1}(1+\gamma\Ltt)^{-i}\leq k\gamma\InftyBound \eqsp.
\end{equation}
In the same way, using for any $i\in \mathbb{N}$, $\rme^{-2i\gamma\Ltt}\leq (1+\gamma\Ltt)^{-2i}\leq 1$, we obtain
\begin{equation}
    \label{eq:bound_alpha_beta_proof_2}
k\gamma\sigma^2\rme^{-2k\gamma\Ltt}\leq \gamma\sigma^2\sum_{i=0}^{k-1}(1+\gamma\Ltt)^{-2i}\leq k\gamma\sigma^2 \eqsp. 
\end{equation}
Combining \eqref{eq:bound_alpha_beta_proof} and \eqref{eq:bound_alpha_beta_proof_2} completes the proof of the first four inequalities. Then, \eqref{eq:sum_alpha_beta_cube} is a simple consequence of  \eqref{eq:sum_alpha_beta_cube_0} and a comparison test. 
\end{proof}

  \subsection{Proof of \Cref{theo:convergence_Q_gamma}}
\label{sec:proof:theo:convergence_Q_gamma}

We show first that any compact sets of $\rset_+$ are $k+1$-small sets (also called $(\varepsilon_k,k)$ Doeblin sets) with an explicit constant $\varepsilon_k >0$ depdending on $k \in\nset$.
To this end, we need the following technical lemma.

  We consider in what follows that $(W_k)_{ k\in\nset}$ is the
  canonical process on
  \sloppy$([0,+\infty)^{\nset}, \mcbb([0,+\infty))^{\otimes \nset})$ and for any
  $w \in \coint{0,\plusinfty}$, $\PP_{w}$ and $\PE_{w}$ correspond to the probability
  and expectation respectively, associated with $Q_{\gamma}$ and the initial condition $\updelta_w$ on this space. 

  \begin{lemma}
    \label{lem:min_w_i_zero}
    Assume \Cref{ass:LipsAnd}-\ref{ass:LipsAnd_2}. For any $k \in\nset$ and $w \in \ooint{0,\plusinfty}$,
    \begin{equation}
      \label{eq:2}
      \probaMarkov{w}{ \min_{i \in\{0,\ldots,k+1\}} W_i >0} \leq 1 - 2 \Phibf\parentheseDeux{ -\frac{\tau_{\gamma}(w)+ \alpha_{k+1}}{2\beta_{k+1}}} \eqsp,
    \end{equation}
where $\alpha_{k+1}, \beta_{k+1}$ are given in \eqref{eq:def_alpha_beta}. 
  \end{lemma}

  \begin{proof}
    The proof is by induction on $k \in \nset$. The proof for $k =0$ follows from \Cref{lem:0_proba}. Assume that the result holds for $k-1 \in \nset$ and for any $w \in \ooint{0,\plusinfty}$. Then, by the Markov property and the assumption hypothesis,  for any $w \in \ooint{0,\plusinfty}$,
    \begin{multline}
      \label{eq:3}
      \probaMarkov{w}{\min_{i \in\{0,\ldots,k\}} W_i >0 } =       \expeMarkov{w}{\1_{(0,+\infty)}(W_1) \probaMarkov{W_1}{\min_{i \in\{0,\ldots,k-1\}} W_i >0 }} \\
      \leq  \expeMarkov{w}{\1_{(0,+\infty)}(W_1) \defEns{1 - 2 \Phibf\parenthese{ -\frac{\tau_{\gamma}(W_1)+ \alpha_{k}}{2\beta_{k}}}}} \eqsp.
    \end{multline}
    The proof is then completed upon using \Cref{lem:technique_lemma_induction_v2}.
  \end{proof}

We now ready to show that any compact set of $\rset_+$ is $k+1$-small. 

  \begin{lemma}
    \label{lem:minorization}
    Assume \Cref{ass:LipsAnd}-\ref{ass:LipsAnd_2}.  Then, for any $k \in \nset$, $w,\tw \in \coint{0,\plusinfty}$,
    \begin{equation}
      \label{eq:4}
      \tvnorm{\updelta_w Q_{\gamma}^{k+1} - \updelta_{\tw} Q_{\gamma}^{k+1}} \leq 1 - 2 \Phibf\parentheseDeux{ -\frac{\tau_{\gamma}(w\vee \tw)+ \alpha_{k+1}}{2\beta_{k+1}}} \eqsp,
    \end{equation}
    where $\alpha_{k+1}, \beta_{k+1}$ are given in \eqref{eq:def_alpha_beta}. 
  \end{lemma}
  \begin{proof}
    We consider again $(G_k)_{k \geq 1}$ and $(U_k)_{k \geq 1}$ two independent sequences of \iid~standard Gaussian and $\ccint{0,1}$-uniform random variables respectively. Define the Markov chains $(W_k)_{k\in \mathbb{N}}$ and $(\tW_k)_{k\in \mathbb{N}}$ starting from $w \in \coint{0,\plusinfty}$ and $\tw \in \coint{0,\plusinfty}$ respectively, for any $k \in \nset$, $W_{k+1}=\mathscr{G}_\gamma (W_k,G_{k+1},U_{k+1})$ and $\tW_{k+1}=\mathscr{G}_\gamma (\tW_k,G_{k+1},U_{k+1})$. Note that the case $w = \tw$ is trivial so we only consider the converse and assume that $w < \tw$, $\tw >0$. Then, we obtain 
    by \Cref{lem:G_inequality} that almost surely $W_{k} \leq \tW_k$ for any $k \in \nset$, which implies that
    \begin{equation}
      \label{eq:5}
      \tvnorm{\updelta_w Q_{\gamma}^{k+1} - \updelta_{\tw} Q_{\gamma}^{k+1}}  \leq       \proba{W_{k+1} \neq \tW_{k+1}} \leq \proba{\min_{i \in \{0,\ldots,k+1\}} \tW_{i} > 0} \eqsp.
    \end{equation}
    Indeed, if $\min_{i \in \{0,\ldots,k+1\}} \tW_{i} =0$, then there exists $i\in \{0,\ldots,k+1\}$, $\tW_i = 0$ which implies since $\tW_i \geq W_i \geq 0$ that $\tW_i = W_i$ and therefore $\tW_{k+1} = W_{k+1}$ by definition of the two processes. The proof is then completed by  \Cref{lem:min_w_i_zero}.
\end{proof}

We easily deduce then the following corollary.

\begin{corollary}
  \label{propo:minorization}
  Assume \Cref{ass:LipsAnd}-\ref{ass:LipsAnd_2}. Let $t_0 >0$. Then, for any $w,\tw \in \coint{0,\plusinfty}$,
    \begin{equation}
      \label{eq:4}
      \tvnorm{\updelta_w Q_{\gamma}^{\ceil{t_0/\gamma}} - \updelta_{\tw} Q_{\gamma}^{\ceil{t_0/\gamma}}} \leq 1 - 2 \Phibf\parentheseDeux{ -\Ltt^{\half} \frac{\tcrw{(1+\bgamma\Ltt)}(w\vee \tw)+ (t_0\tcrw{+\bgamma}) \InftyBound}{\{\tcrw{2} \sigma^2(1-\rme^{-\tcrw{( \Ltt t_0\wedge 2\log(2))}})\}^{\half}}} \eqsp. 
    \end{equation}
\end{corollary}
\begin{proof}
  Note that by \eqref{eq:def_alpha_beta} for any $k \in \nset$,
  $\alpha_{k+1} \leq (k\tcrw{+1}) \gamma \InftyBound$ and $\beta_{k+1}^2 \tcrw{\geq} (\sigma^{2}(1+\gamma\Ltt)/(\tcrw{2}\Ltt))\{1-(1+\gamma\Ltt)^{-2\tcrw{(k+1)}}\} \geq (\sigma^{2}(1+\gamma\Ltt)/(\tcrw{2}\Ltt)) \{1 - \rme^{-2 \tcrw{(k+1)} \tcrw{(\gamma \Ltt/2\wedge \log(2))}}\}$ using that \tcrw{$\log(1+t) \geq t/2 \wedge \log(2)$} for any $t \in \rset_+$. The proof is then completed using \Cref{lem:minorization}, \Cref{ass:LipsAnd}-\ref{ass:LipsAnd_2} and the previous bounds for $k \leftarrow \tcrw{\ceil{t_0/\gamma}-1}$.
\end{proof}

\tcwc{Define $\VlyapD_1,\VlyapDexp_a : \rset_+ \to [1,+\infty)$ for any $w \in \rset_+$ by
\begin{equation}
  \label{eq:eq:def_lyap_convergence_geo}
  \VlyapD_1(w) = 1 + w \eqsp, \qquad  \VlyapDexp_a(w) = \exp(aw) \eqsp.
\end{equation}}

    \begin{lemma}
    \label{lem:drift_exp_sticky_v1}
    Assume \Cref{ass:LipsAnd}-\ref{ass:LipsAnd_2}. Let $a >0$. Then, for any $w  \in \coint{0,\plusinfty}$ and $\gamma \in \ocint{0,\bgamma_1}$,
    \begin{align}
            \label{eq:drift_lin_sticky_v1}
Q_\gamma \VlyapD_1(w)&\leq (1-\gamma\mtt)\VlyapD_1(w)+\gamma R_1(\mtt+ \Ltt)+\gamma \InftyBound+\gamma \mtt \eqsp ,
\\
      \label{eq:drift_exp_sticky_v1}
      Q_{\gamma} \lyapDexp_a (w) &\leq \uplambda_a^{\gamma}\lyapDexp_a (w)  + \gamma \upalpha_a \eqsp, 
    \end{align}
    where $Q_{\gamma}$ is defined by \eqref{def_q_gamma} and  $\uplambda_a, \bgamma_1, R_a, B_a, D_a$ are defined by \eqref{eq:fix_exp_deift_const} and 
    \begin{equation}
      \label{eq:drift_exp_sticky_v2_def_A}
     \upalpha_a =  \log(B_a/\uplambda_a)\lyapDexp_a(R_a) B_a^{\bgamma}-\log(\uplambda_a)+D_a\eqsp.
    \end{equation}
  \end{lemma}
  \begin{proof}
    \eqref{eq:drift_lin_sticky_v1} is a simple consequence of \Cref{prop:drift_v_norm}.
    In addition, by \Cref{prop:fix_drift_exp_drift}, we have 
    \begin{align}
      Q_{\gamma} \VlyapDexp_{a}(w)&\leq \uplambda_a^\gamma \VlyapDsexp_{a}(w)\1_{\coint{R_a,\plusinfty}}(w)+B_a^\gamma \VlyapDsexp_{a}(w)\1_{\coint{0,R_a}}(w) +\gamma D_a\1_{\coint{0,R_a}}(w)+1\\
      &\leq \uplambda_a^\gamma \VlyapDexp_{a}(w)+\uplambda_a^\gamma(B_a^\gamma\uplambda_a^{-\gamma}-1) \VlyapDexp_{a}(R_a) +\gamma D_a+1-\uplambda_a^\gamma \eqsp .
    \end{align}
    Using $\rme^t - 1 \leq t \rme^t$ for any $t \geq 0$, completes the proof. 
  \end{proof}

\begin{proof}[Proof of \Cref{theo:convergence_Q_gamma}]
  Let $a >0$ and $t_0 >0$. By \Cref{lem:drift_exp_sticky_v1} and an easy induction, and using
  that $t\rme^{-t} \leq 1-\rme^{-t}\leq t$ for any $t \geq 0$, we have for any
  $k \in \nset$,
       \begin{align}
Q_\gamma^k \VlyapD_1(w)\leq \uplambda_1^{k\gamma}\VlyapD_1(w)+[ R_1(\mtt+ \Ltt)+ \InftyBound+ \mtt]/\mtt \eqsp ,
 \quad            Q_{\gamma}^k \lyapDexp_a (w) \leq \uplambda_a^{k\gamma}\lyapDexp_a (w)  +k \gamma \upalpha_a\uplambda_a^{-\bgamma} \eqsp.
    \end{align}
    where $Q_{\gamma}$ is defined by \eqref{def_q_gamma}, $\VlyapD_1, \VlyapDexp_a$ by \eqref{eq:eq:def_lyap_convergence_geo}, $\uplambda_a$  by \eqref{eq:fix_exp_deift_const}, $\upalpha_a$ by \eqref{eq:drift_exp_sticky_v2_def_A} and 
    \begin{equation}
      \label{eq:def_const_drift_convergence}
\uplambda_1=\rme^{-\mtt } \eqsp, \qquad \upbeta_1 =  [ R_1(\mtt+ \Ltt)+ \InftyBound+ \mtt]/\mtt \eqsp, \qquad \upbeta_a = (t_0 + \bgamma) \upalpha_a\uplambda_a^{-\bgamma} \eqsp. 
\end{equation}
Then, for any $w  \in \coint{0,\plusinfty}$ and $\gamma \in \ocint{0,\bgamma_1}$,
    \begin{equation}
            \label{eq:drift_sticky_v2}
Q_\gamma^{\ceil{t_0/\gamma}} \VlyapD_1(w)\leq \uplambda_1^{\tcwc{t_0}}\VlyapD_1(w)+ \upbeta_1 \eqsp , \qquad 
      Q^{\ceil{t_0/\gamma}}_{\gamma} \lyapDexp_a (w) \leq \uplambda_a^{\tcwc{t_0}} \lyapDexp_a (w)  +  \upbeta_a \eqsp.
    \end{equation}

    We now only complete the proof show for $\VlyapD = \VlyapD_1$. The result for $\VlyapD = \VlyapDexp_a$, $a >0$ is similar upon replacing $\uplambda_1$ and $\upbeta_1$ by $\uplambda_a$ and $\upbeta_a$ given in \eqref{eq:def_const_drift_convergence} respectively.

  Define $\updelta_1 = 4 \upbeta_1/(1-\uplambda_1) -1$ and $M_1 = \sup \{ w \in \coint{0,\plusinfty} \, : \, \LyapD_1(w) \leq \updelta_1\}$ which is well defined since $\lim_{w \to \plusinfty} \LyapD_1(w) = \plusinfty$. Define in addition,
  \begin{equation}
    \varepsilon_1 = 2 \Phibf\parentheseDeux{ -\Ltt^{\half} \frac{\tcrw{(1+\bgamma\Ltt)}M_1+ (t_0\tcrw{+\bgamma}) \InftyBound}{\{\tcrw{2} \sigma^2(1-\rme^{-\tcrw{( \Ltt t_0\wedge 2\log(2))}})\}^{\half}}} < 1 \eqsp.
  \end{equation}
  Then, $\{\LyapD_1 \leq \updelta_1\}$ is a $(\tcwc{\ceil{t_0/\gamma}},\varepsilon_{\tcwc{1}})$-Doeblin set for $Q_{\gamma}$ and $\uplambda_1 +2\upbeta_1/(1+\updelta_1) <1$. Therefore, \cite[Theorem 19.4.1]{douc:moulines:priouret:soulier:2018}\footnote{\tcwc{There is a $b$ missing in Equation 19.4.2d in \cite[Theorem 19.4.1]{douc:moulines:priouret:soulier:2018} which has been confirmed by one of the authors of \cite{douc:moulines:priouret:soulier:2018}}} implies that for any $k \in \nset$,
  \begin{equation}
    \label{eq:6}
    \norm{\updelta_w Q_{\gamma}^{\tcwc{k}} - \mu_{\gamma}}_{\LyapD_1} \leq \tilde{C} \rho^k\{\LyapD_1(w) + \mu_{\gamma}(\LyapD_1)\} \eqsp,
  \end{equation}
  where a bound on $\mu_{\gamma}(\LyapD_1)$ is provided by \Cref{theo:bound_moment_rset_star_invariant_mes} and
  \begin{align}
    \log(\rho) &= \tcwc{(t_0+\bgamma)^{-1}}\log(1-\varepsilon_1)\log(\bar{\uplambda_1})/\{\log(1-\varepsilon_1)+\log(\bar{\uplambda_1})-\log(\bar{\upbeta}_1)\}\\
    \bar{\uplambda}_1 &= \uplambda_1^{\tcwc{t_0}} + 2\upbeta_1/(1+\updelta_1) \eqsp, \quad \bar{\upbeta}_1 = \uplambda_1^{\tcwc{t_0}} \upbeta_1+\updelta_1 \\
    \tilde{C} &= \tcwc{\rho^{-1}}\{\uplambda_1^{\tcwc{t_0}}+\upbeta_1\}/[1+\bar{\upbeta}_1/\{(1-\varepsilon_1)(1-\bar{\uplambda}_1)\}] \eqsp. 
  \end{align}
\end{proof}

%%% Local Variables:
%%% mode: latex
%%% TeX-master: "main_imsart"
%%% End:

% !TeX root = main_imsart.tex
%\subsection{Postponed proofs of \Cref{sec:contin_limit}}
%In this section, $C$ denotes a constant which can change from line to line. 

\subsection{Proof of \Cref{propo:moment_conv_continuous}}
\label{sec:proof-crefpr_moment_cont}
  The  proof is an easy consequence of  \Cref{lem:moment_four_iterates_conti_limit} below and the definition of $(\Wbfn_t)_{t \geq 0}$. 
Before stating and proving \Cref{lem:moment_four_iterates_conti_limit}, we need the following technical results. 
\begin{lemma}
  \label{lem:moment_martinga_incre_proof_cont}
  Assume \Cref{ass:sticky_cont}. Then, for any $q \in \coint{1,\plusinfty}$, we have for any $\gamma \in \ocint{0,\bgamma}$, 
  \begin{equation}
    \label{eq:14}
    \txts    \expeLigne{\abs{W_1-\tau_{\gamma}(W_0)-\gamma \InftyBound}^q} \leq (4\sigma^2 \gamma)^{q/2}\defEns{ \bfm_{q} +
    2 \sup_{u \geq 0} [u^q \Phibf(-u)] }\eqsp,
  \end{equation}
  where $W_1$ is defined by \eqref{eq:def_r}, $\bfm_{q}$ is the $q$-th moment of the standard Gaussian distribution and $\Phibf$ is its cumulative distribution function. 
\end{lemma}
\begin{proof}
  Let $w_0 \in\coint{0,\plusinfty}$ and $\gamma \in\ocint{0,\bgamma}$.
  By definition \eqref{def_q_gamma} and \eqref{eq:def_bpg}, we have setting $\bar{\tau}_{\gamma}^{\infty}(w_0) = \{\tau_{\gamma}(w_0) +\gamma \InftyBound\}/(2 \sqrt{\sigma^2 \gamma})$, 
  \begin{align}
   & \int_{\rset_+}   \abs{w_1-\tau_{\gamma}(w_0)-\gamma \InftyBound}^q Q_{\gamma}(w_0,\rmd w_1) = (4 \sigma^2 \gamma)^{q/2} \int_{-\infty}^{\bar{\tau}_{\gamma}^{\infty}(w_0)} \abs{g}^q \varphibf(g) \rmd g \\
 &   \qquad\qquad\qquad\qquad\qquad\qquad\qquad -(4 \sigma^2 \gamma)^{q/2} \int_{\bar{\tau}_{\gamma}^{\infty}(w_0)}^{\plusinfty} \abs{g\tcrw{-2\bar{\tau}_{\gamma}^{\infty}(w_0)}}^q\varphibf(g) \rmd g \\
 &\qquad\qquad\qquad\qquad\qquad\qquad\qquad+ \int_{\rset} \absLigne{\tau_{\gamma}(w_0)+\gamma \InftyBound}^q
\varphibf( 2 \bar{\tau}_{\gamma}^{\infty}(w_0)-g) \wedge \varphibf(g) 
    \rmd g \\
&   \qquad\qquad\qquad\qquad\qquad\qquad\qquad \leq (4 \sigma^2 \gamma)^{q/2} \bfm_{q} + 2^{\tcrw{q}+1} \sigma^{\tcrw{q}} \gamma^{\tcrw{q}/2} [\bar{\tau}_{\gamma}^{\infty}(w_0)]^q \Phibf\defEnsLigne{ -\bar{\tau}_{\gamma}^{\infty}(w_0)} \eqsp,
  \end{align}
  which completes the proof. 
\end{proof}

\begin{lemma}
    \label{lem:moment_four_iterates_conti_limit_0}
    Assume \Cref{ass:sticky_cont}. Then, there exists $C \geq 0$ such
    that for any $k \in \nset$ and $\gamma \in \ocint{0,\bgamma}$,
    $\expeLigne{W_k^4} \leq \rme^{C k \gamma} \{\expeLigne{W_0^4} + 1\}$, where $(W_k)_{k \in \nset}$ is defined by \eqref{eq:def_r}.
  \end{lemma}
  \begin{proof}
    By \Cref{ass:sticky_cont}, \eqref{eq:def_r} and \eqref{def_q_gamma} we have that for any $w_0 \in \coint{0,\plusinfty}$ and $\gamma \in \ocint{0,\bgamma}$, setting $W_0 = w_0$ and $\kappa_{\infty}(w_0) = \kappa(w_0)+ \InftyBound$,
    \begin{align}
      %\label{eq:10}
      \int_{\rset_+} w_1^4      Q_{\gamma}(w_0, \rmd w_1)& = \expeLigne{W_1^4} \leq \expeLigne{(w_0 + \gamma \kappa_{\infty}(w_0) - 2 \sqrt{\sigma^2 \gamma} G_1)^4}\\
      & = \{w_0 + \gamma \kappa_{\infty}(w_0)\}^4 + \tcrw{48}\sigma^2 \gamma  \{w_0 + \gamma \kappa_{\infty}(w_0)\}^2 + \tcrw{2^4}\sigma^4 \gamma^2 \eqsp.
    \end{align}
    By  \Cref{ass:sticky_cont}, for any $\ell \in \{2,4\}$, we have that for any $w_0 \in \coint{0,\plusinfty}$, $\gamma \in \ocint{0,\bgamma}$,
    \begin{equation}
      \label{eq:11}
      \{w_0 + \gamma \kappa_{\infty}(w_0)\}^{\ell} \leq w_0^\ell + 2^{\ell-1} \tcrw{(1\vee\bgamma)}^{\ell}\ell \gamma(1+ \Ltt_{\kappa})^{\ell} [\abs{w_0}^{\ell} + \InftyBound^{\tcrw{\ell}}]   \eqsp.
    \end{equation}
    Therefore, we obtain that there exists some constant $C_1,C_2 \geq 0$, such that for any $w_0 \in \coint{0,\plusinfty}$, $\gamma \in \ocint{0,\bgamma}$,
    \begin{equation}
      \label{eq:12}
      \int_{\rset_+} w_1^4      Q_{\gamma}(w_0, \rmd w_1) \leq w_0^4 + C_1 \gamma\{1+w_0^2+w_0^4\} \leq (1+\gamma C_2)w_0^4 + \gamma C_2 \eqsp.       
    \end{equation}
    By an easy induction, we get then that for any $w_0 \in \coint{0,\plusinfty}$, $\gamma \in \ocint{0,\bgamma}$ and $k \in \nset$,
    \begin{equation}
      \label{eq:12}
      \int_{\rset_+} w_1^4      Q_{\gamma}^k(w_0, \rmd w_1) \leq (1+C_2\gamma)^kw_0^4 + C_2 \gamma\sum_{i=0}^{k-1} (1+C_2\gamma)^i \leq \rme^{k\gamma C_2}[ w_0^4 + \tcrw{1}]   \eqsp,
    \end{equation}
    which completes the proof by the Markov property. 
  \end{proof}
  \begin{lemma}
    \label{lem:moment_four_iterates_conti_limit}
    Assume \Cref{ass:sticky_cont}. Then, there exists $C \geq 0$ such
    that for any $k \in \nset$ and $\gamma \in \ocint{0,\bgamma}$,
    $\expeLigne{\max_{\ell \in\{0,\ldots,k\}} [W_\ell-W_0]^4} \leq C (k\gamma)^2 \rme^{C k \gamma} \{\expeLigne{W_0^4} + 1\}$, where $(W_k)_{k \in \nset}$ is defined by \eqref{eq:def_r}.
  \end{lemma}
  \begin{proof}
    Assume that $\expe{W_0^4} < \plusinfty$, otherwise the results holds.  Denote by
    $(\mcf_k)_{k \in\nset}$ the filtration associated with $(W_k)_{k \in \nset}$.    We consider the following decomposition for any $\ell \in \nset$,
    \begin{equation}
      \label{eq:27}
      W_\ell  - W_0 = A_\ell + B_\ell \eqsp, \quad A_\ell = \sum_{i=0}^{\ell-1} \Delta M_i \eqsp,  \quad   B_\ell = \sum_{i=0}^{\ell-1} H_i \eqsp,
    \end{equation}
where using that $\expeLigne{W_{i+1}|\mcf_i} = \tau_{\gamma}(W_i) + \gamma \InftyBound$ by  \Cref{lem:norm_1_eq} and the Markov property,
\begin{equation}
  \label{eq:def_Delta_M_H_moment}
  \Delta M_i = W_{i+1} - \expeLigne{W_{i+1}|\mcf_i} = W_{i+1} -\tau_{\gamma}(W_i) + \gamma \InftyBound \eqsp, \quad H_i =  \tau_{\gamma}(W_i) + \gamma \InftyBound  - W_i \eqsp. 
\end{equation}
Then, using Young's inequality, we get for any $\gamma \in \ocint{0,\bgamma}$ and $k \in\nset$,
\begin{equation}
  \label{eq:bound_moment_sup_proof_0}
\txts  \max_{\ell \in \{0,\ldots,k\}} [W_\ell-W_0]^4 \leq 2^3 \{  \max_{\ell \in \{0,\ldots,k\}} A_{\ell}^4 +  \max_{\ell \in \{0,\ldots,k\}} B_{\ell}^4 \}\eqsp. 
\end{equation}
We now bound the two last terms in the right hand side of this equation. First, by  \Cref{ass:sticky_cont} and Young's inequality, we get for any $\gamma \in \ocint{0,\bgamma}$ and $k \in\nset$,
\begin{equation}
  \txts \expeLigne{  \max_{\ell \in \{0,\ldots,k\}} B_{\ell}^4}  \ \leq \  \tcrw{\expeLigne{k^3 \sum_{i=0}^{k-1} H_i^4}}   \leq \ 2^3 (k\gamma)^4 (1+\Ltt_{\kappa})^4 \{\max_{i\in \{0,\ldots,k-1\}} \expeLigne{W_i^4} + \InftyBound^4\}  \eqsp.  \label{eq:bound_B_moment_sup}
\end{equation}
In addition, by definition \eqref{eq:def_Delta_M_H_moment}, $(\Delta M_i)_{i\in\nset}$ are $(\mcf_i)_{i \in\nset}$-martingale increments. It follows by Burkholder inequality \cite[Theorem 3.2]{burkholder:1973} and Young's inequality that there exists $C_4 \geq 0$ satisfying for any $k \in \nset$ and $\gamma  \in\ocint{0,\bgamma}$,
\begin{equation}
  \label{eq:16}
  \txts \expe{\max_{\ell \in \{0,\ldots,k\}} A^4_{\ell}} \leq C_4 \expeLigne{\{\sum_{i=0}^{k-1}\Delta M_i^2 \}^{2}} \leq C_4 k \sum_{i=0}^{k-1} \expeLigne{\Delta M_i^4}   \eqsp.
\end{equation}
Therefore by \Cref{lem:moment_martinga_incre_proof_cont}, we get that
\begin{equation}
  \txts \expe{\max_{\ell \in \{0,\ldots,k\}} A^4_{\ell}}  \leq \tcrw{C_4 k \sum_{i=0}^{k-1} \expeLigne{\Delta M_i^4}} \leq C_4  (4\sigma^2 k \gamma)^{2}\defEns{ \bfm_{4} +
    2 \sup_{u \geq 0} [u^4 \Phibf(-u)] }   \eqsp,
\end{equation}
where $\bfm_{4}$ is the fourth moment of the standard Gaussian distribution. Combining this result with \eqref{eq:bound_B_moment_sup} and using \Cref{lem:moment_four_iterates_conti_limit_0} in \eqref{eq:bound_moment_sup_proof_0} concludes the proof. 
  \end{proof}  

  \subsection{Proof of \Cref{propo:tight_convergence_continuous}}
\label{sec:proof_tight_cont}

   To show  this result, we use the Komolgorov criteria
  \cite[Corollary 14.9]{kallenberg:2002}: for any $T \geq 0$, there
  exist $C_T \geq 0$ such that for any $n \in \nset$ and $s,t \in \ccint{0,\tcrw{T}}$, $s \leq t$,
  \begin{equation}
    \label{eq:7}
    \expe{\abs{\Wbfn_t-\Wbfn_s}^4} \leq C_T (t-s)^2 \eqsp. 
  \end{equation}
  Note that denoting $k_1^{(n)} = \ceil{s/\gamma_n}$ and $k_2^{(n)} = \floor{t/\gamma_n}$, we have by \eqref{eq:def_wnbf}
  \begin{align}
    \label{eq:9}
    &    \expe{\abs{\Wbfn_t-\Wbfn_s}^4} \\
    &\leq
    \begin{cases}
 (t-s)^{\tcrw{4}}\gamma_n^{-\tcrw{4}} \expeLigne{W_{k^{(n)}_{\tcrw{2}}+1} - W_{k^{(n)}_{\tcrw{2}}}}^4 & \text{ if $k^{(n)}_{\tcrw{2}} < k^{(n)}_{\tcrw{1} }$} \\
          \tcrw{3^3 \tcrw{(t-s)^{\tcrw{4}}\gamma_n^{-\tcrw{4}}}\expeLigne{\{W_{k^{(n)}_2+1}-W_{k^{(n)}_{\tcrw{2}}}\}^4+ \{W_{k^{(n)}_1}-W_{k^{(n)}_1-1}\}^4  } }   \\
          \qquad\qquad\qquad\qquad\qquad\qquad\qquad\qquad\tcrw{+3^3\expeLigne{ \{W_{k^{(n)}_2}-W_{k^{(n)}_1}\}^4}}  & \text{ otherwise}\eqsp.
    \end{cases}
  \end{align}
 \Cref{lem:moment_four_iterates_conti_limit}, \Cref{lem:moment_four_iterates_conti_limit_0} and the Markov property complete the proof.

\subsection{Proof of \Cref{theo:martingale_prob_martinM_N}}
\label{sec:proof_theo:martingale_prob_martinM_N}

Consider the differential operators $\generator, \tilde{\generator}$ defined for any $\psi \in \rmC^2(\rset)$ by
\begin{align}
  \label{eq:def_generator_cont_sticky}
  \generator \psi(w)  &= \{\kappa(w) + \InftyBound\} \psi'(w) + 2 \1_{(0,+\infty)}(w)\sigma^2  \psi''(w) \\
\tilde{    \generator} \psi(w)  &= \{\kappa(w) + \InftyBound\} \psi'(w) + 2 \sigma^2  \psi''(w)  \eqsp,
\end{align}
where $\kappa$ is arbitrary extended on $\rset$. 
Note that $\generator$ is the  \textit{extended} generator associated with
\eqref{eq:def_sticky_cont}.

A crucial step in the proof of  \Cref{theo:martingale_prob_martinM_N} is the following.

\begin{proposition}
  \label{propo:convergence_to_mart_prob}
  Assume \Cref{ass:sticky_cont}. Let $\varphi\in\rmC^{3}(\rset)$, satisfying
  \begin{equation}
  \label{eq:hyp_varphi:propo:convergence_to_mart_prob}
    \sup_{w \in \rset} \{ \absLigne{\varphi}(w)/(1+w^2) + \absLigne{\varphi'}(w)/(1+\absLigne{w}) + \absLigne{\varphi''}(w) + \absLigne{\varphi^{(3)}}(w) \} < \plusinfty \eqsp. 
  \end{equation}
Then,  for any  $N \in \nset$, $(t_1,\ldots,t_N,s,t) \in \coint{0,\plusinfty}^{N+2}$, $0 \leq t_1 \leq \cdots \leq t_N \leq s < t$, $\psi : [0,+\infty)^N \to \rset$, \tcwc{nonnegative}, continuous and bounded, it holds that
  \begin{equation}
    \label{eq:propo:convergence_to_mart_prob_1}
  \lim_{n \to \plusinfty} \expe{\parenthese{\varphi(\Wnbf_t) - \varphi(\Wnbf_s) - \int_{s}^t \generator \varphi(\Wnbf_u) \rmd u} \psi(\Wbfn_{t_1},\ldots,\Wbfn_{t_N})} = 0 \eqsp.
  \end{equation}
If $\varphi''(w) \geq 0$ for any $w \in \rset$, it holds that
  \begin{equation}
        \label{eq:propo:convergence_to_mart_prob_2}
 \limsup_{n \to \plusinfty} \expe{\parenthese{\varphi(\Wnbf_t) - \varphi(\Wnbf_s) - \int_{s}^t \tilde{\generator} \varphi(\Wnbf_u) \rmd u} \psi(\Wbfn_{t_1},\ldots,\Wbfn_{t_N})} \leq 0\eqsp. 
  \end{equation}
\end{proposition}
\begin{proof}
  The proof is postponed to \Cref{sec:proof-crefpr_conv_mart_pb}. 
\end{proof}
Note that while \Cref{propo:convergence_to_mart_prob}-\eqref{eq:propo:convergence_to_mart_prob_1} is in general sufficient to conclude on the convergence of the sequence of processes $\{ (\Wnbf_{t })_{t \geq 0} \, : \, n \in\nset\}$ (see \eg~\cite{ethier:kurtz:1986}), in our setting, it is not enough to complete the proof of \Cref{theo:continuous_limit} since the diffusion coefficient associated with $\generator$ is discontinuous. To circumvent this issue, we adapt to our sequence $\{ (\Wnbf_{t })_{t \geq 0} \, : \, n \in\nset\}$  the same strategy employed in \cite[Proposition 6]{racz:shkolnikov:2015}.

\begin{proposition}
  \label{prop:limit_point_nonnegative}
  Assume \Cref{ass:sticky_cont}. Let $\mubf_{\infty}$ be a limit point of $(\mubf_n)_{n\in\nset}$. Then, $\mubf_{\infty}$-almost everywhere, $\inf_{t \in \coint{0,\plusinfty}} \rmW_t \geq 0$. 
\end{proposition}
\begin{proof}
Without loss of generality, we assume that $(\mubf_n)_{n \in\nset}$ converges to $\mubf_{\infty}$. Since $\omega \mapsto \inf_{t \in \coint{0,\plusinfty}} \omega_t$ is continuous, $\msf = \{ \omega \in \wiener \, : \, \inf_{t \in \coint{0,\plusinfty}} \omega_t \geq 0\}$ is closed. Therefore, by the Portmanteau theorem \cite[Theorem 13.16]{book:klenke:2014}, we obtain that $\mubf_{\infty}(\msf) \geq \limsupn \mubf_n(\msf) = 1$.
\end{proof}

\begin{proof}[Proof of \Cref{theo:martingale_prob_martinM_N}]
  Recall that we denote by $(\mubf_n)_{n \in\nset}$ the sequence of
  distribution on $\wiener$ associated with
  $\{(\Wnbf_t)_{t \geq 0} \, :\, n \in \nset\}$. Let $\mubf_{\infty}$
  be a limit point of this sequence for the convergence in distribution. Without loss of generality, we
  assume that $(\mubf_n)_{n \in\nset}$ converges in distribution to
  $\mubf_{\infty}$. Note that by \Cref{propo:moment_conv_continuous},
  for any continuous function $F: \wiener \to \rset$ such that
  $\abs{F}(\omega) \leq C_T \{1+\sup_{t \in \ccint{0,T}}
  \abs{\omega_t}^{\delta_c}\}$ for $\delta_c \in \coint{0,4}$, $T,C_T \geq 0$, then $F$ is uniformly
  integrable for $(\mubf_n)_{n \in\nset}$ and therefore (see \eg~\cite[Lemma 5.1.7.]{ambrosio:gigli:savare:2008})
  \begin{equation}
    \label{eq:weak_convergence_plus_unif_int_convergence}
    \lim_{n\to \plusinfty} \int_{\wiener} F \rmd \mubf_n = \int_{\wiener}
  F \rmd \mubf_{\infty} \eqsp. 
\end{equation}

We divide then the proof into two parts.
\textbf{First part}: we first show that under $\mubf_{\infty}$, $(\rmM_t)_{t \geq 0}$ is a $(\wienersigma_t)_{t \geq 0}$-martingale.    Considering
    \begin{equation}
    \label{eq:18}
    F_{1} : \omega \mapsto  \parenthese{\varphi_1(\omega_t) - \varphi_1(\omega_s) - \int_{s}^t \generator \varphi_1(\omega_u) \rmd u} \psi(\omega_{t_1},\ldots,\omega_{t_N}) \eqsp,
  \end{equation}
and applying   \Cref{propo:convergence_to_mart_prob}-\eqref{eq:propo:convergence_to_mart_prob_1} to $\varphi_1(w)
    = w$ for any $w \in \rset$, since $\generator
    \varphi_1$ is continuous under \Cref{ass:sticky_cont}, for any $N
    \in \nset$, $(t_1,\ldots,t_N,s,t) \in \coint{0,\plusinfty}^{N+2}$, $0 \leq t_1
    \leq \cdots \leq t_N \leq s < t$, $\psi : \rset_+^N \to
    \rset_{\tcwc{+}}$, continuous and bounded, 
    \begin{equation}
      \label{eq:32}
      \expeW{\mubf_{\infty}}{\parenthese{\rmM_t - \rmM_s} \psi(\rmW_{t_1},\ldots,\rmW_{t_N})} = 0 \eqsp,
    \end{equation}
    where $\expeW{\mubf_{\infty}}{\cdot}$ is the expectation under $\mubf_{\infty}$ on $(\wiener, \wienersigma)$. 
  We obtain by the monotone class theorem and \cite[Theorem 2.3, Chapter 0]{revuz:yor:1994} that the first part of the result holds, \ie~ $(\Mrm_t)_{t \geq 0}$ defined by \eqref{eq:def_Mrm_Nrm} is a $(\wienersigma_t)_{t \geq 0}$-martingale on  $(\wiener,\wienersigma, (\wienersigma_t)_{t \geq 0}, \mubf_{\infty})$. 

\textbf{Second part}:  It remains to  show that under $\mubf_{\infty}$, $(\rmN_t)_{t \geq 0}$ is a $(\wienersigma_t)_{t \geq 0}$-martingale. We first establish setting $\varphi_2(w) = w^2$ for $w \in \rset$, that
  \begin{equation}
    \tilde{\Nrm}_t =   \varphi_2(\rmW_t)-\varphi_2(\rmW_0) - \int_{0}^t\generator \varphi_2(\rmW_u) \rmd u   \eqsp,
  \end{equation}
  is a $(\wienersigma_t)_{t \geq 0}$-submartingale, which easily implies that $(\Nrm_t)_{t \geq 0}$ is a $(\wienersigma_t)_{t \geq 0}$-submartingale. Let $N
    \in \nset$, $(t_1,\ldots,t_N,s,t) \in \coint{0,\plusinfty}^{N+2}$, $0 \leq t_1
    \leq \cdots \leq t_N \leq s < t$, $\psi : \rset_+^N \to
    \rset$, continuous, nonnegative and bounded. Then, consider $F^+_2 =  F^+_{2,1} -  F^+_{2,2}$ on  $\wiener$ with :  
  \begin{align}
    \label{eq:33}
    &  F^+_{2,1} : \omega \mapsto   \defEns{\varphi_2(\omega_t)-\varphi_2(\omega_s) - 2 \int_{s}^t \omega_u (\kappa(\omega_u) + \InftyBound) \rmd u}\psi(\omega_{t_1},\ldots, \omega_{t_{N}}) \\
&    F^+_{2,2} : \omega \mapsto  4 \sigma^2 \defEns{\int_{s}^t \1_{\rset_+^*}(\omega_u)  \rmd u} \psi(\omega_{t_1},\ldots, \omega_{t_{N}})
  \end{align}
  Note that it is easy to check that $F^+_{2,1}$ is continuous and $F^+_{2,2}$ is bounded
  lower semi-continuous  on $\wiener$, \ie~for any $(\omega^n)_{n \in\nset}$
  converging to $\omega^{\infty}$ in $\wiener$ endowed with the uniform convergence on compact set, $\liminf_{n \to \plusinfty} F^+_{2,2}(\omega^{n}) \geq F^+_{2,2}(\omega^{\infty})$.
  Therefore, we obtain by the Portmanteau theorem \cite[Theorem 13.16]{book:klenke:2014} and \eqref{eq:weak_convergence_plus_unif_int_convergence} that
  \begin{equation}
    \label{eq:34}
   \int_{\wiener} F^+_{2,1} \rmd \mubf_{\tcwc{\infty}}  = \limn    \int_{\wiener} F^+_{2,1} \rmd \mubf_n \eqsp, \text{ and }    \int_{\wiener} F^+_{2,2} \rmd \mubf_{\tcwc{\infty}} \leq     \liminfn \int_{\wiener} F^+_{2,2} \rmd \mubf_n \eqsp.
  \end{equation}
  Therefore, \Cref{propo:convergence_to_mart_prob}-\eqref{eq:propo:convergence_to_mart_prob_1} applied with $\varphi \leftarrow \varphi_2$ implies that 
  \begin{equation}
    \label{eq:30}
0 =   \limsupn \int_{\wiener} F^+_2 \rmd \mubf_n \leq \int_{\wiener} F^+_2 \rmd \mubf_{\tcwc{\infty}} \eqsp. 
\end{equation}
Using the same arguments as before, we obtain that under $\mubf_{\infty}$, $(\tilde{\Nrm}_t)_{t \geq 0}$ is a \sloppy$(\wienersigma_t)_{t \geq 0}$-submartingale. Then, it is easy to verify that  $({\Nrm}_t)_{t \geq 0}$ is  a $(\wienersigma_t)_{t \geq 0}$-submartingale. We complete then the proof by showing that $({\Nrm}_t)_{t \geq 0}$ is also a $(\wienersigma_t)_{t \geq 0}$-supermartingale under $\mubf_{\infty}$. To do so, we need the following lemma. 

\begin{lemma}
\label{lem:second_part_proof_cont_theo_main}
  Assume \Cref{ass:sticky_cont}. Then, for any limit point $\mubf_{\infty}$ of $(\mubf_n)_{n \in\nset}$, $\mubf_{\infty}$-almost everywhere, $t \mapsto \qvar{\rmM}_t - 4 \sigma^4 t $ is nonincreasing, where $(\qvar{\rmM}_t)_{ t \geq 0}$ is the quadratic variation of $(\rmM_t)_{t \geq 0}$. 
\end{lemma}
\begin{proof}
  Let $N
    \in \nset$, $(t_1,\ldots,t_N,s,t) \in \coint{0,\plusinfty}^{N+2}$, $0 \leq t_1
    \leq \cdots \leq t_N \leq s < t$, $\psi : \rset_+^N \to
    \rset$, continuous, nonnegative and bounded. Consider now the continuous map
\begin{equation}
  \label{eq:35}
  F_2^- : \omega \mapsto \defEns{\varphi_2(\omega_t)-\varphi_2(\omega_{\tcwc{s}}) - \int_{\tcwc{s}}^t\tilde{\generator} \varphi_2(\omega_u) \rmd u } \psi(\omega_{t_1},\ldots,\omega_{t_N}) \eqsp. 
\end{equation}
Then, by \eqref{eq:weak_convergence_plus_unif_int_convergence} and
\Cref{propo:convergence_to_mart_prob}-\eqref{eq:propo:convergence_to_mart_prob_2},
we get $\limn \int_{\wiener} F_2^- \rmd \mubf_n \leq 0$. Using that  under $\mubf_{\infty}$
$(\rmM_t)_{t \geq 0}$ is a $(\wienersigma_t)_{t \geq 0}$-martingale,
we get that $(\rmM_t^2 - 4 \sigma^2 t)_{t \geq 0}$ is a
$(\wienersigma_t)_{t\geq 0}$ supermartingale. By the
Doob-Meyer decomposition \cite[Theorem 22.5]{kallenberg:2002}, under $\mubf_{\infty}$, there
exists a unique nondecreasing, locally integrable and predictable
process $(\rmC_t)_{t \geq 0}$, such that
$(\rmM_t^2 -4 \sigma^2 t + \rmC_t)_{t \geq 0}$ is a local
$(\wienersigma_t)_{t\geq 0}$-martingale. In addition, under $\mubf_{\infty}$, by  \cite[Theorem 1.8,
Chapter IV]{revuz:yor:1994}, the quadratic variation
$(\qvar{\rmM}_t)_{t \geq 0}$ of $(\rmM_t)_{t \geq 0}$ is a finite  variation process satisfying
$(\rmM_t^2 -\qvar{\rmM}_t)_{ t\geq 0}$ is a
$(\wienersigma_t)_{t\geq 0}$-martingale therefore
$(\rmM_t^2 -4 \sigma^2 t -(\qvar{\rmM}_t - 4 \sigma^2 t) )_{ t\geq 0}$
is a $(\wienersigma_t)_{t\geq 0}$-martingale.   Therefore,
$(\qvar{\rmM}_t - 4 \sigma^2 t - \rmC_t)_{ t\geq 0}$ is a
local $(\wienersigma_t)_{t\geq 0}$-martingale and a finite variation
process. By \cite[Proposition 1.2, Chapter
IV]{revuz:yor:1994}, $\mubf_{\infty}$-almost everywhere, for any
$t \in \coint{0,\plusinfty}$, $ \qvar{\rmM}_t - 4 \sigma^2 t+\rmC_t = 0$, which
completes the proof. 
\end{proof}
 By \Cref{lem:second_part_proof_cont_theo_main}, denoting by $(\qvar{\rmM}_t)_{t \geq 0}$, the quadratic variation of $(\rmM_t)_{t \geq 0}$, see \cite[Theorem 1.8,
Chapter IV]{revuz:yor:1994},  $\mubf_{\infty}$-almost everywhere,
$t \mapsto \qvar{\rmM}_t - 4 \sigma^2 t$ is \tcrw{nonincreasing} and
therefore we get that for any $s,t \in \coint{0,\plusinfty}$, $s \leq t$, $\mubf_{\infty}$-almost everywhere,
\begin{equation}
  \label{eq:proof_cont_limit_second_part_ineq_1}
  \int_{s}^t \1_{\rset_+^*} (\rmW_u) \rmd \qvar{\rmM}_u \leq  4 \sigma^2  \int_{s}^t \1_{\rset_+^*} (\rmW_u) \rmd u \eqsp.
\end{equation}
In addition,  by the occupation times formula \cite[Corollary 1.6, Chapter VI]{revuz:yor:1994} applied twice and \Cref{prop:limit_point_nonnegative}, $\mubf_{\infty}$-almost everywhere,  
\begin{equation}
  \label{eq:37}
  \qvar{\rmM}_t = \int_{0} ^t \rmd \qvar{\rmM}_u = \int_{0} ^t \1_{\rset_+}(\rmW_u) \rmd \qvar{\rmW}_u = \int_{\rset_+} \rmL^a_t \rmd a = \int_{\rset_+^*} \rmL^a_t \rmd a = \int_{0} ^t  \1_{\rset_+^*}(\rmW_u) \rmd \qvar{\rmM}_u \eqsp .
\end{equation}
Using this result and  \eqref{eq:proof_cont_limit_second_part_ineq_1}, we get  that $\qvar{M}_t -\qvar{M}_s \leq 4 \sigma^2 \int_{s}^t \1_{\rset_+^*}(\rmW_u) \rmd u$, for any $s,t \in\coint{0,\plusinfty}$, $s \leq t$. Therefore since $(\rmM^2_t - \qvar{M}_t)_{ t\geq 0}$ is a $(\wienersigma_t)_{t \geq 0}$-martingale under $\mubf_{\infty}$, we conclude that $(\rmN_t)_{t \geq 0}$ is a $(\wienersigma_t)_{t \geq 0}$-supermartingale which completes the proof.
\end{proof}

\subsection{Proof of \Cref{propo:convergence_to_mart_prob}}
\label{sec:proof-crefpr_conv_mart_pb}

We preface the proof by the following technical lemma.

\begin{lemma}
  \label{lem:moment_martinga_incre_proof_cont_quad_approx}
  Assume \Cref{ass:sticky_cont}. Then, for any $q \in \coint{1,\plusinfty}$, we have
\begin{enumerate}[wide, labelwidth=!, labelindent=0pt, label=(\alph*)]
\item  \label{lem:moment_martinga_incre_proof_cont_quad_approx_item_1}
  for any $\gamma \in \ocint{0,\bgamma}$ and $w_0 \in \coint{0,\plusinfty}$, 
  \begin{align}
    \label{eq:14_approx_quad}
&     \txts -2 (4\sigma^2 \gamma)^{q/2} \int_{\btau_{\gamma}^{\infty}(w_0)}^{\plusinfty} \abs{g}^q \varphibf(g) \rmd g \\ 
 \qquad   &  \qquad\txts  \leq  \int_{\rset_+}   \abs{w_1-\tau_{\gamma}(w_0)-\gamma \InftyBound}^q Q_{\gamma}(w_0,\rmd w_1)- (4\sigma^2 \gamma)^{q/2} \bfm_{q} \ \leq \ 0 \eqsp,
  \end{align}
  where $\btau_{\gamma}^{\infty}(w_0) = \{\tau_{\gamma}(w_0) +\gamma \InftyBound\}/(2 \sqrt{\sigma^2 \gamma})$, $Q_{\gamma}$ is defined by  \eqref{def_q_gamma}, $\bfm_{q}$ is the $q$-th moment of the standard Gaussian distribution and $\varphibf$ is its probability density function;
\item  \label{lem:moment_martinga_incre_proof_cont_quad_approx_item_2} for any $\gamma \in \ocint{0,\bgamma}$, 
  \begin{equation}
    \label{eq:14_approx_quad_2}
    \txts   \int_{\rset_+}   \abs{w_1-\gamma \InftyBound}^q Q_{\gamma}(0,\rmd w_1) \leq   3 (  \gamma \InftyBound)^q +   (2  \gamma \InftyBound)^q + \tcrw{q}\bfm_{q-1}(4 \sigma^2 \gamma)^{q/2} \gamma^{1/2} \InftyBound/\sigma  \eqsp.
  \end{equation}
\end{enumerate}
\end{lemma}
\begin{proof}
  \begin{enumerate}[wide, labelwidth=!, labelindent=0pt, label=(\alph*)]
      \item 
  Let $w_0 \in\coint{0,\plusinfty}$ and $\gamma \in\ocint{0,\bgamma}$.
  By definition \eqref{def_q_gamma} and \eqref{eq:def_bpg}, we have setting $\bar{\tau}_{\gamma}^{\infty}(w_0) = \{\tau_{\gamma}(w_0) +\gamma \InftyBound\}/(2 \sqrt{\sigma^2 \gamma})$, 
  \begin{equation}
    \begin{aligned}
    & \qquad  \int_{\rset_+}   \abs{w_1-\tau_{\gamma}(w_0)-\gamma \InftyBound}^q Q_{\gamma}(w_0,\rmd w_1) = (4 \sigma^2 \gamma)^{q/2} \int_{-\infty}^{\bar{\tau}_{\gamma}^{\infty}(w_0)} \abs{g}^q \varphibf(g) \rmd g \\
         &   \qquad\qquad\qquad\qquad\qquad\qquad\qquad -(4 \sigma^2 \gamma)^{q/2} \int_{\bar{\tau}_{\gamma}^{\infty}(w_0)}^{\plusinfty} \abs{g-\tcrw{2\bar{\tau}_{\gamma}^{\infty}(w_0)}}^q\varphibf(g) \rmd g \\
         &\qquad\qquad\qquad\qquad\qquad\qquad\qquad+ \int_{\rset} \absLigne{\tau_{\gamma}(w_0)+\gamma \InftyBound}^q
\varphibf( 2 \bar{\tau}_{\gamma}^{\infty}(w_0)-g) \wedge \varphibf(g) 
    \rmd g \eqsp. \end{aligned} \label{eq:14_approx_quad_proof_0}
  \end{equation}
  Therefore, we obtain that
    \begin{align}
    \label{eq:14_approx_quad_proof}
&    \txts   \int_{\rset_+}   \abs{w_1-\tau_{\gamma}(w_0)-\gamma \InftyBound}^q Q_{\gamma}(w_0,\rmd w_1)- (4\sigma^2 \gamma)^{q/2} \bfm_{q}  \\
      &\qquad    \txts =  - (4 \sigma^2 \gamma)^{q/2} \int_{\bar{\tau}_{\gamma}^{\infty}(w_0)}^{\plusinfty} \abs{g}^q\varphibf(g) \rmd g
        - (4 \sigma^2 \gamma)^{q/2} \int_{\bar{\tau}_{\gamma}^{\infty}(w_0)}^{\plusinfty} \abs{2\bar{\tau}_{\gamma}^{\infty}(w_0) - g}^q\varphibf(g) \rmd g \\
      & \qquad \qquad  
    \txts    +  2 (4\sigma^2 \gamma)^{q/2} \{\btau_{\gamma}^{\infty}(w_0)\}^q \int_{\bar{\tau}_{\gamma}^{\infty}(w_0)}^{\plusinfty} \varphibf(g) \rmd g \eqsp.
    \end{align}
    Using that for $g \in \cointLigne{\bar{\tau}_{\gamma}^{\infty}(w_0), \plusinfty}$, $\{\btau_{\gamma}^{\infty}(w_0)\}^q \leq 2^{-1}[\abs{g}^q + \absLigne{2\bar{\tau}_{\gamma}^{\infty}(w_0) - g}^q]$ \tcrw{and $\absLigne{2\bar{\tau}_{\gamma}^{\infty}(w_0)-g}\leq \absLigne{g}$} completes the proof.
  \item By \eqref{eq:14_approx_quad_proof_0} and since $\tau_{\gamma}(0)=0$ and  $\btau_{\gamma}^{\infty}(0) = \gamma \InftyBound/(2 \sqrt{\sigma^2 \gamma})$, 
    \begin{align}
      \label{eq:31}
&      \txts  \abs{ \int_{\rset_+}   \abs{w_1-\tau_{\gamma}(0)-\gamma \InftyBound}^q Q_{\gamma}(0,\rmd w_1)} \leq  2 (4 \sigma^2 \gamma)^{q/2}\int_{0}^{ \frac{\gamma \InftyBound}{2 \sqrt{\sigma^2 \gamma}}} \abs{g}^{\tcrw{q}} \varphibf(g) \rmd g
      \\
& \txts \qquad \qquad        +  (4 \sigma^2 \gamma)^{q/2}\int_{-\infty}^{- \frac{\gamma \InftyBound}{2 \sqrt{\sigma^2 \gamma}}} [\abs{g  }^{\tcrw{q}}  - \absLigne{2\bar{\tau}_{\gamma}^{\infty}(w_0) + g}^q ]\varphibf(g) \rmd g
      + (\gamma \InftyBound)^q\\
      & \txts \qquad \qquad    
\leq  2 (4 \sigma^2 \gamma)^{q/2}\int_{0}^{ \frac{\gamma \InftyBound}{2 \sqrt{\sigma^2 \gamma}}} \abs{g}^{\tcrw{q}} \varphibf(g) \rmd g + (4 \sigma^2 \gamma)^{q/2}\int_{-2 \frac{\gamma \InftyBound}{2 \sqrt{\sigma^2 \gamma}}}^{- \frac{\gamma \InftyBound}{2 \sqrt{\sigma^2 \gamma}}} \abs{g}^{\tcrw{q}}\varphibf(g) \rmd g
      \\
& \txts \qquad \qquad           +  (4 \sigma^2 \gamma)^{q/2}\int_{-\infty}^{- \frac{\gamma \InftyBound}{ \sqrt{\sigma^2 \gamma}}} \parentheseDeux{(-g )^q  - \defEns{-\frac{\gamma \InftyBound}{ \sqrt{\sigma^2 \gamma}} - g}^q }\varphibf(g) \rmd g
      + (\gamma \InftyBound)^q
        \eqsp.
    \end{align}
    Using that $(4 \sigma^2 \gamma)^{q/2}\int_{0}^{\gamma \InftyBound/(2 \sqrt{\sigma^2 \gamma})} \abs{g}^{\tcrw{q}} \varphibf(g) \rmd g  \leq (\gamma \InftyBound)^q$ and $a^q - (a-h)^q \leq \tcrw{q}a^{q-1}h$ for $a,h \geq 0$, $a\geq h$, completes the proof.
  \end{enumerate}
\end{proof}

\begin{proof}[Proof of \Cref{propo:convergence_to_mart_prob}]

\textbf{Proof of \eqref{eq:propo:convergence_to_mart_prob_1}.}
  Let  $\varphi\in\rmC^{\infty}(\rset^d)$ satisfying \eqref{eq:hyp_varphi:propo:convergence_to_mart_prob}, $(s,t) \in \coint{0,\plusinfty}^{2}$, $ s < t$.
Note that we only need to show that 
\begin{equation}
  \label{eq:17}
  \lim_{n \to \plusinfty} \expe{\abs{\CPE{\varphi(\Wnbf_t) - \varphi(\Wnbf_s) - \int_{s}^t \generator \varphi(\Wnbf_u) \rmd u}{\mcg_s^{(n)}}}}\tcrw{=0}\eqsp,
\end{equation}
setting for any $n \in \nset$, $u \in \coint{0,\plusinfty}$, $\mcg_u^{(n)} = \mcf^{(n)}_{\ceil{u/\gamma_n}}$, where $(\mcf^{(n)}_k)_{k \in \nset}$ is  the filtration corresponding to $(\Wn_k)_{k \in \nset}$.

Define $k_1^{(n)} = \ceil{t/\gamma_n}$ and $k_2^{(n)} = \ceil{s/\gamma_n}$ and  consider the following decomposition
\begin{align}
  \label{eq:decomposition-main-proof-martingale-pb}
  &  \CPE{\varphi(\Wnbf_t) - \varphi(\Wnbf_s) - \int_{s}^t \generator \varphi(\Wnbf_u) \rmd u}{\mcg_s^{(n)}} = \CPE{A_1^{(n)}+A_2^{(n)}+A_3^{(n)}}{\mcg_s^{(n)}} \\
  &  A_1^{(n)}  = \varphi(\Wnbf_t)-\varphi(\Wn_{k_1^{(n)}}) - \{\varphi(\Wnbf_s) - \varphi(\Wn_{k_2^{(n)}})\} \\
  & A_2^{(n)} = -\int_{s}^t \generator \varphi(\Wnbf_u) \rmd u + \gamma_n \sum_{k=k_2^{(n)}}^{k_1^{(n)}-1} \generator \varphi(\Wn_k)  \\
  & A_3^{(n)} = \varphi(\Wn_{k_1^{(n)}}) - \varphi(\Wn_{k_2^{(n)}}) - \gamma_n\sum_{k=k_2^{(n)}}^{k_1^{(n)}-1} \generator \varphi(\Wn_k) \eqsp.
\end{align}

We deal with these three terms separately.

First since $\varphi$ satisfies \eqref{eq:hyp_varphi:propo:convergence_to_mart_prob}, by the fundamental theorem of calculus, there exists $C \geq 0$ such that for any $w_0,w_1 \in \rset$, $\abs{\varphi(w_1) - \varphi(w_0)} \leq C\parenthese{\tcrw{1+}\max(\abs{w_0},\abs{w_1})} \abs{w_0-w_1}$. By \eqref{eq:def_wnbf}\tcrw{, Cauchy–Schwarz inequality and \Cref{lem:moment_four_iterates_conti_limit}}, we get that there exists $C \geq 0$ such that for any $n \in \nset$, 
\begin{equation}
  \label{eq:20}
\txts\tcrw{\expeLigne{\absLigne{A_1^{(n)}}}^4\leq \expeLigne{\absLigne{A_1^{(n)}}^2}^2\leq C \gamma_n^2 \{\max_{i \in \{k_1^{(n)}-1,k_1^{(n)},k_2^{(n)}-1,k_2^{(n)}\}}\expeLigne{\absLigne{\Wn_{i}}^4} +1\}^2} \eqsp. 
\end{equation}
This implies by \tcrw{\Cref{lem:moment_four_iterates_conti_limit_0} that}
  \begin{equation}
  \label{eq:lim-A-1-pb-nartingale}
\lim_{n \to \plusinfty} \expeLigne{    \absLigne{A_1^{(n)}}}  = 0 \eqsp. 
\end{equation}

Regarding $A_2^{(n)}$, we consider the decomposition,%setting $k_u^{(n)}  = \floor{u/\gamma_n}$,
\begin{align}
  \label{eq:19}
  A_2^{(n)} &= A_{2,1}^{(n)} + A_{2,2}^{(n)} \eqsp, \\
  A_{2,1}^{(n)} & =  \tcrw{-}\int^{k_2^{(n)} \gamma_n}_s\generator \varphi(\Wnbf_u) \rmd u + \int_t^{k_1^{(n)}\gamma_n} \generator \varphi(\Wnbf_u) \rmd u  \\
A_{2,2}^{(n)} &=   -\sum_{k=k_2^{(n)}}^{k_1^{(n)}-1} \int_{k \gamma_n}^{(k+1)\gamma_n} \{\generator \varphi(\Wnbf_u) - \generator \varphi (\Wn_{k})\}\rmd u \eqsp. 
\end{align}

Since $\varphi$ satisfies \eqref{eq:hyp_varphi:propo:convergence_to_mart_prob} \tcrw{and by \Cref{lem:moment_four_iterates_conti_limit_0}}, we get that $\lim_{n\to \plusinfty} \expeLigne{\absLigne{A_{2,1}^{(n)}}} = 0$. In addition, we have by definition of $\generator$ \eqref{eq:def_generator_cont_sticky} that for any $n \in \nset$, $k \in \{k_2^{(n)}, \ldots,k_1^{(n)}-1\}$, $u \in \ooint{k \gamma_n, (k+1) \gamma_n}$, 
\begin{align}
  &  \txts \abs{  \generator \varphi(\Wnbf_u) - \generator \varphi (\Wn_{k})} \leqslant B_{u,k}^{\tcrw{(n)}} +  2 \1_{\tcrw{\msa_k}} \sigma^2 \sup_{\rset}\abs{\varphi''} \\
& \txts  B_{u,k}^{(n)} =  (\absLigne{\kappa(\Wn_k)} + \InftyBound)\absLigne{\varphi'(\Wnbf_u)-\varphi'(\Wn_k)} \\
&\quad \txts  +   \absLigne{\varphi'(\Wnbf_u)}\absLigne{\kappa(\Wnbf_u)- \kappa(\Wn_k)}
 + 2 \1_{\tcrw{\msa_k^{\comp}}}\sigma^2\absLigne{ \varphi''(\Wnbf_u)-\varphi''(\Wn_k)}  \eqsp,
\end{align}
where
$\tcrw{\msa_k= \{\Wn_k =0, \Wn_{k+1} \neq 0\}}$. Note that using that $\varphi'$ and $\varphi''$ are
Lipschitz and $\sup_{\tilde{w} \in \coint{0,\plusinfty}} \absLigne{\varphi'}(w)/(1+\abs{w}) < \plusinfty$ by \eqref{eq:hyp_varphi:propo:convergence_to_mart_prob},  \Cref{ass:sticky_cont}, \eqref{eq:def_wnbf}\tcrw{, Cauchy–Schwarz inequality, \Cref{lem:moment_four_iterates_conti_limit_0} and \Cref{lem:moment_four_iterates_conti_limit}}, we get that there exists $C \geq 0$ such that for any $n \in\nset$, $k \in \{k_2^{(n)}, \ldots,k_1^{(n)}-1\}$, $u \in \ooint{k \gamma_n, (k+1) \gamma_n}$, 
\begin{equation}
  \label{eq:21}
\txts  \tcrw{\expeLigne{\absLigne{  B_{u,k}^{(n)} }}^4\leq \expeLigne{\absLigne{  B_{u,k}^{(n)} }^2}^2 \leq C \gamma_n^2 \{\expeLigne{\absLigne{\Wn_0}^4}+1\}^2} \eqsp,
\end{equation}
which implies that
\begin{equation}
  \label{eq:22}
  \lim_{n \to \plusinfty} \sum_{k=k_2^{(n)}}^{k_1^{(n)}-1} \int_{k \gamma_n}^{(k+1)\gamma_n} \expe{\abs{B_{u,k}^{(n)}} } \rmd u = 0 \eqsp. 
\end{equation}
To conclude that $  \lim_{n \to \plusinfty} \expeLigne{\absLigne{A_2^{(n)}}} = 0$,
% \begin{equation}
%   \label{eq:23}
%  \eqsp,
% \end{equation}
it remains to show that 
\begin{equation}
  \label{eq:lim_A_2_pb_martingale_0}
  \lim_{n \to \plusinfty} \gamma_n\sum_{k=k_2^{(n)}}^{k_1^{(n)}-1}\expe{\1_{\tcrw{\msa_k}}} = 0 \eqsp.
\end{equation}
Note that using that by definition, $(\Wn_{k})_{k \in\nset}$ is a
Markov chain with  Markov kernel $Q_{\gamma_n}$ \eqref{def_q_gamma}, the Markov property \tcrw{and \Cref{lem:0_proba}} implies that for any $n  \in \nset$ and 
$k \in \{k_2^{(n)}, \ldots, k_1^{(n)}-1\}$,
\begin{equation}
  \label{eq:lim_A_2_pb_martingale_1}
  \tcrw{\expe{\1_{\tcrw{\msa_k}}} =  \proba{\Wn_{k}=0, \Wn_{k+1}\neq 0} \leq  1-  2 \Phibf[-\InftyBound\sqrt{\gamma_n}/(2 \sigma)]\eqsp.}
\end{equation}
Since $1 -2 \Phibf(-u) \leq u$ for any $u \in \coint{0,\plusinfty}$, we get that there exists $C \geq 0$ such that 
for any $n \in \nset$,
$\gamma_n\sum_{k=k_2^{(n)}}^{k_1^{(n)}-1} 1-2
\Phibf[-\InftyBound\sqrt{\gamma_n}/(2 \sigma)] \leq
C((t-s)+\gamma_n)\gamma_n^{1/2}$ and therefore
\begin{equation}
  \label{eq:lim_A_2_pb_martingale_2}
  \lim_{n \to \plusinfty} \gamma_n\sum_{k=k_2^{(n)}}^{k_1^{(n)}-1} \{1-2
\Phibf[-\InftyBound\sqrt{\gamma_n}/(2 \sigma)]\} = 0\eqsp.
\end{equation}
% Regarding the second term in
% \eqref{eq:lim_A_2_pb_martingale_1}, consider the sequence of
% measurable functions defined for any $n \in\nset$, $\omega \in \Omega$,
% $k \in \nset$ by
% $$f_n(\omega, k) = \gamma_n \1_{\{k_2^{(n)},\ldots,k_1^{(n)}-1\}}(k)
% \1_{\rset_+^*}(\Wn_{k})
% \Phibf[-\tau_{\gamma}^{\infty}(\Wn_k)/\{2(\sigma^2 \gamma_n)^{\half}\}]\eqsp,$$
% on the measure space
% $(\Omega \times \nset, \mcf \otimes 2^{\nset}, \PP \otimes
% \nu_{\mathrm{c}})$, where $2^{\nset}$ is the power set of $\nset$ and
% $\nu_{\mathrm{c}}$ is the counting measure on $\nset$. Note that
% $\PP \otimes \nu_{\mathrm{c}}$ almost everywhere,
% $\lim_{n \to \plusinfty} f_n(\omega,k) = 0$ and in addition,
%  $\sum_{k\in\nset} \int_{\Omega}\tilde{f}_n(\omega,k) \rmd \PP(\omega) \leq (t-s) + \gamma_n$. Therefore by the Lebesgue dominated convergence theorem, we obtain that $\lim_{n \to \plusinfty} \sum_{k\in\nset} \int_{\Omega}\tilde{f}_n(\omega,k) \rmd \PP(\omega) =0$ which implies by definition,
% \begin{equation}
%   \label{eq:26}
% \lim_{n \to \plusinfty}  2 \gamma_n\sum_{k=k_2^{(n)}}^{k_1^{(n)}-1} \expe{\1_{\rset_+^*}(\Wn_{k}) \Phibf[-\tau_{\gamma_n}^{\infty}(\Wn_k)/\{2(\sigma^2 \gamma_n)^{\half}\}]}  = 0 \eqsp.
% \end{equation}
This result combined with \eqref{eq:lim_A_2_pb_martingale_1} in \eqref{eq:lim_A_2_pb_martingale_0} shows that
\begin{equation}
  \label{eq:lim-A-2-pb-nartingale}
  \lim_{n\to \plusinfty} \expe{\absLigne{A_2^{(n)}}} = 0 \eqsp. 
\end{equation}

Finally we deal with $A_3^{(n)}$ from the decomposition
\begin{align}
 A_3^{(n)} = \sum_{k=k_2^{(n)}}^{k_1^{(n)}-1} \varphi(\Wn_{k+1}) - \varphi(\Wn_{k}) - \gamma_n \generator \varphi(\Wn_k) \eqsp.
\end{align}
Set for any $k \in \{k_2^{(n)},\ldots,k_1^{(n)} -1 \}$, $\Delta \Wn_{\tcrw{k+1}} = \Wn_{k+1} - \Wn_{k}$. Using that $\varphi$ is three times continuously differentiable, we
get by Taylor's theorem with Lagrange reminder, that for any
$n  \in \nset$, $k \in \{k_2^{(n)} , \ldots , k_1^{(n)} - 1\}$, there exists
$u_k \in \ccint{0,1}$ satisfying
\begin{align}
  \varphi(\Wn_{k+1}) - \varphi(\Wn_{k})  &= \varphi'(\Wn_k)\Delta \Wn_{\tcrw{k+1}} 
  + (\varphi''(\Wn_k)/2)\{\Delta \Wn_{\tcrw{k+1}}\}^2 
  \\&\qquad\qquad\qquad\qquad+ 6^{-1} \varphi^{(3)}(u_k \Wn_{k+1} +(1-u_k)\Wn_k)\{\Delta \Wn_{\tcrw{k+1}}\}^3  \eqsp.
\end{align}
It follows from  the definition \eqref{eq:def_r}, \Cref{ass:sticky_cont} 
 and Young's inequality, setting $\kappa^{\infty}(w) = \kappa(w) + \InftyBound$ that for any
$n  \in \nset$, $k \in \{k_2^{(n)} , \ldots , k_1^{(n)} - 1\}$,
\begin{equation}
  \label{eq:25}
 \absLigne{\Delta \Wn_{\tcrw{k+1}}}^3\leq 4 \{\gamma_n^3\absLigne{\kappa^{\infty}(\Wn_k)}^3 + \absLigne{ \Wn_{k+1} - (\tau_{\gamma_n}(\Wn_k) + \tcrw{\gamma_n}\InftyBound)}^3 \}\eqsp.
\end{equation}
It follows then using the definition of $\generator$
\eqref{eq:def_generator_cont_sticky}, \eqref{eq:hyp_varphi:propo:convergence_to_mart_prob},
\tcrw{\Cref{lem:moment_four_iterates_conti_limit_0}} and
\Cref{lem:moment_martinga_incre_proof_cont} that
\begin{align}
  \label{eq:go-to-A-4}
  &\lim_{n \to \plusinfty} \expeLigne{\absLigne{A_3^{(n)} - A_4^{(n)}}}  = 0 \eqsp, \quad \text{ where }  A_4^{(n)} = A_{4,1}^{(n)} + A_{4,2}^{(n)} \eqsp,\\
  & A_{4,1}^{(n)} = \sum_{k=k_2^{(n)}}^{k_1^{(n)}-1} [\varphi'(\Wn_k)\Delta \Wn_{k+1} - \gamma_n \varphi'(\Wn_k) \kappa^{\infty}(\Wn_k) ]\\
  & A_{4,2}^{(n)} = \sum_{k=k_2^{(n)}}^{k_1^{(n)}-1}[ (\varphi''(\Wn_k)/2)\{\Delta \Wn_{k+1} \}^2 - 2 \gamma_n \sigma^2 \varphi''(\Wn_k)\1_{\rset_+^*}(\Wn_k)] \eqsp.
\end{align}
Note that by \Cref{lem:norm_1_eq} and the Markov property, we have that
for any $n \in \nset$ and $k \in \{k_2^{(n)} , \ldots , k_1^{(n)} - 1\}$, $ \varphi'(\Wn_k)\CPELigne{\Delta \Wn_{k+1}} {\mcf_{k}^{(n)}}- \gamma_n \varphi'(\Wn_k) \kappa^{\infty}(\Wn_k)  = 0$, which implies that 
% \begin{equation}
%   \label{eq:28}
%   \varphi'(\Wn_k)\CPELigne{\Delta \Wn_{k+1}} {\mcf_{k}^{(n)}}- \gamma_n \varphi'(\Wn_k) \kappa^{\infty}(\Wn_k)  = 0 \eqsp,
% \end{equation}
\begin{align}
  \label{eq:exp_condi-A-4-1}  & \qquad   \CPELigne{A_{4,1}^{(n)}}{\mcg^{(n)}_s} =\CPELigne{A_{4,1}^{(n)}}{\mcf_{k_2^{(n)}}^{(n)}}\\
  & \qquad \qquad  = \txts \sum_{k=k_2^{(n)}}^{k_1^{(n)}-1} \CPELigne{ \varphi'(\Wn_k)\CPELigne{\Delta \Wn_{k+1}} {\mcf_{k}^{(n)}}- \gamma_n \varphi'(\Wn_k) \kappa^{\infty}(\Wn_k) }{\mcf_{k_2^{(n)}}^{(n)}} = 0 \eqsp. 
\end{align}
We now show that $\lim_{n \to \plusinfty} \expeLigne{\absLigne{\CPELigne{A_{4,2}^{(n)}}{\mcg^{(n)}_s}}} = 0$ using the decomposition
\begin{align}
  \label{eq:decomposition-A-4-2}
  & A_{4,2}^{(n)} = A_{4,2,1}^{(n)}+ A_{4,2,2}^{(n)}+A_{4,2,3}^{(n)} \eqsp,\\
A_{4,2,1}^{(n)}& =\sum_{k=k_2^{(n)}}^{k_1^{(n)}-1} \1_{\rset_+^*}(\Wn_k) (\varphi''(\Wn_k)/2)[\{\bar{\Delta} \Wn_{k+1} \}^2 - 4 \sigma^2 \gamma_n] \\
       & A_{4,2,2}^{(n)} =\sum_{k=k_2^{(n)}}^{k_1^{(n)}-1} \1_{\{0\}}(\Wn_k) (\varphi''(0)/2)\{ \bar{\Delta} \Wn_{k+1} \}^2 \\
  &A_{4,2,3}^{(n)} =\sum_{k=k_2^{(n)}}^{k_1^{(n)}-1}  (\varphi''(\Wn_k)/2)[\{\Delta \Wn_{k+1} \}^2 - \{\bar{\Delta} \Wn_{k+1} \}^2]\eqsp ,
     \end{align}
     where $\bar{\Delta} \Wn_{k+1}=\Wn_{k+1}-(\tau_{\gamma_n}(\Wn_k(\omega)) + \tcrw{\gamma_n}\InftyBound)$ and $M\in \rset_+^*$.

%  $(\tilde{f}_n)_{n \in \nset}$ is a sequence of
% measurable functions defined $(\Omega \times \nset, \mcf \otimes 2^{\nset}, \PP \otimes
% \nu_{\mathrm{c}})$. Note that
% $\PP \otimes \nu_{\mathrm{c}}$ almost everywhere,
% $\lim_{n \to \plusinfty} \tilde{f}_n(\omega,k) = 0$ and in addition,
%  $$\sum_{k\in\nset} \int_{\Omega}\tilde{f}_n(\omega,k) \rmd \PP(\omega) \leq [(t-s) + \gamma_n] \int_{\rset} \abs{g}^q \varphibf(g) \rmd g \eqsp.$$ Therefore, we obtain that $\lim_{n \to \plusinfty} \sum_{k\in\nset} \int_{\Omega}\tilde{f}_n(\omega,k) \rmd \PP(\omega) =0$ by the Lebesgue dominated convergence theorem,  which implies by \eqref{eq:proof_A_4_2_1-a} that
%  \begin{equation}
%    \label{eq:lim-A-4-2-1}
% \lim_{n \to \plusinfty}   \expeLigne{  \absLigne{\CPELigne{A_{4,2,1}^{(n)}}{\mcg_s^{(n)}}}} = 0 \eqsp.
% \end{equation}

We first consider \tcrw{$A_{4,2,2}$}. By
\Cref{lem:moment_martinga_incre_proof_cont_quad_approx}-\ref{lem:moment_martinga_incre_proof_cont_quad_approx_item_2}
and the Markov property, we have
\begin{equation}
  \label{eq:29}
  \absLigne{\CPELigne{\tcrw{A_{4,2,2}}^{(n)}}{\mcg_s^{(n)}}} = \absLigne{ \CPELigne{\tcrw{A_{4,2,2}}^{(n)}}{\mcf_{k_2^{(n)}}^{(n)}}} \leq (\varphi''(0)/2)[(t-s)+\gamma_n][7 \gamma_n \InftyBound^2 + \tcrw{8} \sigma \gamma_n^{1/2} \InftyBound ] \eqsp,
\end{equation}
showing that 
\begin{equation}
  \label{eq:limit_4_2_2}
  \lim_{n \to \plusinfty}   \expeLigne{  \absLigne{\CPELigne{\tcrw{A_{4,2,2}^{(n)}}}{\mcg_s^{(n)}}}} = 0 \eqsp .
\end{equation}

Regarding the third term, we first have 
\begin{align}
  \expeLigne{\absLigne{A_{4,2,3}^{(n)}}}\leq\sup_{\rset}\{\abs{ \varphi''}/2 \}\expe{\sum_{k=k_2^{(n)}}^{k_1^{(n)}-1}  \abs{\CPE{\{\Delta \Wn_{k+1} \}^2 - \{\bar{\Delta} \Wn_{k+1} \}^2}{\mcf_{k}^{(n)}}}} \eqsp .
  \label{eq:A_4_2_4_decomposition}
\end{align}
In addition by \Cref{lem:norm_1_eq}, the Markov property and \Cref{ass:sticky_cont}, we have for any $k\in \nset$,
\begin{align}
  &\abs{\CPE{\{\Delta \Wn_{k+1} \}^2 - \{\bar{\Delta} \Wn_{k+1} \}^2}{\mcf_{k}^{(n)}}}\\
  &\qquad\qquad\qquad\qquad\qquad\qquad\qquad=\abs{\gamma_n\kappa^{\infty}(\Wn_k)\CPE{2\Delta \Wn_{k+1} - \gamma_n\kappa^{\infty}(\Wn_k)}{\mcf_{k}^{(n)}}}\\
  &\qquad\qquad\qquad\qquad\qquad\qquad\qquad=(\gamma_n\kappa^{\infty}(\Wn_k))^2\\
  &\qquad\qquad\qquad\qquad\qquad\qquad\qquad\leq 2\gamma_n^2 \parentheseDeux{(1+\Ltt_\kappa)^2\{\Wn_k\}^2+\InftyBound^2} \eqsp .
\end{align}
Then by \Cref{lem:moment_four_iterates_conti_limit_0}, there exists $C \geq 0$ such that for any $k\leq k_1^{(n)}-1$ and $n \in\nsets$,
\begin{equation}
  \expe{\abs{\CPE{\{\Delta \Wn_{k+1} \}^2 - \{\bar{\Delta} \Wn_{k+1} \}^2}{\mcf_{k}^{(n)}}}}\leq 2\gamma_n^2 \parentheseDeux{(1+\Ltt_\kappa)^2\rme^{C t/2} \sqrt{\expeLigne{W_0^4} + 1}+\InftyBound^2} \eqsp .
\end{equation}
Combining this with \eqref{eq:A_4_2_4_decomposition} and using that $\gamma_n(k_1^{(n)}-k_2^{(n)})\leq (t-s+\gamma_n) $ we obtain
\begin{equation}
  \label{eq:limit_A_4_2_3}
  \lim_{n \to \plusinfty} \expeLigne{\absLigne{A_{4,2,3}^{(n)}}}=0 \eqsp .
\end{equation}

     Regarding the first term in \eqref{eq:decomposition-A-4-2}, we
     consider for $M >0$, the following decomposition
     \begin{align}
       A_{4,2,1}^{(n)} & = D^{(n)}_1(M) + D_2^{(n)}(M) \eqsp,
     \end{align}
     where
     \begin{align}
         & D^{(n)}_1(M) =\sum_{k=k_2^{(n)}}^{k_1^{(n)}-1} \1_{\ooint{0,M\sqrt{\gamma_n}}}(\Wn_k) (\varphi''(\Wn_k)/2)[\{\bar{\Delta} \Wn_{k+1} \}^2 - 4 \sigma^2 \gamma_n] \\
  &D_2^{(n)}(M)=\sum_{k=k_2^{(n)}}^{k_1^{(n)}-1} \1_{\coint{M\sqrt{\gamma_n},\plusinfty}}(\Wn_k) (\varphi''(\Wn_k)/2)[\{\bar{\Delta} \Wn_{k+1} \}^2 - 4 \sigma^2 \gamma_n] \eqsp.
     \end{align}
Then, by \Cref{lem:moment_martinga_incre_proof_cont_quad_approx}-\ref{lem:moment_martinga_incre_proof_cont_quad_approx_item_1}
, we have using the Markov property that
\begin{align}
  \label{eq:proof_A_4_2_1-a}
  \expeLigne{\absLigne{D_1^{(n)}(M)}}&\leq \sup_{\rset}\{\abs{ \varphi''}/2 \}\expe{\sum_{k=k_2^{(n)}}^{k_1^{(n)}-1} \1_{\ooint{0,M\sqrt{\gamma_n}}}(\Wn_k) \abs{\CPELigne{\{\Delta \Wn_{k+1} \}^2}{\mcf_{k}^{(n)}} - 4 \sigma^2 \gamma_n}}\\
  & \leq 4\gamma_n \sigma^2  \sup_{\rset}\{\abs{ \varphi''} \} \sum_{k=k_2^{(n)}}^{k_1^{(n)}-1} \expe{\1_{\ooint{0,M\sqrt{\gamma_n}}}(\Wn_k) \Upsilon(\btau_{\gamma_n}^{\infty}(\Wn_k(\omega)))}\eqsp,
  \label{eq:A_4_2_1_decomposition}
% \text{ where }  \tilde{f}_n(\omega, k) &=\gamma_n \1_{\{k_2^{(n)},\ldots,k_1^{(n)}-1\}}(k) \1_{\rset_+^*}(\Wn_k(\omega)) \Upsilon(\btau_{\gamma}^{\infty}(\Wn_k(\omega))) \eqsp,
\end{align}
where $\Upsilon(u) = \int_{u}^{\plusinfty} \abs{g}^2 \varphibf(g) \rmd g$,  $\btau_{\gamma_n}^{\infty}(\Wn_k(\omega)) = [ \tau_{\gamma_n}(\Wn_k(\omega)) + \tcrw{\gamma_n}\InftyBound ]/\{2(\sigma^2 \gamma_n)^{\half}\}$.
In addition by \Cref{lem:0_proba}, \Cref{ass:sticky_cont} and using that for any $x\in \rset_+$, $1-2\Phibf(-x)\leq \sqrt{2/\pi} x$, we have for any $w\in \rset_+$
\begin{align}
  &Q_{\gamma_n} \1_{\rset_+^*}(w)= 1 - 2\Phibf\parenthese{-\frac{\tau_{\gamma_n}(w) + \gamma_n \InftyBound}{2\sigma\sqrt{\gamma_n}}}\\
  &\qquad\qquad\leq \1_{\rset_+^*}(w)- 2\Phibf\parenthese{-\frac{(1+\Ltt_{\kappa})M + \sqrt{\gamma_n} \InftyBound}{2\sigma}}\1_{\ooint{0,M\sqrt{\gamma}}}(w)+(2\pi\sigma^2)^{-\half}\sqrt{\gamma_n}\InftyBound \eqsp .
\end{align}
Therefore for any $l\in \nset^*$,
\begin{equation}
  \sum_{k=0}^{l-1}  Q_{\gamma_n}^k \1_{\ooint{0,M\sqrt{\gamma}}}(w)\leq 2^{-1}\Phibf\parenthese{-\frac{(1+\Ltt_{\kappa})M + \sqrt{\gamma_n} \InftyBound}{2\sigma}}^{-1}\parentheseDeux{\1_{\rset_+^*}(w)+l(2\pi\sigma^2)^{-\half}\sqrt{\gamma_n}}\eqsp .
\end{equation}
Using that $k_1^{(n)}-k_2^{(n)}\leq 1+(t-s)/\gamma_n$, we have
\begin{align}
  &\gamma_n \sum_{k=k_2^{(n)}}^{k_1^{(n)}-1}\CPE{\1_{\ooint{0,M\sqrt{\gamma_n}}}(\Wn_k) }{\mcf_{k_2^{(n)}}^{(n)}}=\gamma_n \sum_{k=0}^{k_1^{(n)}-k_2^{(n)}-1}Q_{\gamma_n}^k \1_{\ooint{0,M\sqrt{\gamma}}}(\Wn_{k_2^{(n)}})\\
  &\quad\leq 2^{-1}\Phibf\parenthese{-\frac{(1+\Ltt_{\kappa})M + \sqrt{\gamma_n} \InftyBound}{2\sigma}}^{-1}\parentheseDeux{\gamma_n\1_{\rset_+^*}(\Wn_{k_2^{(n)}})+(\gamma_n+t-s)(2\pi\sigma^2)^{-\half}\sqrt{\gamma_n}} \eqsp .
  \label{eq:drift_sqrt_gamma}
\end{align}
Combining \eqref{eq:drift_sqrt_gamma} with \eqref{eq:A_4_2_1_decomposition} and using that $\Upsilon(u) \leq  \int_{0}^{\plusinfty} \abs{g}^2 \varphibf(g) \rmd g$ for any $u\in \rset_+$ we obtain that for any $M\in\rset_+^*$,
 \begin{equation}
   \label{eq:lim-A-4-2-1}
\lim_{n \to \plusinfty}   \expeLigne{\absLigne{D_1^{(n)}(M)}} = 0 \eqsp.
\end{equation}
We now consider $D_2^{(n)}$. Similarly to \eqref{eq:A_4_2_1_decomposition} we obtain
\begin{equation}
  \expeLigne{\absLigne{D_2^{(n)}(M)}} 
\leq 4\gamma_n \sigma^2  \sup_{\rset}\{\abs{ \varphi''} \} \sum_{k=k_2^{(n)}}^{k_1^{(n)}-1} \expe{\1_{\coint{M\sqrt{\gamma_n},\plusinfty}}(\Wn_k) \Upsilon(\btau_{\gamma}^{\infty}(\Wn_k(\omega)))}\eqsp.
  \label{eq:A_4_2_2_decomposition}
\end{equation}
In addition there exists $N_1$ such that for $n\geq N_1$, $1-\gamma_n \Ltt_{\kappa}>1/2$, which implies by  \Cref{ass:sticky_cont} that for any $n \geq N_1$,
\begin{align}
  \1_{\coint{M\sqrt{\gamma_n},\plusinfty}}(\Wn_k) \Upsilon(\btau_{\gamma}^{\infty}(\Wn_k(\omega)))&\leq \1_{\coint{M\sqrt{\gamma_n},\plusinfty}}(\Wn_k) (\btau_{\gamma}^{\infty}(\Wn_k(\omega)))^{-1}\sup_{t\in \rset_+}t\Upsilon(t)\\
  &\leq 4\sigma\sup_{t\in \rset_+}t\Upsilon(t)/M \eqsp .
\end{align}
Note that by using Cauchy–Schwarz inequality we have $\Upsilon(t)^2\leq \Phibf(t)\bfm_{4}$, therefore $\sup_{t\in \rset_+}t\Upsilon(t)<\plusinfty$.
Combining this with \eqref{eq:A_4_2_2_decomposition} and using that $\gamma_n(k_1^{(n)}-k_2^{(n)})\leq (t-s+\gamma_n)$ we obtain that for any $M\in\rset_+^*$,
\begin{equation}
  \limsup_{n\to\plusinfty}\expeLigne{\absLigne{D_2^{(n)}(M)}} 
\leq 16\sigma^3\sup_{\rset}\{\abs{ \varphi''} \}\sup_{t\in \rset_+}\{t\Upsilon(t)\}(t-s)/M \eqsp .
\label{eq:limite_4_2_2}
\end{equation}
Then by \eqref{eq:lim-A-4-2-1} and \eqref{eq:limite_4_2_2},
\begin{align}
    \label{eq:limit_4_2_1_4_2_2}
  \limsup_{n\to \plusinfty} \expeLigne{\absLigne{A_{4,2,1}^{(n)}}}\leq \limsup_{M\to \plusinfty}\limsup_{n\to \plusinfty} \defEns{\expeLigne{\absLigne{D_1^{(n)}(M)}} + \expeLigne{\absLigne{D_2^{(n)}(M)}}} =0 \eqsp .
\end{align}
% Let $\varepsilon>0$. Then by \eqref{eq:limite_4_2_2} there exist $\bar{M}$ and $\bar{N}_1$ such that for any $n\geq \bar{N}_1$, $\expeLigne{\absLigne{A_{4,2,2}^{(n)}(\bar{M})}}\leq \varepsilon/2$. In addition by \eqref{eq:lim-A-4-2-1} there exists $\bar{N}_2$ such that for any $n\geq \bar{N}_2$, $\expeLigne{\absLigne{D_1^{(n)}(\bar{M})}}\leq \varepsilon/2$. Therefore for any $n\geq \max(\bar{N}_1,\bar{N}_2)$, $\expeLigne{\absLigne{D_1^{(n)}(\bar{M})+A_{4,2,2}^{(n)}(\bar{M})}}\leq \varepsilon$. 

Combining this result,  \eqref{eq:limit_4_2_1_4_2_2}-\eqref{eq:limit_A_4_2_3}-\eqref{eq:limit_4_2_2}-\eqref{eq:decomposition-A-4-2}-\eqref{eq:exp_condi-A-4-1}-\eqref{eq:go-to-A-4},  we get  that $\lim_{n \to \plusinfty}   \expeLigne{  \absLigne{\CPELigne{A_{3}^{(n)}}{\mcg_s^{(n)}}}} = 0$. Plugging this result and \eqref{eq:lim-A-2-pb-nartingale}-\eqref{eq:lim-A-1-pb-nartingale} in  \eqref{eq:decomposition-main-proof-martingale-pb} completes the proof.

\textbf{Proof of \eqref{eq:propo:convergence_to_mart_prob_2}.} The proof follows exactly the same lines as  \eqref{eq:propo:convergence_to_mart_prob_1} but we use that the only different and non negligible terms are $A_{4,2,1}^{(n)}$ and $A_{4,2,2}^{(n)}$ which becomes
\begin{equation}
  \label{eq:15}
A_{4,2,1}^{(n)}+A_{4,2,2}^{(n)} = \sum_{k=k_2^{(n)}}^{k_1^{(n)}-1} (\varphi''(\Wn_k)/2)[\{\bar{\Delta} \Wn_{k+1} \}^2 - 4 \gamma_n \sigma^2 ] \eqsp.
\end{equation}

Using \eqref{eq:14_approx_quad} in \Cref{lem:moment_martinga_incre_proof_cont_quad_approx}, the assumption that $\varphi''(w) \geq 0$ for any $w \in \rset$, and the Markov property, we get that for any $n  \in \nset$, $\expeLigne{A_{4,2,1}^{(n)}+A_{4,2,2}^{(n)} \, | \, \mcg_s} \leq 0$, which concludes the proof. 
\end{proof}

%%% Local Variables:
%%% mode: latex
%%% TeX-master: "main_imsart"
%%% End:

\subsection{Postponed proofs of \Cref{sec:application}}
\label{subsec:postponed-application}

\begin{proof}[Proof of \Cref{prop:Euler}]
From \Cref{ass:ODEbounded}  we know that $z_\theta(t)\in\msk$ for all $t\in[0,T]$.  In particular, using that for all $0\leqslant s \leqslant t \leqslant T$,
$z_\theta(t) - z_{\theta}(s) \ = \ \int_s^t F_\theta(z_\theta(u),u)\dd u$,
we get that
\begin{equation}
  \label{eq:38}
  \|z_\theta(t) - z_{\theta}(s)\|\leqslant \mathtt{C}_F(t-s) \eqsp. 
\end{equation}

Let $k\in\N$ be such that $\tilde z_\theta^h(kh)\in\msk$ (this is for instance the case of $k=0$).
% , then $\|\tilde z_\theta^h(t)-\tilde z_\theta^h(k)\| \leqslant \mathtt{C}_F(t-kh)$ for all $t\in [kh,(k+1)h)$.
 Then, for $t\in [kh,(k+1)h]$, using by \eqref{eq:EulerEDO} that
\[z_\theta(t)-\tilde z_\theta^h\po t\pf     =   z_\theta(kh)-\tilde z_\theta^h\po kh\pf + \int_{kh}^t \po F_\theta(z_\theta(s),s)- F_\theta(\tilde z^h_\theta(kh),kh)\pf\dd s \,,\]
and setting $f(t)=\norm{z_\theta(t)-\tilde z_\theta^h\po t\pf}$, we get by \eqref{eq:38} and \Cref{ass:ODEbounded} for any $h >0$ and $t \in \ccintLigne{kh,(k+1)h}$,
\[f(t) \ \leqslant \ (1+\Ltt_F' h) f(kh) + \Ltt_F'(1+\mathtt{C}_F) \frac {h^2}2  \,. \]
Assuming that $h\leqslant \bar h$ where $\bar h$ is sufficiently small so that
\[\frac12 \Ltt_F'(1+\mathtt{C}_F) \bar h T e^{\Ltt_F'T}   < \delta\,, \]
we get by a  direct induction that, for all $k \in \{ 0,\ldots,\lfloor T/h\rfloor\} $, $\tilde z_\theta^h(kh)\in\msk$  and for all $t\in [0,T]$,
\[f(t) \ \leqslant \ \frac12 \Ltt_F'(1+\mathtt{C}_F)   hT e^{\Ltt_F' T} \]
Conclusion follows with 
\[ \mathtt{C} \ = \ \frac12 \Ltt_F'(1+\mathtt{C}_F)NT e^{\Ltt_F'T}\,. \]
\end{proof}

\begin{proof}[Proof of \Cref{prop:EDOmain}]
  We have by \Cref{ass:ODEbounded2} and \eqref{eq:def_s_i},  $\rmT_\gamma(\theta) = \theta - \gamma \nabla U(\theta)$ with
  \begin{equation}
    \label{eq:def_nabla_U}
    \nabla U(\theta) \ = \ \nabla U_0(\theta) + \sum_{i=1}^N \bfs_i \po z_\theta(t_i)\pf
  \end{equation}
where $z_\theta$ solves \eqref{ed:ODEz} and $\bfs_i$ are defined in \eqref{eq:def_s_i}. For all $\theta,\tilde\theta\in\R^d$ and $t\in[0,T]$,
\[z_\theta(t) -  z_{\tilde \theta}(t) \ = \ \int_0^t \po F_\theta(z_\theta(s),s)-F_{\tilde\theta}\po \tilde z_\theta(s),s\pf\pf\dd s\]
and thus, using \eqref{eq:FdemoODE}, Grönwall's inequality implies that 
\[\normLigne{z_\theta(t) -  z_{\tilde \theta}(t)} \ \leqslant \ \Ltt_F \normLigne{\theta-\tilde \theta} t e^{\Ltt_F't}\qquad\forall t\geqslant0\,.\]
In particular, by \eqref{eq:def_nabla_U} and \eqref{eq:def_s_i},
\[\normLigne{\nabla U(\theta)-\nabla U(\tilde\theta)} \ \leqslant \ \po \Ltt_U + \Ltt_{\sbf} \Ltt_F\sum_{i=1}^N t_i e^{\Ltt_F't_i}\pf \normLigne{\theta-\tilde \theta} \,.\]
Moreover, similarly, we get  if  $\normLigne{\theta-\tilde\theta}\geqslant R_U$,
\[\ps{\theta-\tilde\theta}{\nabla U(\theta)-\nabla U(\tilde\theta)} \  \geqslant \ \mtt_U \normLigne{\theta-\tilde\theta}^2 - 2N \mathtt{C}_{\sbf} \normLigne{\theta-\tilde\theta}\,.\]
Combining the last two estimates yields \Cref{ass:LipsAnd}. Finally, \Cref{ass:bound}  follows using \Cref{ass:ODEsolver},  \eqref{eq:def_tilde_b_ODE} and \eqref{eq:def_s_i} from
\[\normLigne{\nabla U(\theta)-\tilde b_h(\theta)} \leqslant \Ltt_{\sbf} \sum_{i=1}^N \normLigne{z_\theta(t_i) - \Psi_i^h(\theta)} \ \leqslant \ \mathtt{C}_{\Psi} \Ltt_{\sbf}  h^{\alpha}
  \,.\]
%\alain{je ne trouve pas tout à fait cela}
\end{proof}

%%% Local Variables:
%%% mode: latex
%%% TeX-master: "main_imsart"
%%% End:

%%% Local Variables:
%%% mode: latex
%%% TeX-master: "main_imsart"
%%% End:

\paragraph{Acknowledgments.}
The work of AG is funded in part by the Project EFI ANR-17-CE40-0030 of the French National Research Agency. AD acknowledges support  of the Lagrange Mathematical and Computing Research Center. A.E. has been supported by the Hausdorff Center for Mathematics. Gef\"ordert durch die Deutsche Forschungsgemeinschaft (DFG) im Rahmen der Exzellenzstrategie des Bundes und der L\"ander - GZ 2047/1, Projekt-ID 390685813.

\bibliographystyle{plain}
\bibliography{../../bibliography/bibliography}

\tableofcontents
\appendix

% !TeX root = main.tex

\section{Proof of Proposition~\ref{theo:convergence_markov_chain_rsetd}}
\label{sec:theo:convergence_markov_chain_rsetd}

Recall that under \Cref{ass:LipsAnd}, we have for any $\gamma \in\ocint{0,\bgamma}$,
\begin{equation}
  \label{eq:proof_theo:convergence_markov_chain_rsetd_ass}
  \begin{aligned}
&\sup_{x  \in \rset^d, \, x \neq 0} \{ \norm{\rmT_{\gamma}(x) -
  \rmT_{\gamma}(0)} / \norm{x} \}\leq (1+\gamma \Ltt) \eqsp, \\
& \sup_{x \in \rset^d, \, \norm{x} \geq R_1} \{ \norm{\rmT_{\gamma}(x) -
 \rmT_{\gamma}(0)} / \norm{x} \}\leq (1-\gamma \mtt)  \eqsp.
\end{aligned} 
\end{equation}
We show that there
  exist $\lambda \in \ooint{0,1}$,
  $A \geq 0$ such that for any $\gamma \in \ocint{0,\tcwc{\bgamma}}$ and
  $x \in \rset^d$,
  \begin{equation}
    \label{eq:drift_R_gamma_V_c_proof_1}
    R_{\gamma}V_c(x) \leq \lambda^{\gamma} \tcwc{V_c(x)} + \gamma A  \eqsp. 
  \end{equation}
  Denote $\rmT_{\infty}=\sup_{\gamma\in \ocint{0,\bgamma}}\gamma^{-1}\norm{\rmT_{\gamma}(0)}$, we show that it holds with
  \begin{equation}
        \label{eq:proof_ergo_R_gamma_def_constant}
  \begin{aligned}
   & \lambda =\exp\parentheseLigne{-c\mtt M^2/8} \eqsp, \quad     M = R_1 \vee (16 d \sigma^2 /\mtt)^{\half}\vee \left.\parentheseDeux{4\rmT_{\infty}+2\bgamma \rmT_{\infty}^2}\right/\mtt \eqsp,
    \\
&    A =  \exp(c M^2 + \bgamma\{\tcwc{c}B_1 M^2+B_2 - \log(\lambda)\})\{\tcwc{c}B_1 M^2+B_2-\log(\lambda)\}\eqsp, \,  \\
 &   B_1 = 4 C_1+2(1+8 c \sigma^2 \bgamma) (1+\bgamma L)\rmT_{\infty} \eqsp, \quad C_1 = C_2 \vee C_2^2 \bgamma \eqsp, \quad C_2 = (2L\tcwc{+}L^{\tcwc{2}}\bgamma)\vee (8c \sigma^2)\eqsp ,\\
 & c=\mtt/(32\sigma^2)\eqsp , \qquad B_2 = 2dc\sigma^2+2(1+8 c \sigma^2 \bgamma) (1+\bgamma L)\rmT_{\infty}+c(1+8 c \sigma^2 \bgamma)\bgamma\rmT_{\infty}^2  \eqsp. 
  \end{aligned}
\end{equation}
Define for any $x \in \rset^d$,
  $\overline{\rmT}_{\gamma}(x) = \rmT_{\gamma}(x)-\rmT_{\gamma}(0)$. Note
  that for any $\gamma \in \ocint{0,\tcwc{\bgamma}}$,
  $2 c \sigma^2\gamma < 1$ by definition of $c$
  \eqref{eq:proof_ergo_R_gamma_def_constant} and \tcwc{using that $\bgamma\leq 1/\mtt$}. Let $x \in \rset^d$ and
  $\gamma \in \ocint{0,\tcwc{\bgamma}}$. Then, we obtain using that $\int_{\rset} \rme^{az +bz^2 -z^2/2} \rmd z = (2\uppi(1-2b)^{-1})^{\tcwc{1/2}}\rme^{a^2/(2(1-2b))}$ for any $a \in \rset$ and $b \in \coint{0,1/2}$,
  \begin{align}
    R_{\gamma}V_c(x) & = (2\uppi)^{-d/2}\int_{\rset^d}\exp(c \normLigne[2]{\rmT_{\gamma}(x)+(\sigma^2 \gamma)^{\half} z} - \norm[2]{z}/2) \rmd z \\
    \label{eq:proof_ergo_R_gamma_eq_base}
                     & = (1-2c\sigma^2 \gamma)^{-d/2} \exp\defEnsLigne{c(1-2c\sigma^2\gamma)^{-1} \normLigne[2]{\rmT_{\gamma}(x)}}\\
                     &\leq (1-2c\sigma^2 \gamma)^{-d/2} \exp\defEns{c(1-2c\sigma^2\gamma)^{-1} \parentheseDeux{\normLigne{\overline{\rmT}_{\gamma}(x)}+\norm{\rmT_{\gamma}(0)}}^2}\eqsp .
  \end{align}
  We now distinguish the case $\norm{x} \geq M$ and $\norm{x} < M$.
  In the first case, we get by \eqref{eq:proof_theo:convergence_markov_chain_rsetd_ass},
  $(1-2 c \sigma^2 \gamma)^{-1} \leq \tcwc{1+}8 c \sigma^2 \gamma$ and
  $(1-\mtt \gamma)^2 \leq 1 - \mtt \gamma$ since
  $2 c \sigma^2 \gamma \leq 1/2$ and $\gamma \leq 1/\mtt$, \tcwc{by definition of $c$
  \eqref{eq:proof_ergo_R_gamma_def_constant} and using that $\bgamma\leq 1/\mtt$},
  \begin{align}
    R_{\gamma}V_c(x) &\leq (1-2c \gamma \sigma^2)^{-d/2} \exp\parentheseDeux{c (1+8 c \sigma^2 \gamma) \parentheseDeux{{(1-\mtt \gamma)\norm[2]{x}}+2\gamma \norm{x}\rmT_{\infty}+\gamma^2\rmT_{\infty}^2}} \\
\label{eq:proof_ergo_R_gamma_eq_base_1}
&    \leq  \exp\parentheseDeux{c (1-\mtt \gamma/8) \norm[2]{x} + 2 d c \sigma^2 \gamma - \mtt c \gamma \norm[2]{x}/8} \leq \lambda^{\gamma} \tcwc{V_c}(x)   \eqsp,
\end{align}
where we used for the penultimate inequality that $-\log(1-t) \leq 2 t$ for $t \in \ccint{0,1/2}$.
For the case $\norm{x} \leq M$, by \eqref{eq:proof_theo:convergence_markov_chain_rsetd_ass} and
$(1-2 c \sigma^2 \gamma)^{-1} \leq \tcwc{1+}8 c \sigma^2 \gamma$ since
$2 c \sigma^2 \gamma \leq 1/2$, $\gamma \leq 1/\mtt$,
and using $-\log(1-t) \leq 2 t$ for $t \in \ccint{0,1/2}$ again,
  \begin{align}
    R_{\gamma}V_c(x) &\leq (1-2c \gamma \sigma^2)^{-d/2} \exp\parentheseDeux{c (1+8 c \sigma^2 \gamma) \parenthese{(1+\Ltt \gamma)^2\norm[2]{x}+2(1+\gamma\Ltt)\gamma \norm{x}\rmT_{\infty}}} \\
    &\qquad\qquad\qquad\qquad\qquad\qquad\qquad\qquad\qquad\qquad \times\exp\parenthese{c (1+8 c \sigma^2 \gamma)\gamma^2\rmT_{\infty}^2}\\ 
%\label{eq:proof_ergo_R_gamma_eq_base_1}
&    \leq  \exp\parentheseDeux{c (1+ \gamma B_1) \norm[2]{x} + \gamma B_2 } \eqsp,
\end{align}
where we used for the last line that for any $x\in \rset^d$, $\norm{x}\leq 1+\norm{x}^2$.
Using that $\rme^{t} -1 \leq t \rme^{t}$ for $t \geq 0$, we obtain that
\begin{align}
  %\label{eq:1}
&  R_{\gamma}V_c(x) \tcwc{\leq} \lambda^{\gamma} \tcwc{V_c}(x) + \lambda^{\gamma}\tcwc{V_c}(x)\{\exp\parentheseDeux{c \gamma B_1 \norm[2]{x} + \gamma B_2  -\gamma \log(\lambda)} -1 \} \\
  & \leq  \lambda^{\gamma} \tcwc{V_c}(x) + \gamma \tcwc{V_c}(x)\{\tcwc{c}B_1 \norm[2]{x} +  B_2  - \log(\lambda)\} \exp\parentheseDeux{ \bgamma (\tcwc{c}B_1\norm[2]{x} +B_2-\log(\lambda)} \eqsp,
\end{align}
which combined with \eqref{eq:proof_ergo_R_gamma_eq_base_1} completes the proof of \eqref{eq:drift_R_gamma_V_c_proof_1}.

In addition, by \cite[Theorem 19]{durmus2019high} and using that for any $i\in \mathbb{N}$, $\rme^{-2i\gamma\Ltt}\leq (1+\gamma\Ltt)^{-2i}$, we have, for any $x,y\in \rset^d$ and $t_0\in \rset_+^*$,
\begin{equation}
  \label{eq:minorization_prop_1}
  \tvnorm{\updelta_x R^{\lceil t_0/\gamma\rceil }_{\gamma}-\updelta_y R^{\lceil t_0/\gamma\rceil }_{\gamma}}\leq 1-2\Phibf\parenthese{-\frac{\norm{x-y}}{2\sigma^2t_0\rme^{-2(t_0+\bgamma)\Ltt}}}
\end{equation}
where $\Phibf$ is the cumulative distribution of the one-dimensional Gaussian distribution with mean $0$ and variance $1$.

The proof is completed by combining \eqref{eq:drift_R_gamma_V_c_proof_1}, \eqref{eq:minorization_prop_1} and \cite[Theorem 19.4.1]{douc:moulines:priouret:soulier:2018}\footnote{\tcwc{There is a $b$ missing in Equation 19.4.2d in \cite[Theorem 19.4.1]{douc:moulines:priouret:soulier:2018}}}.

\section{Proof of Proposition \ref{theo:ergo_tilde_R}}
\label{sec:proof-theo:ergo_tilde_R}

First note that under \Cref{ass:LipsAnd} and \Cref{ass:bound}, for all $x \in \rset^d$, since $$\norm{\trmT_\gamma(x) -
 \rmT_\gamma(0)} \leq \gamma \InftyBound + \norm{\rmT_\gamma(x) -
 \rmT_\gamma(0)}$$ and from \eqref{eq:proof_theo:convergence_markov_chain_rsetd_ass},
\begin{equation}
  \label{eq:proof_theo:convergence_markov_chain_rsetd_ass_tilde_R}
 \sup_{x \in \rset^d, \, \normLigne{x} \geq R_1} \parentheseDeux{  {\normLigne{\trmT_\gamma(x) -
 \rmT_\gamma(0)} }/\{(1-\gamma \mtt) \normLigne{x} + \gamma \InftyBound\}} \leq  1  \eqsp.
\end{equation}

% Define for any $x \in \rset^d$
%   \begin{equation}
% \label{eq:tilde_Rproof_ergo_R_gamma_def_tilde_V}
% \text{      $\overline{V}_{c,\gamma}(x) = \exp(c \normLigne[2]{x-\rmT_{\gamma}(0)})$, for
%   $c = \tcwc{\mtt}/(16\sigma^2)$ } \eqsp.
%   \end{equation}
%  Note that
%   $\lim_{\norm{x} \to \plusinfty} \overline{V}_{c,\gamma}(x) = \plusinfty$ and
%   $\lim_{\norm{x} \to \plusinfty} \{\overline{V}_{c,\gamma}/
%   V_c\}(x)=1$.
  
  \tcwc{
In addition, by \eqref{eq:def_R_gamma}, for any compact set $\msk$, for any $y\in \msk$ and $\msa \in \mathcal{B}(\rset^d)$,
\begin{align}
  \tR_\gamma(y,\msa) &= (2\uppi \sigma^2 \gamma)^{-\nicefrac{d}{2}}\int_{\rset^d} \1_{\msa}(y') \exp\defEns{-\norm[2]{y'-\trmT_{\gamma}(y)}/(2\sigma^2 \gamma)} \rmd y' \\
  &\geq 2^{-d/2}\tcr{\inf_{y\in \msk}}\parenthese{\rme^{-\normLigne{\trmT_{\gamma}(y)}^2/(\sigma^2 \gamma)}}(\uppi \sigma^2 \gamma)^{-\nicefrac{d}{2}}\int_{\rset^d} \1_{\msa}(y') \exp\defEns{-\norm[2]{y'}/(\sigma^2 \gamma)} \rmd y'\eqsp .
\end{align}
 Note that by \Cref{ass:LipsAnd} and \Cref{ass:bound}, for any compact set $\msk$, $\tcr{\inf_{y\in \msk}}\parentheseLigne{\rme^{-\normLigne{\trmT_{\gamma}(y)}^2/(\sigma^2 \gamma)}}>0$.
As a result, by \cite[Definition 9.3.5, Definition 9.2.2, Definition 9.1.1]{douc:moulines:priouret:soulier:2018}, $\tilde{R}_{\gamma}$ is strongly aperiodic,
$\Leb$-irreducible and all compact sets are $1$-small.
  }
  
   It is therefore sufficient to show by
  \cite[Theorem 16.0.1]{MeynTweedie} that there
  exists $c >0$,
  such that for any $\gamma \in \ocint{0,\tcwc{\bgamma}}$, there exists
  $\lambda_{\gamma} \in \coint{0,1}$, a compact $\msk\subset\rset^d$ and $A_{\gamma} \geq 0$, such
  that% $\tilde{R}_{\gamma} V_c \leq \lambda_{\gamma} V_c + A_{\gamma}$
  \begin{equation}
    \label{eq:tilde_Rdrift_R_gamma_V_c_proof_1}
\tilde{R}_{\gamma} V_c \leq \lambda_{\gamma} V_c + A_{\gamma}\1_\msk  \eqsp. 
  \end{equation}
  %We show that it holds with .
%   \begin{equation}
%         \label{eq:tilde_Rproof_ergo_R_gamma_def_constant}
%   \begin{aligned}
%    & \bgamma_1 = \bgamma \wedge\{1/\mtt\} \eqsp, \quad 
% \lambda =\exp\parentheseLigne{-c\mtt M^2/4} \eqsp, \quad     M = R_1 \vee (8 d \sigma^2 /\mtt)^{\half} \eqsp,
%     \\
% &    A = \lambda^{\bgamma} \exp(c M^2 + \bgamma\{B_1 M^2+B_2 - \log(\lambda)\})\{B_1 M^2+B_2-\log(\lambda)\}\eqsp, \, B_2 = 2dc\sigma^2 = d\mtt/8 \eqsp, \\
%  &   B_1 = 4 C_1 \eqsp, \quad C_1 = C_2 \vee C_2^2 \bgamma \eqsp, \quad C_2 = (2L)\vee(L\bgamma)\vee (8c \sigma^2) = (2L)\vee(L\bgamma)\vee (\mtt/2) \eqsp. 
%   \end{aligned}
% \end{equation}
\tcwc{Let $c = \tcwc{\mtt}/(16\sigma^2)$ and }define for any $x \in \rset^d$,
  $\overline{\rmT}_{\gamma}(x) = \trmT_{\gamma}(x)-\rmT_{\gamma}(0)$. Note
  that for any $\gamma \in \ocint{0,\tcwc{\bgamma}}$,
  $2 c \sigma^2\gamma \leq 1$ by definition of $c$ and \tcwc{using that $\bgamma\leq 1/\mtt$}. Let $x \in \rset^d$ and
  $\gamma \in \ocint{0,\tcwc{\bgamma}}$. Then, we obtain using that $\int_{\rset} \rme^{az +bz^2 -z^2/2} \rmd z = (2\uppi(1-2b)^{-1})^{\tcwc{1/2}}\rme^{a^2/(2(1-2b))}$ for any $a \in \rset$ and $b \in \coint{0,1/2}$,
  \begin{align}
    \tilde{R}_{\gamma}V_c(x) & = (2\uppi)^{-d/2}\int_{\rset^d}\exp(c \normLigne[2]{\tilde{\rmT}_{\gamma}(x)+(\sigma^2 \gamma)^{\half} z} - \norm[2]{z}/2) \rmd z \\
    \label{eq:proof_ergo_R_gamma_eq_base}
                     & = (1-2c\sigma^2 \gamma)^{-d/2} \exp\defEnsLigne{c(1-2c\sigma^2\gamma)^{-1} \normLigne[2]{\tilde{\rmT}_{\gamma}(x)}}\\
                     &\leq (1-2c\sigma^2 \gamma)^{-d/2} \exp\defEns{c(1-2c\sigma^2\gamma)^{-1} \parentheseDeux{\normLigne{\overline{\rmT}_{\gamma}(x)}+\norm{\rmT_{\gamma}(0)}}^2}\eqsp .
  \end{align}

If $\norm{x} \geq M$, we get by \eqref{eq:proof_theo:convergence_markov_chain_rsetd_ass_tilde_R},
  $(1-2 c \sigma^2 \gamma)^{-1} \leq \tcwc{1+}8 c \sigma^2 \gamma$ and
  $(1-\mtt \gamma)^2 \leq 1 - \mtt \gamma$ since
  $2 c \sigma^2 \gamma \leq 1/2$ and $\gamma \leq 1/\mtt$, by definition of $c$ and \tcwc{using that $\bgamma\leq 1/\mtt$},
  \begin{align}
    \tilde{R}_{\gamma}V_c(x) &\leq (1-2c \gamma \sigma^2)^{-d/2} \exp\parentheseDeux{c (1+8 c \sigma^2 \gamma)\{ (1-\mtt \gamma) \norm{x} +\gamma\InftyBound+\norm{\rmT_{\gamma}(0)}\}^2} \\
\label{eq:tilde_Rproof_ergo_R_gamma_eq_base_1}
&    \leq  (1-2c \gamma \sigma^2)^{-d/2} \exp\parentheseDeux{c (1+\mtt \gamma/2)\{ (1-\mtt \gamma) \norm{x} +\gamma\InftyBound+\norm{\rmT_{\gamma}(0)}\}^2} \\
&\leq (1-2c \gamma \sigma^2)^{-d/2} \exp\parentheseDeux{c (1-\mtt \gamma/2)\norm{x}^2 +B_1\norm{x}+B_2} \eqsp,
\end{align}
for some $B_{1,\gamma},B_{2,\gamma}\in \rset$.

Therefore, we get $\limsup_{\norm{x} \to \plusinfty}     [\tilde{R}_{\gamma}V_c(x)/ V_c(x)] = 0$ which completes the proof of \eqref{eq:tilde_Rdrift_R_gamma_V_c_proof_1}.

\section{Proof of Theorem~\ref{coro:exp_moment}}
\label{sec:proof:coro:exp_moment}
For any $a >0$, define $\LyapDsexp_a : \rset_+ \to [0,+\infty)$ for any $w \in \rset_+$ by
\begin{equation}
  \label{eq:def_v_a}
  \LyapDsexp_a(w)= \exp(aw)-1 \eqsp.
\end{equation}

The proof of \Cref{coro:exp_moment} is based on the following proposition which combined technical lemmas gathered in \Cref{sec:technical-results}.

  \begin{proposition}
    \label{prop:fix_drift_exp_drift}
    Assume \Cref{ass:LipsAnd}-\ref{ass:LipsAnd_2}. For any $w\in \rset_+$ and $\gamma\in \ocint{0,\bgamma_1}$,
    \begin{equation}
      Q_{\gamma} \VlyapDsexp_{a}(w)\leq \uplambda_a^\gamma \VlyapDsexp_{a}(w)\1_{\coint{R_a,\plusinfty}}(w)+B_a^\gamma \VlyapDsexp_{a}(w)\1_{\coint{0,R_a}}(w) +\gamma D_a\1_{\coint{0,R_a}}(w) \eqsp ,
    \end{equation}
    where 
    \begin{align}
      \label{eq:fix_exp_deift_const}
      \begin{aligned}
        \tilde{R}_a &= 1 \vee R_1 \vee [(4a \sigma^2 +2 \InftyBound)/\mtt] \vee 16\sigma^2 a /\mtt\eqsp,\qquad    \uplambda_a = \exp(-a \mtt \tilde{R}_a/8)\eqsp , \\
        C_a&= a(\InftyBound+2a\sigma^2)\rme^{a\bgamma(\InftyBound+2a\sigma^2)}+2\sigma^2 (2\uppi)^{-1/2} a \rme^{(a+2\sigma\bgamma^{\half}a)^2/2}\eqsp ,\\
        \quad R_a&=\tilde{R}_a\vee a^{-1}\log\parenthese{1+C_a/(-\log(\uplambda_a)\uplambda_a^{2\bgamma})}\\
        \bgamma_1&=\bgamma\wedge 1/(-\log(\uplambda_a))\wedge1/(4\sigma^2)\eqsp ,\qquad B_a=\rme^{a(\InftyBound+2a\sigma^2+\Ltt R_a)} \\
        D_a&=a(\InftyBound+2a\sigma^2+\Ltt R_a)\rme^{a\bgamma(\InftyBound+2a\sigma^2+\Ltt R_a)}\\
        &\qquad\qquad\qquad\qquad\qquad+2\sigma^2 (2\uppi)^{-1/2} a \rme^{(a+2\sigma\bgamma^{\half}a)^2/2}+2\sigma^2 a^2 \uplambda_a^{2\bgamma}\VlyapDsexp_{a}(R_a) \eqsp. 
      \end{aligned}
    \end{align}
  \end{proposition}
  \begin{proof}
    Combining \Cref{lem:big_fix_exact_decomp}, \Cref{lem:big_fix_dev_Phi} and \Cref{lem:big_fix_last_term} we have for any $w\in \rset_+$ and $\gamma\in \ocint{0,\bgamma_1}$,
    \begin{align}
      \label{eq:big_fix_drift_exp_1}
      Q_{\gamma} \VlyapDsexp_{a}(w)&\leq\parenthese{\rme^{a(\tau_\gamma(w)+\gamma\InftyBound)+2a^2\sigma^2\gamma}-1}+4\rme^{2a^2\sigma^2\gamma}\sigma^2 a \gamma (2\uppi)^{-1/2}  \rme^{(a+2\sigma\bgamma^{\half}a)^2/2}\\
      &\qquad\qquad\qquad\qquad\qquad+2\sigma^2\gamma a^2\rme^{a(\tau_\gamma(w)+\gamma\InftyBound)+2a^2\sigma^2\gamma} \eqsp .
    \end{align}

    In addition, for any $w\in \rset_+$ and $\gamma\in \ocint{0,\bgamma}$ 
    \begin{align}
      \label{eq:fix_exp_drift_1}
      \parenthese{\rme^{a(\tau_\gamma(w)+\gamma\InftyBound)+2a^2\sigma^2\gamma}-1}&=\rme^{a(\tau_\gamma(w)-w+\gamma\InftyBound)+2a^2\sigma^2\gamma}\parenthese{\rme^{aw}-1}\\&\qquad+\parenthese{\rme^{a(\tau_\gamma(w)-w+\gamma\InftyBound)+2a^2\sigma^2\gamma}-1}\\
      &\leq \rme^{a(\tau_\gamma(w)-w+\gamma\InftyBound)+2a^2\sigma^2\gamma}\VlyapDsexp_{a}(w)\\
      &\qquad +\parenthese{\rme^{a(\tau_\gamma(w)-w+\gamma\InftyBound)+2a^2\sigma^2\gamma}-1} \eqsp .
    \end{align}
  By \Cref{ass:LipsAnd}-\ref{ass:LipsAnd_2},
  \begin{align}
    \label{eq:fix_exp_drift_2}
    \rme^{a(\tau_\gamma(w)-w+\gamma\InftyBound)+2a^2\sigma^2\gamma}\leq \rme^{a\gamma(-\mtt w+\InftyBound+2a\sigma^2)}\1_{\coint{R_a,\plusinfty}}+\rme^{a\gamma(\Ltt w+\InftyBound+2a\sigma^2)}\1_{\coint{0,R_a}} \eqsp .
  \end{align}
Then using that $R_a\geq \tilde{R}_a$, for any $t\in \rset_+ $, $\rme^{t}-1\leq t\rme^{t}$ and combining \eqref{eq:fix_exp_drift_1}, \eqref{eq:fix_exp_drift_2} and \eqref{eq:fix_exp_deift_const}
\begin{align}
  \label{eq:big_fix_drift_exp_2}
  \begin{aligned}
    \parenthese{\rme^{a(\tau_\gamma(w)+\gamma\InftyBound)+2a^2\sigma^2\gamma}-1}&\leq \uplambda_a^{4\gamma}\VlyapDsexp_{a}(w)\1_{\coint{R_a,\plusinfty}}+B_a^\gamma \VlyapDsexp_{a}(w)\1_{\coint{0,R_a}}\\
    & +\gamma a(\InftyBound+2a\sigma^2)\rme^{a\bgamma(\InftyBound+2a\sigma^2)} \1_{\coint{R_a,\plusinfty}}\\
    & + \gamma(a\InftyBound+2a^2\sigma^2+a\Ltt R_a)\rme^{a\gamma\InftyBound+2a^2\sigma^2\gamma+a\Ltt\gamma R_a}\1_{\coint{0,R_a}}(w) \eqsp .
  \end{aligned}
\end{align}
In addition, using that for any $t\in \rset_-$, $\rme^{t}-1\leq t+t^2/2$, $\gamma\leq 1/(-\log(\uplambda_a))$ and by \eqref{eq:fix_exp_deift_const}, for any $w\in \coint{R_a,\plusinfty}$, and $\gamma\in \ocint{0,\bgamma_1}$ we have 
\begin{align}
  \label{eq:big_fix_drift_exp_4}
  \begin{aligned}
    \parenthese{\uplambda_a^{4\gamma}\VlyapDsexp_{a}(w)+\gamma C_a}&\leq \uplambda_a^{2\gamma}\parenthese{1+\uplambda_a^{2\gamma}-1+\gamma C_a/(\uplambda_a^{2\gamma}\VlyapDsexp_{a}(w))}\VlyapDsexp_{a}(w)\\
    &\leq \uplambda_a^{2\gamma}\parenthese{1+\gamma\log(\uplambda_a^2)+\gamma^2\log(\uplambda_a^2)^2/2+\gamma C_a/(\uplambda_a^{2\gamma}\VlyapDsexp_{a}(w))}\VlyapDsexp_{a}(w)\\
    &\leq \uplambda_a^{2\gamma}\parenthese{1+\gamma\log(\uplambda_a^2)/2+\gamma C_a/(\uplambda_a^{2\gamma}\VlyapDsexp_{a}(w))}\VlyapDsexp_{a}(w)\leq \uplambda_a^{2\gamma}\VlyapDsexp_{a}(w) \eqsp .
  \end{aligned}
\end{align}
In addition, for any $w\in \coint{R_a,\plusinfty}$ and $\gamma\in \ocint{0,\bgamma_1}$,
\begin{equation}
  \label{eq:big_fix_drift_exp_3}
  \uplambda_a^{2\gamma} \VlyapDsexp_{a}(w)+2\sigma^2\gamma a^2 \rme^{a(\tau_\gamma(w)+\gamma\InftyBound)+2a^2\sigma^2\gamma}\leq \uplambda_a^{2\gamma} \VlyapDsexp_{a}(w)+2\sigma^2\gamma a^2 \uplambda_a^{2\gamma}\VlyapDsexp_{a}(w) \eqsp .
\end{equation}
Using \eqref{eq:fix_exp_deift_const} we have 
\begin{equation}
  \uplambda_a^{2\gamma}\parenthese{1+\gamma2\sigma^2 a^2}=\rme^{2\gamma\log(\uplambda_a)+\log(1+\gamma2\sigma^2 a^2)}\leq \rme^{2\gamma\log(\uplambda_a)+\gamma2\sigma^2 a^2}\leq \uplambda_a^{\gamma} \eqsp ,
\end{equation}
which completes the proof when combined with \eqref{eq:big_fix_drift_exp_1}, \eqref{eq:big_fix_drift_exp_4}, \eqref{eq:big_fix_drift_exp_2} and \eqref{eq:big_fix_drift_exp_3}.
  \end{proof}

  \begin{proof}[Proof of \Cref{coro:exp_moment}]
  Let $\bdelta \in \ocint{0,\{ \Ltt^{-1} \wedge (\sigma\rme^{-1}/\InftyBound)^2\}}$ and  $\gamma \in \ocint{0,\bgamma_1}$.    We show that \eqref{eq:bound_moment_exp_rset_star_invariant_mes} holds with
    \begin{align}
      \label{eq:8}
      \constMoment_3 &= (B_a /\lambda_a)^{\bgamma} a(\Ltt R_a +2 \sigma^2 a +\InftyBound +m\tilde{R}_a/8)  \eta_{R_a} \LyapDsexp_a (R_a) / \abs{\log(\lambda_a)} \\
      &\qquad+\left.\parentheseDeux{D_a\eta_{R_a}+A_a}\right/\abs{\log(\uplambda_a)}\eqsp,
    \end{align}
    where $\uplambda_a,R_a,\tilde{R}_a,B_a,D_a$ are defined in \eqref{eq:fix_exp_deift_const}, $A_a$ in \Cref{lem:fix_drift_exp_0}, and $\eta_R$ in \eqref{eq:def_eta_R}. 
    
    By \Cref{lem:fix_drift_exp_0}, \Cref{prop:fix_drift_exp_drift} and since $\VlyapDsexp(0) = 0$ and $\mu_{\gamma}$ is invariant for $Q_{\gamma}$, we have
        \begin{align}
      \label{eq:drift_exp_sticky_v0_2}
\int_{(0,+\infty)}     \LyapDsexp_a (w) \rmd \mu_{\gamma}(w)&=\int_{(0,+\infty)}     Q_{\gamma}\LyapDsexp_a (w) \rmd \mu_{\gamma}(w)+Q_{\gamma}\LyapDsexp_a (0)\mu_\gamma(\{0\})\\
 &\leq \uplambda_a^{\gamma} \int_{\coint{R_a,\plusinfty}}     \LyapDsexp_a (w) \rmd \mu_{\gamma}(w) + B_a^{\gamma}\int_{\ooint{0,R_a}}     \LyapDsexp_a (w) \rmd \mu_{\gamma}(w) \\
&\qquad +\gamma D_a\mu_\gamma(\ooint{0,R_a})+\gamma\InftyBound A_a\eqsp. 
    \end{align}
     Rearranging terms yields
        \begin{align}
      \int_{(0,+\infty)}     \LyapDsexp_a (w) \rmd \mu_{\gamma}(w)& \leq \{B_a^{\gamma} - \uplambda_a^{\gamma}\}/\{1-\uplambda_a^{\gamma}\}  \int_{\ooint{0,R_a}} \VlyapDsexp_a(w) \rmd \mu_{\gamma}(w)\\
      &\quad +\gamma D_a\mu_\gamma(\ooint{0,R_a})/(1-\uplambda_a^{\gamma})+\gamma \InftyBound A_a/(1-\uplambda_a^{\gamma})\\
&\leq \{B_a^{\gamma} \uplambda_a^{- \gamma} - 1 \}/\{\uplambda_a^{-\gamma}-1\} \LyapDsexp_a (R_a) \InftyBound \eta_{R_a} \\
&\qquad +\InftyBound\gamma\uplambda_a^{-\gamma}\left.\parentheseDeux{D_a\eta_{R_a}+A_a}\right/(\uplambda_a^{-\gamma}-1) \eqsp,
\end{align}
where we have used \Cref{theo:bound_mu_0_R_1} applied to $R \leftarrow R_a$ in the last inequality. The proof is then completed upon using that for any $t \geq 0$, $t \leq \rme^{t} - 1 \leq t \rme^{t}$.
  \end{proof}

  \subsection{Technical results}
\label{sec:technical-results}
  
  \begin{lemma}
    \label{lem:big_fix_exact_decomp}
    Let $a >0$. Then, for any $\gamma >0$, and $w\in \rset_+$
    \begin{align}
      Q_{\gamma} \VlyapDsexp_{a}(w)&=\rme^{2a^2\sigma^2\gamma}\Bigg[\rme^{a(\tau_{\gamma}(w) +\gamma \InftyBound)}\Bigg(\Phibf\parenthese{\frac{\tau_{\gamma}(w) +\gamma \InftyBound}{2\sigma\gamma^{1/2}}+2\sigma\gamma^{1/2}a}\\
      &\qquad\qquad\qquad\qquad\qquad\qquad\qquad\qquad-\Phibf\parenthese{-\frac{\tau_{\gamma}(w) +\gamma \InftyBound}{2\sigma\gamma^{1/2}}+2\sigma\gamma^{1/2}a}\Bigg)\\
      &\qquad \qquad \qquad+2\sinh\parenthese{a(\tau_{\gamma}(w) +\gamma \InftyBound)}\Phibf\parenthese{-\frac{\tau_{\gamma}(w) +\gamma \InftyBound}{2\sigma\gamma^{1/2}}+2\sigma\gamma^{1/2}a}\Bigg]\\
      &\qquad -1+2\Phibf\parenthese{-\frac{\tau_{\gamma}(w) +\gamma \InftyBound}{2\sigma\gamma^{1/2}}} \eqsp ,
    \end{align}
    where $\sinh$ is the hyperbolic sine.
  \end{lemma}
  \begin{proof}
By \eqref{def_q_gamma} and \Cref{lem:0_proba}, for any $w\in \rset_+$, $\gamma >0$
\begin{align}
  Q_{\gamma} \VlyapDsexp_{a}(w)&=\int_{\rset} \VlyapDsexp_{a}\parenthese{\tau_{\gamma}(w)   +\gamma \InftyBound-2 \sigma \gamma^{1/2} g} \{1- \bpg\parenthese{\tau_{\gamma}(w) + \gamma \InftyBound , g} \} \varphibf(g)  \rmd g\\
  &=\int_{\rset} \rme^{a\parenthese{\tau_{\gamma}(w)   +\gamma \InftyBound-2 \sigma \gamma^{1/2} g}} \{1- \bpg\parenthese{\tau_{\gamma}(w) + \gamma \InftyBound , g} \} \varphibf(g)  \rmd g\\
  &\qquad\qquad\qquad\qquad\qquad\qquad\qquad\qquad-1+2\Phibf\parenthese{-\frac{\tau_{\gamma}(w) +\gamma \InftyBound}{2\sigma\gamma^{1/2}}}\eqsp .
\end{align}
    In addition, by \eqref{eq:def_bpg} and using changes of variable,
    \begin{align}
      &\int_{\rset} \rme^{a\parenthese{\tau_{\gamma}(w)   +\gamma \InftyBound-2 \sigma \gamma^{1/2} g}} \{1- \bpg\parenthese{\tau_{\gamma}(w) + \gamma \InftyBound , g} \} \varphibf(g)  \rmd g\\
      &=\int_{-\infty}^{(\tau_{\gamma}(w)   +\gamma \InftyBound)/2} \rme^{a\parenthese{\tau_{\gamma}(w)   +\gamma \InftyBound-2  g}} \defEns{1- 1\wedge \frac{\varphibf_{\sigma^2\gamma}\parenthese{ \tau_{\gamma}(w) + \gamma \InftyBound- g } }{\varphibf_{\sigma^2\gamma}\parenthese{ g}} } \varphibf_{\sigma^2\gamma}(g)  \rmd g\\
      &=\int_{-\infty}^{(\tau_{\gamma}(w)   +\gamma \InftyBound)/2} \rme^{a\parenthese{\tau_{\gamma}(w)   +\gamma \InftyBound-2  g}} \varphibf_{\sigma^2\gamma}(g)   \rmd g\\
      &\qquad \qquad \qquad\qquad\qquad-\int_{-\infty}^{(\tau_{\gamma}(w)   +\gamma \InftyBound)/2} \rme^{a\parenthese{\tau_{\gamma}(w)   +\gamma \InftyBound-2  g}} \varphibf_{\sigma^2\gamma}\parenthese{ \tau_{\gamma}(w) + \gamma \InftyBound- g }    \rmd g\\
      &=\int_{-\infty}^{(\tau_{\gamma}(w)   +\gamma \InftyBound)/2} \rme^{a\parenthese{\tau_{\gamma}(w)   +\gamma \InftyBound-2  g}} \varphibf_{\sigma^2\gamma}(g)   \rmd g\\
      &\qquad \qquad \qquad\qquad\qquad-\int_{-\infty}^{-(\tau_{\gamma}(w)   +\gamma \InftyBound)/2} \rme^{a\parenthese{-\tau_{\gamma}(w)   -\gamma \InftyBound-2  g}} \varphibf_{\sigma^2\gamma}\parenthese{  g }    \rmd g\\
      &=\int_{-(\tau_{\gamma}(w)   +\gamma \InftyBound)/2}^{(\tau_{\gamma}(w)   +\gamma \InftyBound)/2} \rme^{a\parenthese{\tau_{\gamma}(w)   +\gamma \InftyBound-2  g}} \varphibf_{\sigma^2\gamma}(g)   \rmd g\\
      &\qquad \qquad \qquad\qquad\qquad+\int_{-\infty}^{-(\tau_{\gamma}(w)   +\gamma \InftyBound)/2} 2\sinh\parenthese{a(\tau_{\gamma}(w)   +\gamma \InftyBound)}\rme^{-2  ag} \varphibf_{\sigma^2\gamma}\parenthese{  g }    \rmd g \eqsp .
      % &= \rme^{a\parenthese{\tau_{\gamma}(w)   +\gamma \InftyBound}+2a^2\sigma^2\gamma}\bigg(\Phibf\parenthese{\frac{\tau_{\gamma}(w) +\gamma \InftyBound}{2\sigma\gamma^{1/2}}+2\sigma\gamma^{1/2}a}\\
      % &\qquad\qquad\qquad\qquad\qquad\qquad\qquad\qquad\qquad\qquad -\Phibf\parenthese{-\frac{\tau_{\gamma}(w) +\gamma \InftyBound}{2\sigma\gamma^{1/2}}+2\sigma\gamma^{1/2}a}\bigg)\\
      % &\qquad \qquad \qquad\qquad+ 2\sinh\parenthese{a(\tau_{\gamma}(w)   +\gamma \InftyBound)}\rme^{2a^2\sigma^2\gamma}\Phibf\parenthese{-\frac{\tau_{\gamma}(w) +\gamma \InftyBound}{2\sigma\gamma^{1/2}}+2\sigma\gamma^{1/2}a} \eqsp .
    \end{align}
This concludes the proof.
  \end{proof}
  \begin{lemma}
    \label{lem:big_fix_dev_Phi}
    Let $a >0$. Then, for any $\gamma >0$, and $w\in \rset_+$
    \begin{multline}
      \Phibf\parenthese{\frac{\tau_{\gamma}(w) +\gamma \InftyBound}{2\sigma\gamma^{1/2}}+2\sigma\gamma^{1/2}a}-\Phibf\parenthese{-\frac{\tau_{\gamma}(w) +\gamma \InftyBound}{2\sigma\gamma^{1/2}}+2\sigma\gamma^{1/2}a}\\
      \qquad\qquad \leq 1-2\Phibf\parenthese{-\frac{\tau_{\gamma}(w) +\gamma \InftyBound}{2\sigma\gamma^{1/2}}}+2\sigma^2\gamma a^2  \eqsp.
    \end{multline}
  \end{lemma}
  \begin{proof}
    We consider the decomposition
    \begin{multline}
      \Phibf\parenthese{\frac{\tau_{\gamma}(w) +\gamma \InftyBound}{2\sigma\gamma^{1/2}}+2\sigma\gamma^{1/2}a}-\Phibf\parenthese{-\frac{\tau_{\gamma}(w) +\gamma \InftyBound}{2\sigma\gamma^{1/2}}+2\sigma\gamma^{1/2}a} \\
 = 1-2\Phibf\parenthese{-\frac{\tau_{\gamma}(w) +\gamma \InftyBound}{2\sigma\gamma^{1/2}}}+ A_1 + A_2 \eqsp,
        \end{multline}
        where
        \begin{align}
          \label{eq:40}
          A_1 &  =       \Phibf\parenthese{\frac{\tau_{\gamma}(w) +\gamma \InftyBound}{2\sigma\gamma^{1/2}}+2\sigma\gamma^{1/2}a} -       \Phibf\parenthese{\frac{\tau_{\gamma}(w) +\gamma \InftyBound}{2\sigma\gamma^{1/2}}}\\
          A_2 & =  \Phibf\parenthese{-\frac{\tau_{\gamma}(w) +\gamma \InftyBound}{2\sigma\gamma^{1/2}}} - \Phibf\parenthese{-\frac{\tau_{\gamma}(w) +\gamma \InftyBound}{2\sigma\gamma^{1/2}}+2\sigma\gamma^{1/2}a}  \eqsp.
        \end{align}
   By Taylor's theorem, we have for any $\gamma >0$, and $w\in \rset_+$
    \begin{equation}
      A_1 \leq 2\sigma\gamma^{1/2}a(2\uppi)^{\half}\rme^{-2^{-1}\parenthese{\frac{\tau_{\gamma}(w) +\gamma \InftyBound}{2\sigma\gamma^{1/2}}}^2} \eqsp .
    \end{equation}
    Similarly,
    \begin{align}
      &
        \abs{A_2+2\sigma\gamma^{1/2}a(2\uppi)^{\half}\rme^{-2^{-1}\parenthese{\frac{\tau_{\gamma}(w) +\gamma \InftyBound}{2\sigma\gamma^{1/2}}}^2}}\\
       &\qquad\qquad\qquad\qquad\leq \abs{\int_{-\frac{\tau_{\gamma}(w) +\gamma \InftyBound}{2\sigma\gamma^{1/2}}}^{-\frac{\tau_{\gamma}(w) +\gamma \InftyBound}{2\sigma\gamma^{1/2}}+2\sigma\gamma^{1/2}a}-t\varphibf(t)\parenthese{-\frac{\tau_{\gamma}(w) +\gamma \InftyBound}{2\sigma\gamma^{1/2}}+2\sigma\gamma^{1/2}a-t}\rmd t}\\
      % &\phantom{\Phibf\parenthese{-\frac{\tau_{\gamma}(w) +\gamma \InftyBound}{2\sigma\gamma^{1/2}}+2\sigma\gamma^{1/2}a}}\leq \Phibf\parenthese{-\frac{\tau_{\gamma}(w) +\gamma \InftyBound}{2\sigma\gamma^{1/2}}}+2\sigma\gamma^{1/2}a(2\uppi)^{\half}\rme^{-2^{-1}\parenthese{\frac{\tau_{\gamma}(w) +\gamma \InftyBound}{2\sigma\gamma^{1/2}}}^2}\\
      & \qquad\qquad\qquad\qquad \leq \sup_{t \in \rset}\abs{t\varphibf(t)}\int_{0}^{2\sigma\gamma^{1/2}a} t \rmd t \eqsp .
    \end{align}
    The proof is complete since $\sup_{t \in \rset}\abs{t\varphibf(t)}\leq 1$.
  \end{proof}

  \begin{lemma} 
    \label{lem:big_fix_last_term}
    Let $a >0$. Then, for any $\gamma \in \ocint{0, 1/(4\sigma^2)}$, and $w\in \rset_+$
    \begin{align}
      &-2\Phibf\parenthese{-\frac{\tau_{\gamma}(w) +\gamma \InftyBound}{2\sigma\gamma^{1/2}}}\parenthese{\rme^{a(\tau_\gamma(w)+\gamma\InftyBound)+2a^2\sigma^2\gamma}-1}\\
      &\qquad\qquad\qquad\qquad+2\rme^{2a^2\sigma^2\gamma}\sinh\parenthese{a(\tau_{\gamma}(w) +\gamma \InftyBound)}\Phibf\parenthese{-\frac{\tau_{\gamma}(w) +\gamma \InftyBound}{2\sigma\gamma^{1/2}}+2\sigma\gamma^{1/2}a}\\
      &\leq 4\rme^{2a^2\sigma^2\gamma}\sigma^2 a \gamma (2\uppi)^{-1/2}  \rme^{(a+2\sigma\gamma^{\half}a)^2/2} \eqsp .
    \end{align}
  \end{lemma}
\begin{proof}
  Using that $\sinh(t)-\rme^{t}+1\leq 0$ for $t\geq 0$, we get
  \begin{align}
    &-2\Phibf\parenthese{-\frac{\tau_{\gamma}(w) +\gamma \InftyBound}{2\sigma\gamma^{1/2}}}\parenthese{\rme^{a(\tau_\gamma(w)+\gamma\InftyBound)+2a^2\sigma^2\gamma}-1}\\
      &\qquad+2\rme^{2a^2\sigma^2\gamma}\sinh\parenthese{a(\tau_{\gamma}(w) +\gamma \InftyBound)}\Phibf\parenthese{-\frac{\tau_{\gamma}(w) +\gamma \InftyBound}{2\sigma\gamma^{1/2}}+2\sigma\gamma^{1/2}a}\\
      &\leq 2\Phibf\parenthese{-\frac{\tau_{\gamma}(w) +\gamma \InftyBound}{2\sigma\gamma^{1/2}}}\rme^{2a^2\sigma^2\gamma}\parentheseDeux{\sinh\parenthese{a(\tau_{\gamma}(w) +\gamma \InftyBound)}-\rme^{a(\tau_\gamma(w)+\gamma\InftyBound)}+1} \\
    &\qquad +2\rme^{2a^2\sigma^2\gamma}\sinh\parenthese{a(\tau_{\gamma}(w) +\gamma \InftyBound)}A_1 \\
    & \qquad \leq 2\rme^{2a^2\sigma^2\gamma}\sinh\parenthese{a(\tau_{\gamma}(w) +\gamma \InftyBound)} A_1  \eqsp,\\
      & \text{ with } A_1 = \parentheseDeux{\Phibf\parenthese{-\frac{\tau_{\gamma}(w) +\gamma \InftyBound}{2\sigma\gamma^{1/2}}+2\sigma\gamma^{1/2}a}-\Phibf\parenthese{-\frac{\tau_{\gamma}(w) +\gamma \InftyBound}{2\sigma\gamma^{1/2}}}}\eqsp.
  \end{align}
  In addition for any $t\in \rset_+$, since $(\rme^{u}-1)/u \leq \rme^u$ for any $u\in \rset_+^*$, we have 
  \begin{align}
    \Phibf(-t+2\sigma\gamma^{1/2}a)-\Phibf(-t)=\int_{-t}^{-t+2\sigma\gamma^{1/2}a} \varphibf(x) \rmd x &\leq(2\uppi)^{-1/2}\rme^{-t^2/2}\int_{0}^{2\sigma\gamma^{1/2}a} \rme^{xt} \rmd x\\
    &\leq 2\gamma^{1/2}\sigma a(2\uppi)^{-1/2}\rme^{-t^2/2+2\sigma\gamma^{1/2}at} \eqsp .
  \end{align}
Note that $(2\sigma\gamma^{1/2})^{-1}\sinh\parenthese{a(\tau_{\gamma}(w) +\gamma \InftyBound)}\leq \sinh\parenthese{a(\tau_{\gamma}(w) +\gamma \InftyBound)(2\sigma\gamma^{1/2})^{-1}}$ using the condition $(2\sigma\gamma^{\half})^{-1}\geq 1$. It yields
\begin{align}
  &\sinh\parenthese{a(\tau_{\gamma}(w) +\gamma \InftyBound)}
   A_1 \leq 4\sigma^2 a \gamma (2\uppi)^{-1/2}  \sup_{t\in \rset_+}\parenthese{\sinh(at)\rme^{-t^2/2+2\sigma\gamma^{1/2}at}} \eqsp .
\end{align}
The proof is complete since for any $t\in \rset_+$, $\sup_{t\in \rset_+}\parentheseLigne{\sinh(at)\rme^{-t^2/2+2\sigma\gamma^{1/2}at}}\leq 2^{-1}\rme^{(a+2\sigma\gamma^{\half}a)^2/2}$.
\end{proof}

\begin{lemma}
  \label{lem:fix_drift_exp_0}
  Assume \Cref{ass:LipsAnd}-\ref{ass:LipsAnd_2}. Let $a >0$. Then, for any $\gamma \in \ocint{0,\bgamma}$,
  \begin{equation}
    Q_{\gamma} \VlyapDsexp_{a}(0)\leq \gamma \InftyBound A_a \eqsp ,
  \end{equation}
  where $A_a=4\rme^{a \bgamma (\InftyBound+2a\sigma^2)}\left.\parentheseDeux{\bgamma(2a\sigma^2+\InftyBound/2)^2+\sigma^2}\middle/ (2\sigma^2)\right.$.
\end{lemma}
\begin{proof}
  By definition \eqref{def_q_gamma}, we have
    \begin{align}
      Q_{\gamma} \VlyapDsexp_{a}(0) &= \int_{\rset} (\exp(a \{\gamma \InftyBound- 2 \sigma \gamma^{\half}g\})-1)\{1- \bpg\parenthese{\gamma \InftyBound, g} \} \varphibf(g) \rmd g \\
                            &= \int_{-\infty}^{\gamma \InftyBound/2}    (\exp(a \{\gamma \InftyBound- 2 g\})-1) \{ \varphibf_{\sigma^2\gamma}(g) - \varphibf_{\sigma^2\gamma}( \gamma \InftyBound - g)\} \eqsp. \\
    \end{align}
    In addition, by using that $1-\rme^{-t}\leq t$ for any $t\in \rset_+$, we have for any $g\in \ccint{-\infty,\gamma \InftyBound/2}$,
    \begin{align}
      \varphibf_{\sigma^2\gamma}(g) - \varphibf_{\sigma^2\gamma}( \gamma \InftyBound - g)
      &=\varphibf_{\sigma^2\gamma}(g)\parentheseDeux{1-\rme^{-\InftyBound(\gamma\InftyBound-2g)/(2\sigma^2)}}\\
      &\leq \varphibf_{\sigma^2\gamma}(g)\parentheseDeux{\InftyBound(\gamma\InftyBound-2g)/(2\sigma^2)} \eqsp .
    \end{align}
    % In the same way as proof of \Cref{lem:drift_exp_sticky_v0}, 
    % \begin{equation}
    %   \int_{-\infty}^{\gamma \InftyBound/2}    \exp(a \{\gamma \InftyBound- 2 g\})  \varphibf_{\sigma^2\gamma}(g)\rmd g \leq \rme^{a  \gamma \InftyBound +2 a^2 \sigma^2 \gamma} \eqsp .
    % \end{equation}
    Finally, using 
    \begin{align}
      &\int_{-\infty}^{\gamma \InftyBound/2}    (\gamma\InftyBound-2g)(\exp(a \{\gamma \InftyBound- 2 g\})-1)  \varphibf_{\sigma^2\gamma}(g)\rmd g \\
      &\qquad\qquad\qquad\qquad\leq \int_{-\infty}^{\gamma \InftyBound/2}    (\gamma\InftyBound-2g)^2\exp(a \{\gamma \InftyBound- 2 g\})  \varphibf_{\sigma^2\gamma}(g)\rmd g\\
      &\qquad\qquad\qquad\qquad= 4\int_{-\infty}^{0}    g^2\exp(-2ag )  \varphibf_{\sigma^2\gamma}(g+\gamma\InftyBound/2)\rmd g\\
      &\qquad\qquad\qquad\qquad= 4(\sqrt{2\uppi}\sigma\gamma^{\half})^{-1}\rme^{a \gamma \InftyBound+2a^2\sigma^2\gamma}\int_{-\infty}^{0}    g^2\rme^{-(g+\gamma\InftyBound/2+2a\sigma^2\gamma)^2/(2\sigma^2\gamma)}  \rmd g \\
      &\qquad\qquad\qquad\qquad \leq 4\rme^{a \gamma \InftyBound+2a^2\sigma^2\gamma}\parentheseDeux{\gamma^2(2a\sigma^2+\InftyBound/2)^2+\sigma^2\gamma} \eqsp ,
    \end{align}
    completes the proof.
\end{proof}

%%% Local Variables:
%%% mode: latex
%%% TeX-master: "main_imsart"
%%% End:

\end{document}